\documentclass[final, 3p, times]{elsarticle}
\biboptions{sort&compress}
\usepackage{graphicx}
\usepackage{amstext, amsmath, amssymb, amsthm}
\usepackage[shadow]{todonotes}
\usepackage{booktabs}
\usepackage{url}
\usepackage{import}
\usepackage{amsfonts}
\usepackage{amsbsy}
\usepackage{fancyvrb}
\usepackage{stmaryrd}
\usepackage{pdfsync}
\usepackage{graphicx}
\usepackage{algorithm}
\usepackage[shadow]{todonotes}
\usepackage[numbers]{natbib}
\usepackage{enumerate}
\usepackage[utf8]{inputenc}
\usepackage[T1]{fontenc}
\usepackage[style=english]{csquotes}
\usepackage{textcomp}	
\usepackage{enumerate}
\usepackage[english]{babel}
\usepackage{amsmath, amssymb, amstext, amsfonts, mathrsfs, mathtools}
\usepackage{verbatim}
\usepackage{xcolor}
\usepackage{graphicx}
\usepackage{epstopdf}
\usepackage[font=footnotesize]{subfig}

\usepackage[plainpages=false]{hyperref}
\hypersetup{
		bookmarksopen=true,
		bookmarksnumbered=true,
		pdfborder={0 0 0},
		pdftitle={A stabilized Nitsche cut finite element method for the Oseen problem},
		pdfauthor={A. Massing, B. Schott, W.A. Wall},
		pdfcreator={A. Massing, B. Schott, W.A. Wall},
		colorlinks=true,
		linkcolor=blue,
		urlcolor=blue	
	}

\newcommand{\Figref}[1]{Fig.\,\ref{#1}}

\newtheorem{theorem}{Theorem}

\newtheorem{lemma}[theorem]{Lemma}
\newtheorem{corollary}[theorem]{Corollary}
\newdefinition{remark}[theorem]{Remark}

\renewenvironment{proof}{\noindent \newline \textbf{Proof.}}{\hfill \mbox{\fbox{} } \newline}

\numberwithin{equation}{section}
\numberwithin{figure}{section}
\numberwithin{table}{section}
\numberwithin{theorem}{section}

\newcommand{\RR}{\mathbb{R}}
\newcommand{\NN}{\mathbb{N}}
\newcommand{\foralls}{\forall\,}
\newcommand{\dx}{\,\mathrm{d}x}

\DeclareMathOperator{\diam}{diam}

\DeclareMathOperator{\RE}{Re}

\DeclareMathOperator{\CFL}{CFL}

\renewcommand{\div}{\nabla \cdot}
\newcommand{\grad}{\nabla}

\newcommand{\restr}[2]{ \left. #1 \right|_{#2}}
\newcommand{\tnorm}[1]{{\vert\kern-0.25ex\vert\kern-0.25ex\vert #1 
    \vert\kern-0.25ex\vert\kern-0.25ex\vert}} 
\newcommand{\jump}[1]{\llbracket #1 \rrbracket}
\newcommand{\norm}[1]{\| #1 \|}

\newcommand{\normLtwo}[2]{\| #1 \|_{#2}}

\newcommand{\onehalf}{\frac{1}{2}}

\newcommand{\apriori}{\emph{a~priori }}

\DefineVerbatimEnvironment{code}{Verbatim}{frame=single,rulecolor=\color{blue}}

\newcommand{\bfepsilon}{{\pmb\epsilon}}
\newcommand{\bfeps}{{\pmb\epsilon}}
\newcommand{\bfu}{\boldsymbol{u}}
\newcommand{\bff}{\boldsymbol{f}}
\newcommand{\bfg}{\boldsymbol{g}}

\newcommand{\bfv}{\boldsymbol{v}}

\newcommand{\bfn}{\boldsymbol{n}}
\newcommand{\bft}{\boldsymbol{t}}
\newcommand{\bfx}{\boldsymbol{x}}
\newcommand{\bfP}{\boldsymbol{P}}
\newcommand{\bfzero}{\boldsymbol{0}}
\newcommand{\bfbeta}{\boldsymbol{\beta}}
\newcommand{\bfI}{\boldsymbol{I}}

\newcommand{\GammaIn}{\Gamma_{\mathrm{in}}}

\newcommand{\Oast}{\Omega^{\ast}}
\newcommand{\Oasth}{\Omega^{\ast}_h}

\newcommand{\pO}{\Gamma}

\newcommand{\Fast}{\mathcal{F}_{\Gamma}}
\newcommand{\piast}{\pi^{\ast}_h}
\newcommand{\bfpi}{\boldsymbol{\pi}_h}
\newcommand{\bfpiast}{\boldsymbol{\pi}^{\ast}_h}
\newcommand{\Piast}{\Pi^{\ast}_h}
\newcommand{\tnast}{\tn_{\ast}}
\newcommand{\mcF}{\mathcal{F}}
\newcommand{\mcT}{\mathcal{T}}
\newcommand{\mcV}{\mathcal{V}}
\newcommand{\mcW}{\mathcal{W}}
\newcommand{\mcX}{\mathcal{X}}
\newcommand{\mcQ}{\mathcal{Q}}

\newcommand{\mcO}{\mathcal{O}}

\newcommand{\mcP}{\mathcal{P}}

\newcommand{\tn}{|\mspace{-1mu}|\mspace{-1mu}|}

\newcommand{\bfw}{\boldsymbol{w}}

\newcommand{\nablan}{\partial_{\bfn}}

\newcommand{\mean}[1]{\langle #1 \rangle}

\newcommand{\nuhalf}{\nu^\frac{1}{2}}

\newcommand{\sigmahalf}{\sigma^{\frac{1}{2}}}

\journal{\phantom{journal}}

\begin{document}

\begin{frontmatter}

\title{A stabilized Nitsche cut finite element method for the Oseen problem}

\author[umit]{A.~Massing}
\ead{andre.massing@umu.se}
\ead[url]{www.andremassing.com}

\author[lnm]{B.~Schott\corref{cor1}}
\ead{schott@lnm.mw.tum.de}
\ead[url]{www.lnm.mw.tum.de/staff/benedikt-schott}
\cortext[cor1]{Corresponding author}

\author[lnm]{W.A.~Wall}
\ead{wall@lnm.mw.tum.de}
\ead[url]{www.lnm.mw.tum.de/staff/wall}

\address[umit]{Department of Mathematics and Mathematical Statistics, Ume\aa{} University, SE-901 87 Ume\aa{}, Sweden}
\address[lnm]{Institute for Computational Mechanics, Technical University of Munich, Boltzmannstraße 15, 85747 Garching, Germany}

\begin{abstract}
We propose a stabilized Nitsche-based cut finite element formulation for the
Oseen problem in which the boundary of the domain is allowed to cut through
the elements of an easy-to-generate background mesh.  Our formulation is
based on the continuous interior penalty (CIP) method of
\citet{BurmanFernandezHansbo2006} which penalizes jumps of velocity and
pressure gradients over inter-element faces to counteract instabilities
arising for high local Reynolds numbers and the use of equal order
interpolation spaces for the velocity and pressure.  Since the mesh does not
fit the boundary, Dirichlet boundary conditions are imposed weakly by a
stabilized Nitsche-type approach.  The addition of CIP-like ghost-penalties
in the boundary zone allows to prove that our method is inf-sup stable and to
derive optimal order \apriori~error estimates in an energy-type norm,
irrespective of how the boundary cuts the underlying mesh.  All applied
stabilization techniques are developed with particular emphasis on low and
high Reynolds numbers.  Two- and three-dimensional numerical examples
corroborate the theoretical findings. Finally, the proposed method is applied
to solve the transient incompressible Navier-Stokes equations on a complex
geometry.
\end{abstract}


\end{frontmatter}

\section{Introduction}
\label{sec:introduction}
Many important phenomena in science and engineering are modeled by a system of
partial differential equations (PDEs) posed on complicated, three-dimensional 
domains.
The numerical solution of PDEs
based on the finite element method
requires the 
generation of high quality meshes
to ensure both a proper geometric resolution
of the domain features and good approximation properties of the numerical scheme.
But even today, the generation of such meshes 
can be a time-consuming and challenging task
that can easily account for large portions of the
time and human resources in the overall simulation work flow. 
For instance, the simulation of many industrial application problems
requires a series of highly non-trivial preprocessing steps to transform 
CAD data into conforming domain discretizations which respect
complicated features of the geometric model.
The problem is even more pronounced if the geometry
of the model domain changes substantially in the course of the
simulation, e.g., in the simulation of multiphase flows, where the interface between
different fluid phases can undergo large and even topological changes
when bubbles merge or break up or drops pinch off. 
Then even modern Arbitrary-Lagrangian-Eulerian based mesh moving algorithms 
break down and a costly remeshing is the only resort.

A potential remedy to these challenges are flexible, so-called
unfitted finite element schemes which allow to embed complex or
changing domain parts freely into a static and easy-to-generate
computational domain.
For instance, to cope with large interface motion in
incompressible two-phase flows,
the discretization schemes in
\cite{ChessaBelytschko2003,GrossReusken2007,HansboLarsonZahedi2014,SchottRasthoferGravemeierWall2015}
combined an implicit, level set based
description of the fluid phase interface with an extended finite
element approach.
For fluid-structure interaction problems where
the structure might undergo large deformations,
numerical methods which combine 
fixed-grid Eulerian approach for the
fluid with a Lagrangian description for the structural body have been
developed in, e.g., \cite{Gerstenberger2008, Legay2006, CourtFournieLozinski2014, CourtFournie2015},
including three-dimensional fluid-structure-contact interactions~\cite{Mayer2010}.  
Furthermore, in several numerical schemes aiming at complex
fluid-structure interactions the idea of using finite element method on composite grids,
originally proposed in \cite{Hansbo2003}, has been picked
up.  In \cite{Shahmiri2011,MassingLarsonLoggEtAl2015,SchottShahmiriKruseWall2015}, an
additional fluid patch covering the boundary layer region around
structural bodies is embedded into a fixed-grid background fluid mesh
to possibly capture high velocity gradients near the structure when simulating
high-Reynolds-number incompressible flows.

As the computational mesh is not fitted to the domain boundary,
a common theme of the aforementioned unfitted finite element approaches
is the weak imposition of boundary or interface conditions posed on parts of the
embedded domain by means of Lagrange multipliers or Nitsche-type methods, see, e.g.,~\cite{Gerstenberger2009, CourtFournieLozinski2014,Shahmiri2011,BurmanHansbo2010,BurmanHansbo2012}.
Building upon and extending these ideas,
the cut finite element method (CutFEM) as a
particular unfitted finite element framework has gained rapidly
increasing attention in science and engineering,
see~\cite{BurmanClausHansboEtAl2014} for a recent overview.
A distinctive feature of the CutFEM approach is that it provides a
general, theoretically founded stabilization framework which, roughly speaking,
transfers stability and approximation properties from a finite element
scheme posed on a standard mesh to its cut finite element counterpart.
As a result, a wide range of problem classes, 
ranging from two-phase and fluid-structure
interaction
problems~\cite{SchottRasthoferGravemeierWall2015,MassingLarsonLoggEtAl2015}
to surface and surface-bulk 
PDEs~\cite{BurmanHansboLarson2015,BurmanHansboLarsonEtAl2016a,BurmanHansboLarsonEtAl2016c,HansboLarsonZahedi2016}, 
and embedding methods such as overlapping 
meshes~
\cite{Hansbo2003,Massing2012,MassingLarsonLogg2013,MassingLarsonLoggEtAl2013,SchottShahmiriKruseWall2015}
or
implicitly defined surfaces~\cite{BurmanHansboLarson2015,BurmanClausHansboEtAl2014},
has been treated by CutFEM based discretization schemes
in a transparent and unified way.
However, for fluid related problems,
stability and a priori error analysis of CutFEM type approaches 
has only been performed for simplified prototype problems,
such as the Poisson problem~\cite{BurmanHansbo2010,BurmanHansbo2012}, the Stokes problem~\cite{Burman2014b,MassingLarsonLoggEtAl2013,HansboLarsonZahedi2014,BurmanClausMassing2015},
and recently, for a low Reynolds-number
fluid-structure interaction problem governed by Stokes' equations in \cite{Burman2014a}.
For more complex fluid problems governed by the transient incompressible
Navier-Stokes equations at high Reynolds numbers, there is a lack of
numerical analysis.

In the present work we propose and analyze a cut finite element
method for the Oseen model problem. 
The Oseen problem comprises a set of linear equations which naturally arise in many
linearization and time-stepping methods for the transient, non-linear
incompressible Navier-Stokes equations.
Our formulation is based on the continuous interior penalty (CIP)
method of \citet{BurmanFernandezHansbo2006} which penalizes jumps of
velocity and pressure gradients over inter-element faces to counteract
instabilities arising for high local Reynolds numbers and the use of
equal order interpolation spaces for the velocity and pressure.
Since the mesh does not fit the boundary, Dirichlet boundary
conditions are imposed weakly by a stabilized Nitsche-type approach.
In contrast to \cite{BurmanFernandezHansbo2006},
additional measures are necessary 
to prove that the proposed scheme is inf-sup stable 
and satisfies optimal order \apriori~error estimates
irrespective of how the boundary cuts the underlying mesh.
Extending the approach taken in
\cite{Burman2010,BurmanHansbo2012,Burman2014b,MassingLarsonLoggEtAl2013} 
to provide geometrically robust \apriori~error and condition number estimates
for the Poisson and Stokes problem,
our method uses different face-based ghost-penalty
stabilizations for the velocity and pressure fields which are defined 
in the vicinity of the embedded boundary.
In \cite{SchottWall2014}, it was shown how these interface-zone
stabilization techniques can be naturally combined with the continuous interior penalty
method from~\cite{BurmanFernandezHansbo2006} to solve transient convection-dominant
incompressible Navier-Stokes equations on cut meshes.  
However, so far a numerical analysis of that method was still
outstanding.
The present paper now provides the theoretical corroboration of the
fluid formulation introduced in \cite{SchottWall2014} with focus
on the different ghost-penalty stabilizations in
the high-Reynolds-number regime.

A main challenge in the presented numerical analysis is to compensate
the lack of a suitable CutFEM variant of the $L^2$ projection operator
which in the fitted mesh case is an instrumental tool in the
theoretical analysis of CIP stabilized methods.
As a remedy, we derive stability and approximation results for norms
which are more natural in residual-based stabilization methods for the
Oseen problem. This approach allows to employ alternative
interpolation operators such as the Cl\'ement operator for which
proper CutFEM counterparts can be defined.
By adding suitable ghost-penalty stabilizations
we gain sufficient control over the advective derivative and the
incompressibility constraint with respect to the entire active part of
the computational mesh.
Consequently, we are able to prove that our scheme obeys an inf-sup
condition in a ghost-penalty enhanced energy-type norm and thus
satisfies the corresponding \apriori~energy norm error estimate.
All estimates are optimal independent of the positioning of the
boundary within the non-boundary fitted background mesh.
For the first time, we present a numerical analysis for
ghost-penalty operators scaled with non-constant coefficients
accounting for different flow regimes, covering the
treatment of instabilities arising from the convective term
and from the incompressibility constraint on cut meshes.
As a by-product of our numerical analysis, we
show how the continuous interior penalty stabilization terms
give control over slightly stronger norm contributions as they typically arise
in residual-based stabilization method~\cite{LubeRapin2006,BraackBurmanJohnEtAl2007}.

The paper is organized as follows: 
We conclude this section by summarizing our basic notation.
Then the Oseen model problem is briefly reviewed in Section~\ref{sec:oseen-problem}.
In Section~\ref{sec:cut-finite-element}, we 
formulate the 
stabilized Nitsche-type cut finite element method for the Oseen problem,
starting with the introduction of the proper cut finite element
spaces, followed by a review of the classical 
continuous interior penalty (CIP) method.
We explain how to extend the CIP scheme to the case of unfitted meshes,
discuss the need for additional ghost-penalty stabilization
techniques in the vicinity of the boundary zone for low and high
Reynolds numbers, and conclude this section
by stating the main \apriori~estimate for our cut finite element method.
Next, two- and three-dimensional test cases 
in Section~\ref{sec:numexamples} confirm the main theoretical result.
The applicability of our method to solve transient
incompressible Navier-Stokes equations is demonstrated by means of a
challenging complex three-dimensional helical pipe flow.
Afterwards, we present the numerical analysis of our proposed formulation.
In Section~\ref{sec:interpolation-est}, basic approximation properties,
interpolation operators and norms
are introduced and the importance of the different ghost-penalty terms
is elaborated.
Sections~\ref{sec:stability-properties} and \ref{sec:apriori-analysis} are devoted to the stability and
\apriori~error analysis of the proposed method. Therein, main focus is
directed to inf-sup stability and optimality of the error estimates in all flow regimes. 
Summarizing comments and an outlook to potential application fields for our
numerical scheme in Section~\ref{sec:conclusions} conclude this work.

\subsection{Basic Notation}
\label{ssec:notation}
Throughout this work, $\Omega \subset \RR^d$, $d = 2,3$ denotes an
open and bounded domain with Lipschitz boundary 
$\Gamma = \partial \Omega$. For $U \in \{\Omega, \Gamma \}$ and $ 0
\leqslant m < \infty$, $1 \leqslant q \leqslant \infty$, let $W^{m,q}(U)$ be the
standard Sobolev spaces
consisting of those $\RR$-valued functions defined on $U$ which possess $L^q$-integrable 
weak derivatives up to order $m$. Their associated norms
are denoted by $\|\cdot \|_{m,q,U}$. 
As usual, we write $H^m(U) = W^{m,2}$
and $(\cdot,\cdot)_{m,U}$ and $\|\cdot\|_{m,U}$ for the associated inner product
and norm. If unmistakable, we
occasionally write $(\cdot,\cdot)_{U}$ and $\|\cdot \|_{U}$ for the
inner products and norms associated with $L^2(U)$, with $U$ being a
measurable subset of $\RR^d$. 
For $s > 1/2$, we use the notation $[H_{\bfg}^s(\Omega)]^d$
to denote the set of all $\RR^d$-valued functions in $[H^s(\Omega)]^d$ whose
boundary traces are equal to $\bfg$.  
Moreover, $H_0(\div;\Omega) \subset [L^2(U)]^d$ denotes the space of divergence-free
functions, and $L^2_0(\Omega)$ denotes the
function space consisting of functions $u \in L^2(\Omega)$ with
zero average.
Finally, any norm $\|\cdot\|_{\mcP_h}$ used in this work which
involves a collection of geometric entities $\mcP_h$ should be
understood as broken norm defined by $\|\cdot\|_{\mcP_h}^2 =
\sum_{P\in\mcP_h} \|\cdot\|_P^2$ whenever $\|\cdot\|_P$ is well-defined,
with a similar convention for scalar products $(\cdot,\cdot)_{\mcP_h}$.

\section{The Oseen Problem}
\label{sec:oseen-problem}

After applying a time discretization method
and a linearization step, 
many solution algorithms for the non-linear Navier-Stokes equations
can be reduced to solving a sequence of auxiliary problems of Oseen type
for the velocity field $\bfu$ and the pressure field $p$:
\begin{alignat}{2}
  \label{eq:oseen-problem-momentum}
  \sigma \bfu + \bfbeta\cdot\nabla\bfu - \nabla\cdot(2\mu\bfepsilon(\bfu)) + \nabla p 
&= \bff \quad &&\text{ in } \Omega,
  \\
  \div\bfu &= 0 \quad &&\text{ in } \Omega,
  \label{eq:oseen-problem-compressible}
  \\
  \bfu &= \bfg
  \quad &&\text{ on } \Gamma.
  \label{eq:oseen-problem-boundary}
\end{alignat}
Here, $\bfepsilon(\bfu) = 1/2(\nabla \bfu + (\nabla \bfu)^T)$
denotes the rate-of-deformation tensor,
$\bfbeta \in [W^{1,\infty}(\Omega)]^d \cap H_0(\div;\Omega)$
the given divergence-free advective velocity field,
$\bff\in [L^2(\Omega)]^d$ the body force 
and $\bfg\in [H^{1/2}(\Gamma)]^d$ the given boundary data.
The reaction coefficient $\sigma$ and the viscosity $\mu$ are assumed
to be positive real-valued constants.
The corresponding weak formulation of the Oseen
problem~\eqref{eq:oseen-problem-momentum}--\eqref{eq:oseen-problem-boundary} is to find the velocity and the pressure field
 $(\bfu,p) \in
\mcV_{\bfg} \times \mcQ = [H_{\bfg}^1(\Omega)]^d \times L^2_0(\Omega)$ 
such that
\begin{align}
  a(\bfu, \bfv) + b(p, \bfv) - b(q, \bfu) = l(\bfv) \quad \foralls
  (\bfv,q) \in \mcV_{\bfzero} \times \mcQ,
  \label{eq:oseen_weak}
\end{align}
where
\begin{align}
  \label{eq:a-form-def}
  a(\bfu, \bfv) &:= (\sigma \bfu, \bfv)_{\Omega} +
  (\bfbeta\cdot\nabla\bfu, \bfv)_{\Omega} + (2\mu\bfepsilon(\bfu),
  \bfepsilon(\bfv))_{\Omega},    \\
  \label{eq:b-form-def}
  b(p, \bfv) &:= - (p, \div\bfv)_{\Omega}, 
  \\
  l(\bfv) &:= (\bff, \bfv)_{\Omega}.
  \label{eq:l-form-def}
\end{align}
The well-posedness and solvability of the continuous
problem~\eqref{eq:oseen_weak} is well-known, see for instance the
textbook by Girault and Raviart~\cite{RaviartGirault1986}.

\section{A Cut Finite Element Method for the Oseen Problem}
\label{sec:cut-finite-element}

\subsection{Computational Meshes and Cut Finite Element Spaces}
\label{ssec:cutfem-spaces}
Let $\widehat{\mcT}_h = \{T\}$ be a quasi-uniform mesh consisting of
shape-regular simplices $T$ with mesh size parameter~$h$ which covers
the physical domain~$\Omega$.
For the background mesh $\widehat{\mcT}_h$ we define the
\emph{active} (background) mesh
\begin{align}
  \mcT_h := \{ T \in \widehat{\mcT}_h : T \cap \Omega \neq \emptyset \},
\end{align}
consisting of all elements in $\widehat{\mcT}_h$ which intersect
$\Omega$.
Denoting the union of all elements $T \in \mcT_h$ by $\Oasth$,
we call $\mcT_h$ a \emph{fitted mesh} if
$\overline{\Omega} = \overline{\Oast_h}$
and an \emph{unfitted mesh} if
$\overline{\Omega} \subsetneq \overline{\Oast_h}$. 
To each active mesh, we associate the subset of elements
that intersect the boundary $\pO$
\begin{equation}
  \label{eq:define-cutting-cell-mesh}
  \mcT_{\pO} := \{T \in \mcT_h: T \cap \pO \neq \emptyset \}.
\end{equation}
\begin{figure}[tb]
  \begin{center}
    \includegraphics[width=0.45\textwidth]{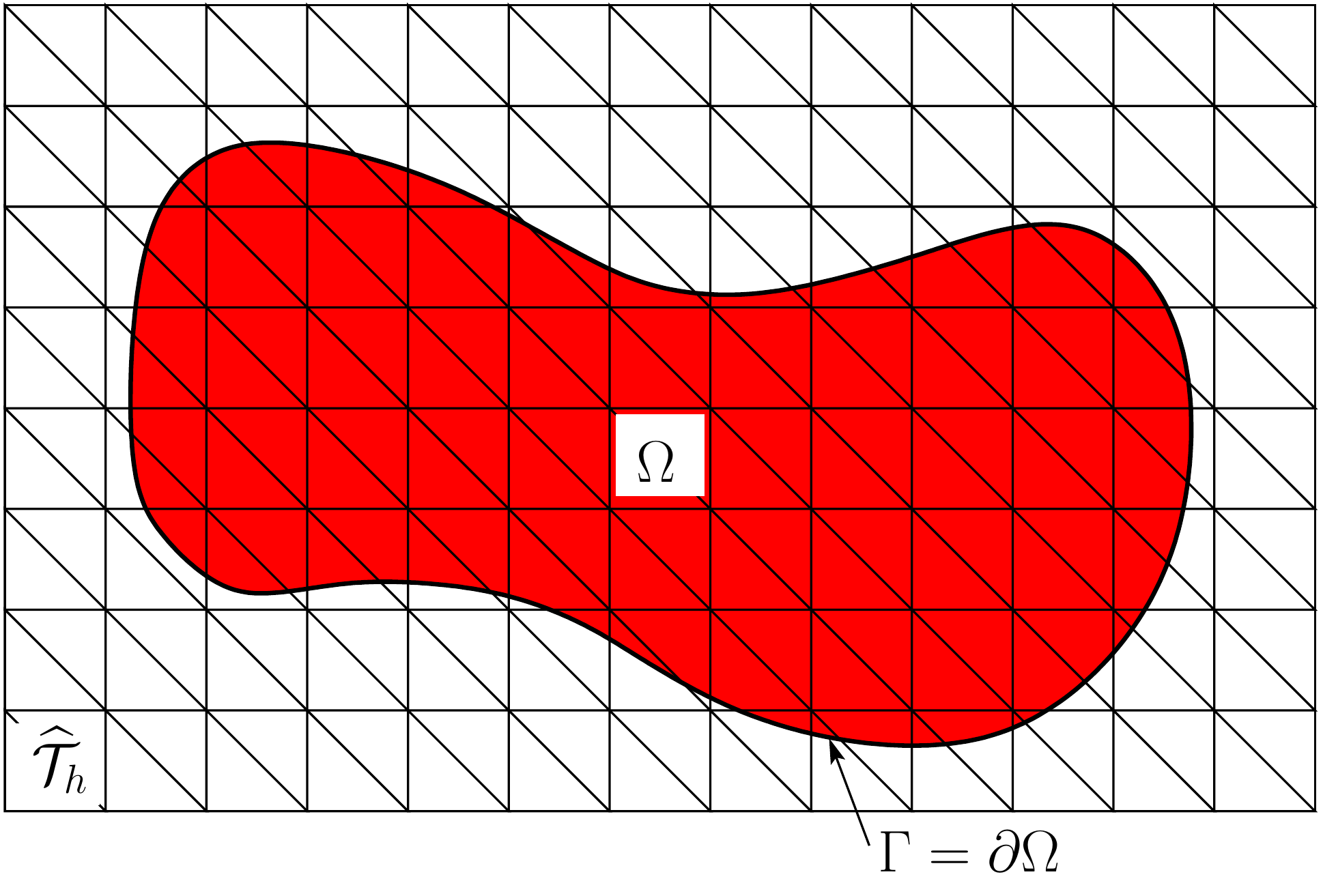} 
    \qquad
    \includegraphics[width=0.45\textwidth]{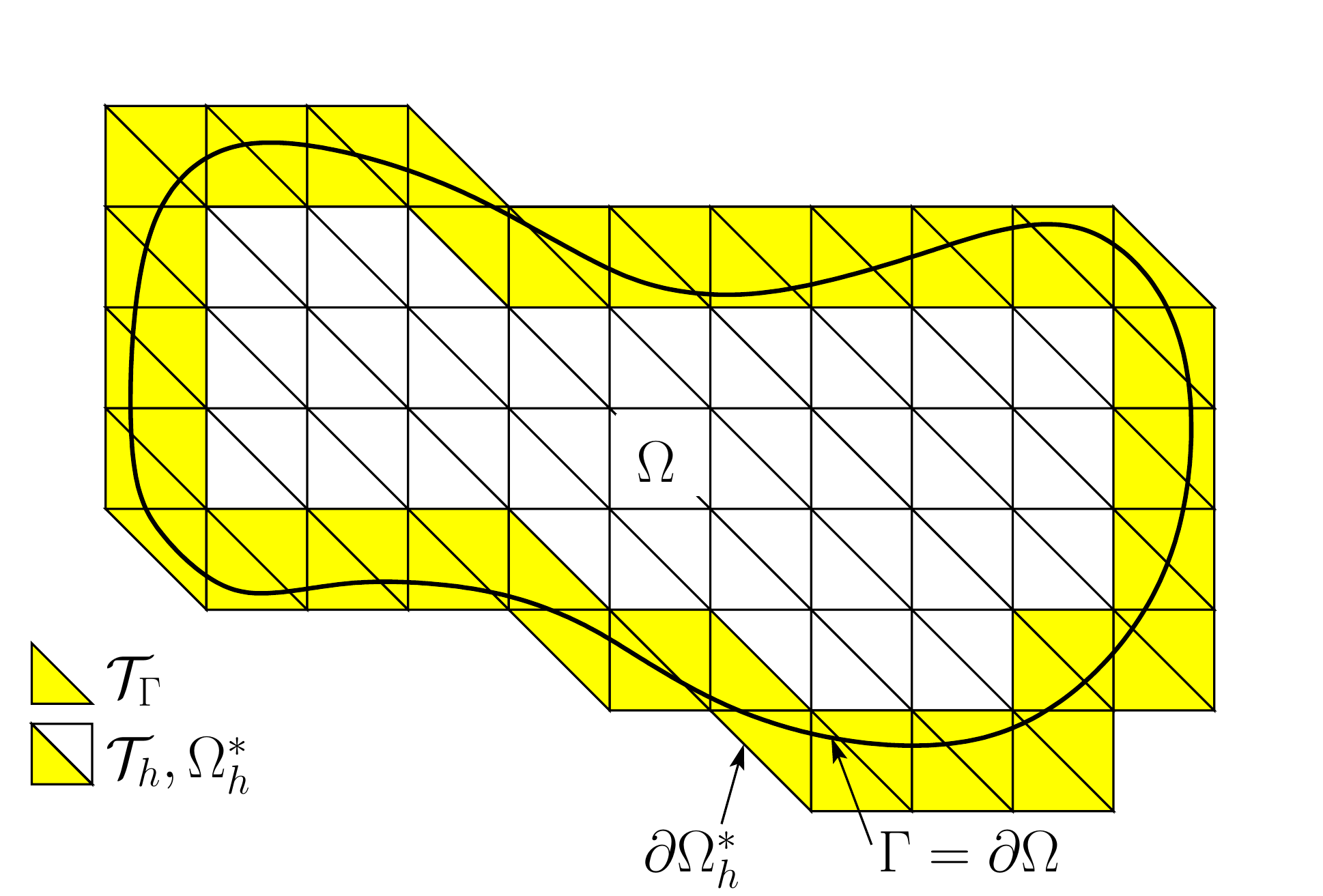}
  \end{center}
  \caption{Left: The physical domain $\Omega$ is defined as the
    inside of a given boundary $\pO$ embedded into a fixed
    background mesh $\widehat{\mcT}_h$. Right: The fictitious
    domain $\Oast_h$ is the union of the minimal subset $\mcT_h
    \subset \widehat{\mcT}_h$ covering $\Omega$.}
  \label{fig:computational-domain}
\end{figure}
The set of all facets,
i.e.~edges of elements in two dimensions and faces of elements in three
dimensions, are denoted by $\mcF_h$.
We let $\mcF_i$ be the set of all \emph{interior facets} $F$ which are
shared by exactly two elements, denoted by $T^+_F$ and $T^-_F$.
Further, we introduce the notation $\mcF_{\Gamma}$ for the set of all interior facets
belonging to elements intersected by the boundary $\pO$,
\begin{equation}
  \label{eq:ghost-penalty-facets}
  \mcF_{\Gamma} := \{ F \in \mcF_i :\;
  T^+_F \cap \pO \neq  \emptyset
  \vee
  T^-_F \cap \pO \neq  \emptyset
  \}.
\end{equation}
To ensure that $\Gamma$ is reasonably resolved by $\mcT_h$,
we require that the quasi-uniform $\mcT_h$ and the boundary $\pO$ satisfy the following
geometric conditions from \cite{HansboHansbo2002,BurmanHansbo2012,MassingLarsonLoggEtAl2014,BurmanClausMassing2015}:
\begin{itemize}
\item G1: The intersection between $\Gamma$ and a facet $F \in
  \mcF_i$ is simply connected; that is, $\Gamma$ does not
  cross an interior facet multiple times.
\item G2: For each element $T$ intersected by $\Gamma$, there exists a
  plane $S_T$ and a piecewise smooth parametrization $\Phi: S_T \cap T
  \rightarrow \Gamma \cap T$.
\item G3: We assume that there is an integer $N>0$ such that for each
  element $T \in \mcT_{\pO}$, there exists an element $T' \in
  \mcT_h \setminus \mcT_{\pO}$ and at most $N$ elements
  $\{T\}_{j=1}^N$ such that $T_1 = T,\,T_N = T'$ and $T_j \cap T_{j+1}
  \in \mcF_i,\; j = 1,\ldots, N-1$.  In other words, the
  number of facets to be crossed in order to ``walk'' from a cut
  element $T$ to a non-cut element $T' \subset \Omega$ is bounded.
\end{itemize}
\begin{figure}[tb]
  \begin{center}
    \includegraphics[width=0.40\textwidth]{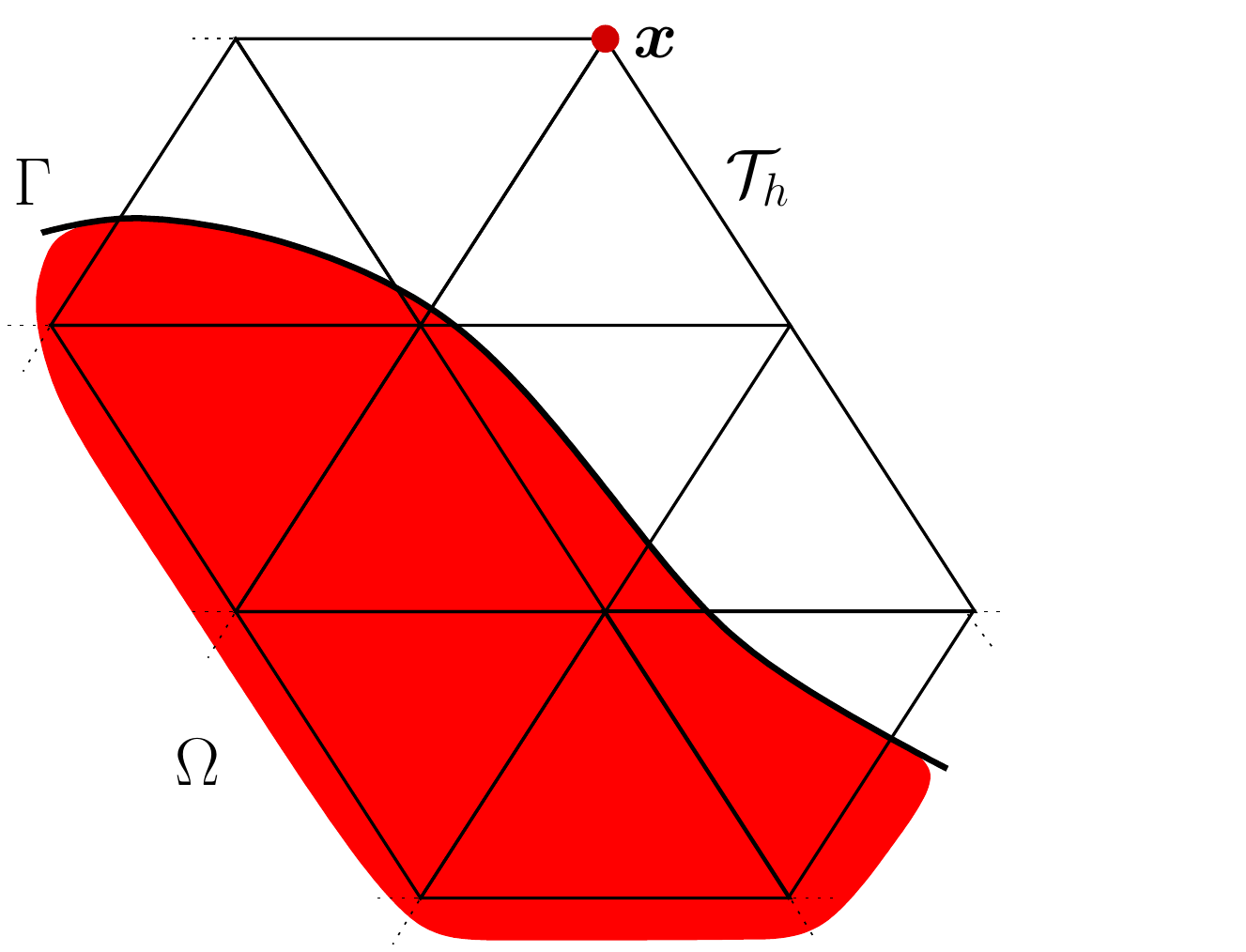} \qquad
    \includegraphics[width=0.40\textwidth]{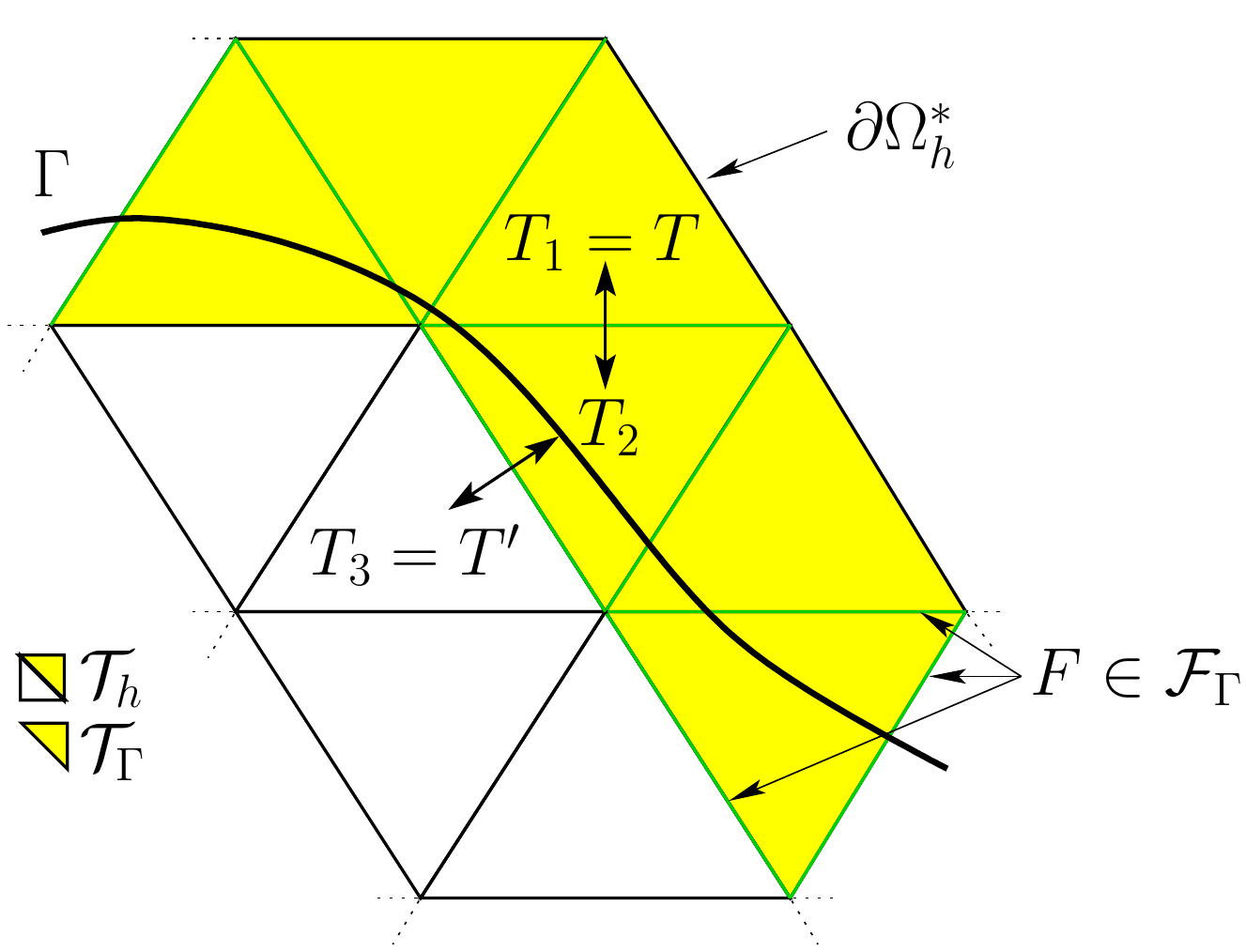}
    \caption{The boundary zone of the fictitious domain. Left: The
      background mesh $\mcT_h$ and the physical domain~$\Omega$. Observe that for the elements associated with the node
      $\bfx$, only a small fraction resides inside the domain
      $\Omega$. Right: Elements colored in yellow are intersected by the
      boundary and are therefore part of the mesh
      $\mcT_{\pO}$. Interior facets belonging to elements intersected by the
      boundary ($\Fast$) are marked in green. Arrows indicate the shortest ``walk'' from cut element~$T$ to an uncut element~$T'$.}
    \label{fig:boundary-zone}
  \end{center}
\end{figure}
Figures~\ref{fig:computational-domain} and \ref{fig:boundary-zone} summarize the notation.
Next, for a given mesh $\mcT_h$, we denote by $\mcX_h$
the finite element spaces consisting of continuous 
piecewise polynomials of order $k$
\begin{align}
  \mcX_h &= \left\{ v_h \in C^0(\Oast_h): \restr{v_h}{T} \in \mcP^k(T)\, \foralls T \in \mcT_h \right\}.
\end{align}
Using equal order interpolation spaces,
the discrete velocity space $\mcV_h$, the discrete pressure space $\mcQ_h$
and the total approximation space $\mcW_h$
are then defined by
\begin{align}
  \mcV_h = [\mcX_h]^d, \quad \mcQ_h = \mcX_h, \quad \mcW_h = \mcV_h
  \times \mcQ_h.
  \label{eq:approximation_spaces}
\end{align}

\subsection{A Short Review of the Continuous Interior Penalty Method for the Oseen Problem}
\label{ssec:stabilized-oseen}
Assuming for the moment that $\mcT_h$ is a fitted tessellation of $\Omega$,
it is well-known that a direct discretization of the weak
formulation~(\ref{eq:oseen_weak}) using equal-order interpolation
spaces $\mcW_h = \mcV_h \times \mcQ_h$
suffers from two problems.
First, the resulting scheme does not satisfy an inf-sup
condition and consequently, is not stable in the sense of
Babu\v{s}ka--Brezzi~\cite{BrezziFortin1991}.
Second, it leads to
spurious oscillations in the numerical
solution and sub-optimal error estimates
in the case of convection-dominant flow.
To counteract both effects, the discrete
form~\eqref{eq:oseen_weak} typically needs
to be stabilized, see~\cite{BraackBurmanJohnEtAl2007} for an overview
over various stabilization techniques for finite element based
discretizations of the Oseen problem.

In this work, we employ the continuous interior penalty (CIP) method proposed
by \citet{BurmanFernandezHansbo2006}, which is a
symmetric stabilization technique penalizing the jump of the velocity and pressure 
gradients over element facets.
More precisely, the stabilization operators are defined by
\begin{align}
  \label{eq:cip-s_beta}
  s_{\beta}(\bfu_h, \bfv_h)
  &:=
  \gamma_{\beta}
  \sum_{F\in\mathcal{F}_i}
  \mean{\phi_{\beta}}|_F h
  (\jump{\bfbeta \cdot\grad \bfu_h},\jump{\bfbeta \cdot\grad\bfv_h})_F,
  \\
  \label{eq:cip-s_u}
  s_{u}(\bfu_h, \bfv_h)
  &:=
  \gamma_u
  \sum_{F\in\mathcal{F}_i}
  \mean{\phi_{u}}|_F h
  (\jump{\div\bfu_h},\jump{\div\bfv_h})_F,
  \\
  \label{eq:cip-s_p}
  s_p(p_h, q_h) &:=
  \gamma_p 
  \sum_{F\in\mathcal{F}_i}
  \mean{\phi_{p}}|_F h
  ( \jump{\bfn_F \cdot \nabla p_h},\jump{\bfn_F \cdot \nabla q_h})_F,
\end{align}
where for any, possibly vector-valued, piecewise discontinuous function $\phi$
on $\mcT_h$,
the jump $\jump{\phi}$ and average $\mean{\phi}$
over an interior facet $F\in\mcF_i$ 
is given by
\begin{align}
  \jump{\phi}|_F 
= (\phi_{F}^+ - \phi_{F}^-),
  \qquad
  \mean{\phi}|_F 
= \tfrac{1}{2}(\phi_{F}^+ + \phi_{F}^-),
  \label{eq:jump-definition}
\end{align}
with $\phi^{\pm}(\bfx) = \lim_{t\to0^+} \phi(\bfx \pm t \bfn_F)$
for some chosen  normal unit vector $\bfn_F$ on $F$.
The element-wise constant stabilization parameters $\phi_{\beta}$, $\phi_u$ and $\phi_p$
are chosen as
\begin{align}
  \label{eq:cip-s_scalings}
  \phi_{u,T}
  = \mu +\|\bfbeta\|_{0,\infty,T} h + \sigma h^2,
\qquad
  \phi_{\beta,T} = \phi_{p,T} =  \dfrac{h^2}{\mu
+\|\bfbeta\|_{0,\infty,T} h + \sigma h^2 }.
\end{align}
To ease the notation we often simply write $\phi_{\beta,F}$,
$\phi_{u,F}$, and $\phi_{p,F}$ for their respective face averages.
Note that since $\bfbeta \in
[W^{1,\infty}(\Omega)]^d \subseteq [C^{0,1}(\Omega)]^d$
it holds that
$\bfbeta$ is (Lipschitz)-continuous by assumption and therefore
$\bfbeta\cdot\bfn_F$ is single valued on facets $F\in\mcF_i$.
Now the CIP augmented finite element formulation for 
the Oseen problem is to find $U_h=(\bfu_h, p_h) \in \mcW_h$ such that
for all $V_h=(\bfv_h, q_h) \in \mcW_h$
\begin{align}
 A_h(U_h,V_h) + S_h(U_h, V_h)
&= L_h(V_h),
 \label{eq:oseen-discrete-fitted-stabilized}
\end{align}
where
\begin{align}
A_h(U_h,V_h)
  &:= a_h(\bfu_h, \bfv_h)+b_h(p_h, \bfv_h)-b_h(q_h,\bfu_h),
 \label{eq:Ah-form-def}
\\
S_h(U_h,V_h) 
&:= 
s_{\beta}(\bfu_h, \bfv_h)
+ s_{u}(\bfu_h, \bfv_h)
+ s_{p}(p_h, q_h)
 \label{eq:Sh-form-def}
\end{align}
with
\begin{align}
  \label{eq:ah-form-def}
  a_h(\bfu_h, \bfv_h)
  &:= a(\bfu_h, \bfv_h)
  - ( (\bfbeta\cdot\bfn)\bfu_h, \bfv_h )_{\GammaIn}
  +
 (
  \gamma(\phi_u/h) \bfu_h\cdot\bfn,
  \bfv_h\cdot\bfn
  )_{\Gamma}  \nonumber\\
  &\quad\quad
   - ( 2\mu\bfepsilon(\bfu_h)\bfn, \bfv_h )_{\Gamma}
   - ( \bfu_h, 2\mu\bfepsilon(\bfv_h)\bfn )_{\Gamma}
  + (
  \gamma(\mu/h)\bfu_h,\bfv_h
  )_{\Gamma},
  \\
  \label{eq:bh-form-def}
  b_h(p_h, \bfv_h) &:= b(p_h, \bfv_h)+
  ( p_h,
  \bfv_h\cdot\bfn
  )_{\Gamma},
  \\
  L_h(V_h)
  &:=
  l(\bfv_h)
  - ( (\bfbeta\cdot\bfn)\bfg, \bfv_h )_{\GammaIn}
  +( \gamma (\phi_u/h) \bfg\cdot\bfn, \bfv_h\cdot\bfn )_\Gamma \nonumber \\
&\quad\quad
  - ( \bfg, 2\mu\bfepsilon(\bfv_h)\bfn )_{\Gamma}
  + ( \gamma(\mu/h)\bfg,\bfv_h )_{\Gamma}
  - ( \bfg \cdot \bfn, q_h )_{\Gamma}.
  \label{eq:Lh-form-def}
\end{align}

\begin{remark}
The boundary condition~\eqref{eq:oseen-problem-boundary} is
imposed weakly using Nitsche's method,
which was originally introduced in~\cite{Nitsche1971}
and then extended to the Oseen problem in, e.g., \cite{BurmanFernandezHansbo2006,Bazilevs2007}.
This technique results in additional boundary terms in~\eqref{eq:ah-form-def}--\eqref{eq:Lh-form-def}.
For viscous-dominant flow, the boundary condition is imposed in all
spatial directions using a symmetric Nitsche formulation; 
see viscosity-scaled boundary terms, which are consistently added to
enforce \mbox{$\bfu-\bfg=\bfzero$}.
On the contrary, convection-dominant
flows only require particular control of mass conservation in
wall-normal direction $(\bfu-\bfg)\cdot\bfn=0$, whereas full control
over the boundary conditions in wall-normal and wall-tangential
directions has to be ensured only on convective-dominant inflow
boundaries $\GammaIn =\{\bfx \in \Gamma : (\bfbeta\cdot\bfn)(\bfx) <
0\}$.
\end{remark}
\begin{remark}
  \label{rem:stabilization-param}
  We point out that the stabilization
  parameters~\eqref{eq:cip-s_scalings}
  are scaled differently in
  \citet{BurmanFernandezHansbo2006}. Our choice corresponds to the
  scaling proposed by \citet{Codina2008} for the orthogonal subscale
  method and \citet{KnoblochTobiska2013} for the local projection
  stabilization.
  Compared to \cite{BurmanFernandezHansbo2006}, a
  reactive scaling is added to the stabilization parameters
  which has two effects. First, it allows
  us to establish stability and approximation properties
  using norms with contributions which are
  more typical for residual-based stabilization methods,
  see~(\ref{eq:oseen-norm-up}).
  Second, for $\bfbeta, \mu \to 0$,
  inf-sup condition~(\ref{eq:inf-sup_condition_Ah_weak_norm_I})
  and the \apriori estimates~(\ref{eq:apriori-estimate-cutfem}) do not degenerate
  as they formally would do in \cite{BurmanFernandezHansbo2006}.
\end{remark}
\begin{remark}
  \label{rem:transient-stokes-symmetric}
  For the transient Stokes equations, \citet{Burman2009} showed that
  fully discretized schemes employing symmetric pressure stabilizations
  are unconditionally stable when the initial data is properly
  preprocessed. Thus, for CIP stabilized methods, no
  time step related stabilization is needed in the small time-step
  limit and from this perspective,
  the incorporation of $\sigma$ in the stabilization
  parameter seems to be a theoretically unsatisfactory artifact
  of our theoretical analysis.
  The extension and improvement of the presented numerical analysis
  of our cut finite element method to cover fully space and time
  discretized flow problems in the small time-step limit
  is subject of future research.
\end{remark}

\subsection{A Stabilized Nitsche-type Cut Finite Element Method for the Oseen Problem}
\label{ssec:cutfem-oseen}
A major challenge in translating a fitted finite element formulation into its cut finite element
counterpart is to maintain the stability and approximation properties
of the underlying scheme irrespective of how the boundary of the domain
cuts the background mesh.
To extend the stability properties of the CIP method into the
fictitious domain $\Oast_h$ defined by the active background mesh,
we add so-called
ghost-penalties~\cite{Burman2010,BurmanHansbo2012,Burman2015a,MassingLarsonLoggEtAl2014} 
consisting of CIP-type jump penalties of order $k$:
\begin{align}
  \label{eq:ghost-penalty-beta}
  g_{\beta}(\bfu_h, \bfv_h)
  :=&
  \gamma_{\beta}
  \sum_{F\in\Fast}
  \sum_{0\leqslant j \leqslant k-1}
   \phi_{\beta,F}
      h^{2j+1}
  (
  \jump{\bfbeta\cdot \nabla\nablan^j\bfu_h}
  ,\jump{\bfbeta\cdot \nabla\nablan^j\bfv_h}
      )_F,
  \\
  \label{eq:ghost-penalty-u}
  g_{u}(\bfu_h, \bfv_h)
  :=&
  \gamma_{u}
  \sum_{F\in\Fast}
  \sum_{0\leqslant j \leqslant k-1}
 \phi_{u,F} h^{2j+1}
  (
  \jump{\div \nablan^j \bfu_h},
  \jump{\div \nablan^j \bfv_h}
  )_F,
  \\
  \label{eq:ghost-penalty-p}
  g_p(p_h, q_h) :=&
  \gamma_{p}
  \sum_{F\in\Fast}
  \sum_{1\leqslant j \leqslant k}
 \phi_{p,F} h^{2j-1}
(\jump{\nablan^j p_h}, \jump{\nablan^j q_h})_F,
\end{align}
where the $j$-th normal derivative $\nablan^j v$ is given by 
$\nablan^j v = \sum_{| \alpha | = j}D^{\alpha} v(\bfx)
  \bfn^{\alpha}$ for multi-index \mbox{$\alpha = (\alpha_1, \ldots,
  \alpha_d)$}, $|\alpha| = \sum_{i} \alpha_i$ and $\bfn^{\alpha} =
  n_1^{\alpha_1} n_2^{\alpha_2} \cdots n_d^{\alpha_d}$.
Note that if the finite element base space $\mcX_h$ consists of piecewise polynomials
of order $k=1$, these ghost-penalties reduce precisely to the CIP stabilization~(\ref{eq:cip-s_beta})--(\ref{eq:cip-s_p}), but only considered on $\mcF_{\Gamma}$.
In addition, we will need ghost-penalties to stabilize the viscous and
reactive parts of the bilinear form $a_h$ defined in (\ref{eq:ah-form-def}):
  \begin{align}
  \label{eq:ghost-penalty-sigma}
  g_{\sigma}(\bfu_h, \bfv_h)
  :=&
  \gamma_{\sigma}
  \sum_{F\in\Fast}
  \sum_{1\leqslant j \leqslant k}
  \sigma
  h^{2j+1}(\jump{\nablan^j \bfu_h},\jump{\nablan^j \bfv_h})_F,
  \\
  \label{eq:ghost-penalty-nu}
  g_{\mu}(\bfu_h, \bfv_h)
  :=&
  \gamma_{\mu}
  \sum_{F\in\Fast}
  \sum_{1\leqslant j \leqslant k}
  \mu
  h^{2j-1}(\jump{\nablan^j \bfu_h},\jump{\nablan^j \bfv_h})_F.
  \end{align}
Note that by the definition of $\mcF_{\Gamma}$, see~\eqref{eq:ghost-penalty-facets},
ghost-penalties are only evaluated
on facets in the vicinity of the boundary.
We are now in the position to formulate a
ghost-penalty enhanced continuous interior penalty method 
for the Oseen problem: find
$U_h=(\bfu_h, p_h) \in \mcW_h $ such that $\foralls V_h=(\bfv_h, q_h)\in \mcW_h$
\begin{align}
 A_h(U_h,V_h) + S_h(U_h,V_h) + G_h(U_h,V_h) = L_h(V_h),
 \label{eq:oseen-discrete-unfitted}
\end{align}
where $G_h(\cdot, \cdot)$ denotes the sum of all ghost-penalty
operators~\eqref{eq:ghost-penalty-beta}--\eqref{eq:ghost-penalty-nu}
and $A_h,S_h,L_h$ are defined as in~\eqref{eq:Ah-form-def}--\eqref{eq:Lh-form-def}.
  \begin{remark}
  \label{rem:gbeta-simple}  
  Following the discussion in~\cite{BraackBurmanJohnEtAl2007,Burman2005},
  it is possible to replace the convection and incompressibility related
  stabilization forms (\eqref{eq:cip-s_beta}, \eqref{eq:ghost-penalty-beta} and \eqref{eq:cip-s_u}, \eqref{eq:ghost-penalty-u})
  by a single stabilization and ghost penalty operator of the form
  \begin{align}
    \overline{s}_{\beta}(\bfu_h, \bfv_h)
    &:= \gamma_{\beta} \sum_{F \in \mcF_i} 
      \overline{\phi}_{\beta} h (\jump{\nablan \bfu_h}, \jump{\nablan \bfv_h})_{F},
      \label{eq:sbeta-simple-def}
    \\
    \overline{g}_{\beta}(\bfu_h, \bfv_h)
    &:= \gamma_{\beta} \sum_{F \in \Fast} \sum_{1\leqslant j\leqslant k}
    \overline{\phi}_{\beta} h^{2j -1} (\jump{\nablan^j \bfu_h}, \jump{\nablan^j \bfv_h})_{F},
  \label{eq:gbeta-simple-def}
  \end{align}
  with
  $\overline{\phi}_{\beta} = \|\bfbeta \|_{0,\infty, F}^2\phi_{\beta}$.
  We refer to Lemma~\ref{lem:gbeta-simple} for the details.
  Note that employing $\|\bfbeta \|_{0,\infty,T}$ in  $\overline{\phi}_{\beta}$
  introduces some additional (order preserving) cross-wind diffusion.
  The use of $\overline{s}_{\beta}$ and $\overline{g}_{\beta}$ greatly
  simplifies the implementation of the purposed method as
  each employed stabilization is then the sum of properly scaled face contributions
  of the form $(\jump{\nablan^j \bfu_h}, \jump{\nablan^j \bfv_h})_{F}$.
\end{remark}
\begin{remark}
  Note that the classical CIP method was introduced on fitted meshes and
  that only the gradient and no higher-order derivatives are penalized.
\end{remark}
\begin{remark}
We like to comment on the use of $\bfbeta$ in the unfitted mesh case.
From a practical point of view, $\bfbeta$ will be either given as
analytical expression 
or as the finite element approximation of $\bfu$ from 
a previous time or iteration step when solving the incompressible Navier-Stokes equations.
From a theoretical point of view, it is well know that
for any fixed Lipschitz-domain $\Oast$ satisfying 
$\Oast_h \subset \Oast \;\foralls h \lesssim 1$,
there exists an extension $\bfbeta^\ast \in [W^{1,\infty}(\Oast)]^d$
from $\Omega$ to $\Oast$ satisfying
  $\| \bfbeta^{\ast} \|_{1,\infty,\Oast}
\lesssim
  \| \bfbeta \|_{1,\infty,\Omega}
$.
To simplify the notation, we will always write
$\bfbeta$, even for its extension $\bfbeta^{\ast}$. 
\end{remark}

\subsection{Summary of Stability and A Priori Error Estimates for the Proposed Cut Finite Element Method}
We conclude this section by summarizing the main theoretical results
for the cut finite element formulation~(\ref{eq:oseen-discrete-unfitted})
and postpone the detailed numerical analysis 
to Section~\ref{sec:interpolation-est}--\ref{sec:apriori-analysis}.
The numerical analysis will utilize
the natural energy-norm for the velocity
given by
\begin{align}
\label{eq:oseen-norm-unfitted-u}
\tn \bfu_h \tn_h^2 &:= \tn \bfu_h \tn^2
+ g_{\sigma}(\bfu_h, \bfu_h)
+ g_{\mu}(\bfu_h, \bfu_h)
+ g_{\beta}(\bfu_h, \bfu_h)
+ g_u(\bfu_h, \bfu_h),
\intertext{where}
  \label{eq:oseen-norm-u}
  \tn \bfu_h  \tn^2 &:=
 \| \sigma^{1/2} \bfu_h \|_{\Omega}^2
 + \| \mu^{1/2} \grad \bfu_h\|_{\Omega}^2
+ \| (\gamma \mu/h)^{1/2} \bfu_h \|^2_{\Gamma} + s_u(\bfu_h, \bfu_h)
\nonumber
 \\
   & \quad
+ \| |\bfbeta \cdot \bfn |^{1/2} \bfu_h \|_{\Gamma}^2
+ \| (\gamma \phi_u/h )^{1/2} \bfu_h \cdot \bfn \|^2_{\Gamma}
+  s_{\beta}(\bfu_h, \bfu_h).
\end{align}
Adding the pressure related stabilization terms $s_p$ and $g_p$ we obtain the semi-norm
\begin{align}
  |U_h|_h^2 := |(\bfu_h, p_h)|_h^2 = \tn \bfu_h \tn_h^2 + |p_h|^2_h
  \qquad \text{ with } \qquad | p_h |^2_h := s_p(p_h, p_h) + g_p(p_h,p_h).
  \label{eq:Ah-semi-norm}
\end{align}
Finally, the main analytical results will be stated using the ghost-penalty
augmented energy norm
\begin{align}
  \tn U_h \tn_h^2 
  := | U_h |_h^2 + \| \phi_u^{\onehalf} \nabla \cdot \bfu_h \|_{\Omega}^2
  + \dfrac{1}{1 + \omega_h} \|\phi_{\beta}^{\onehalf}(\bfbeta\cdot \nabla \bfu_h + \nabla p_h) \|_{\Omega}^2
  + \Phi_p \| p_h\|_{\Omega}^2,
  \label{eq:oseen-norm-up}
\end{align}
where
\begin{align}
  \label{eq:phi_p_definition}
  \Phi_p^{-1} 
  := \sigma C_P^2 + \norm{\bfbeta}_{0,\infty,\Omega}C_P 
  + \mu 
  + \left(\frac{\norm{\bfbeta}_{0,\infty,\Omega}C_P}{\sqrt{\mu + \sigma C_P^2}}\right)^2,
  \qquad
  \omega_h 
  := \dfrac{h^2 |\bfbeta|_{1,\infty,\Omega}}{\mu + \sigma h^2}.
\end{align}
Therein, $C_P$ denotes the so-called Poincar\'e constant as defined in \eqref{eq:Poincare-I} in Section~\ref{sec:interpolation-est}.
\begin{remark}
The concept of ghost-penalties was first introduced by
\citet{Burman2010} and \citet{BurmanHansbo2012} to formulate optimally convergent fictitious
domain methods for the Poisson problem.
As for instance shown in \cite{Burman2010}, using norms which are formulated only
in terms of the actual physical domain $\Omega$ leads to suboptimal
and non robust a priori error and condition number estimates due the
possible appearance of small cut elements $|T \cap \Omega| \ll |T|,\;
T\in \mcT_h$ in the vicinity of the boundary $\pO$.
Augmenting the original bilinear form
with the ghost-penalty stabilization
extends, roughly speaking, the naturally induced norms
from the physical domain~$\Omega$ to the entire fictitious domain
$\Oast_h$ defined by the active background mesh $\mcT_h$.
A more detailed mathematical explanation will be given in Section~\ref{ssec:norms}.
\end{remark}
In Sections~\ref{sec:stability-properties} and~\ref{sec:apriori-analysis}
we will prove the following inf-sup condition and a priori error estimates
with the hidden constant being independent of the mesh size $h$, and the
relative position of the boundary with respect to the active background mesh:
  \begin{itemize}
  \item For $U_h \in V_h$ it holds
 \begin{equation}
\label{eq:inf-sup_condition_Ah_weak_norm_I}
\tn U_h \tn_h
\lesssim
\sup_{V_h\in \mcW_h \setminus \{0\}}
\dfrac{
A_h(U_h,V_h) + S_h(U_h,V_h)
+ G_h(U_h,V_h)
}
{\tn V_h \tn_h }.
\end{equation}
\item 
Let $U = (\bfu, p) \in [H^r(\Omega)]^d \times H^s(\Omega)$
be the weak solution of the Oseen problem~\eqref{eq:oseen_weak} and let
$U_h = (\bfu_h,p_h) \in \mcV_h \times \mcQ_h$ be the discrete solution
of problem~\eqref{eq:oseen-discrete-unfitted}. Then
\begin{align}
    \tn \bfu -  \bfu_h \tn
  + \Phi_p^{\onehalf}\| p - p_h \|_{\Omega}
    &\lesssim
    (1+\omega_h)^{\onehalf}
    \bigl(
    \mu + \|\bfbeta \|_{0,\infty,\Omega} h + \sigma h^2
    \bigr)^{\onehalf}
    h^{r_{u} - 1}
    \| \bfu \|_{r_{u},\Omega}
      \nonumber
    \\
    &\phantom{\lesssim}\quad
      +
    \left(
    \Phi_p +
   \max_{T\in \mcT_h}
   \left\{\dfrac{1}{\mu + \|\bfbeta \|_{0,\infty, T} h + \sigma h^2}
   \right\}
   \right)^{\onehalf}
   h^{ s_p}
   \| p \|_{s_p, \Omega},
    \label{eq:apriori-estimate-cutfem}
  \end{align} 
where $r_{u} := \min\{r, k+1\}$ and $s_p := \min\{s, k+1\}$ with $k$
being the polynomial order of the discrete velocity and pressure
spaces.
  \end{itemize}

\section{Numerical Examples}
\label{sec:numexamples}

To validate our proposed stabilized cut finite element method,
different numerical examples are investigated.  Theoretical results
for the Oseen equations obtained from the \apriori~error analysis, see \eqref{eq:apriori-estimate-cutfem}
and Theorem~\ref{thm:apriori-estimate}, will
be confirmed by several basic test examples: the Taylor problem in two
dimensions and the Beltrami-flow problem in three dimensions.
Thereby, convergence properties are examined for the low- and the
high-Reynolds-number regime.  Finally, we demonstrate the
applicability of our stabilized method for solving the time-dependent
Navier-Stokes equations on complex three-dimensional geometries.  For
this purpose we show results of a flow through a helical pipe.

At this point, we would like to refer to a preceding
work~by \citet{SchottWall2014} in which our formulation has been
investigated by means of a number of various flow scenarios.
The numerical examples provided therein include detailed investigations of the different stabilization operators
and compare the numerical approach to other methods by means of computed lift and drag values for the flow around a cylinder.
For the applicability of our cut finite element method to more complex flow scenarios,
the interested reader is referred to some further publications, which are based on the present flow formulation:
\citet{SchottShahmiriKruseWall2015} extended the present formulation to an overlapping mesh approach in which, for instance,
statistical measures of the turbulent flow in a lid-driven cavity at $\RE=10000$ have been compared to a boundary-fitted mesh approach.
Application of our method to low- and high-Reynolds-number incompressible two-phase flows
has been provided by \citet{SchottRasthoferGravemeierWall2015}.

Moreover, the publication \cite{SchottWall2014} includes several studies on the choice of
stabilization parameters involved in our formulation, which provides
the basis for all examples proposed in this work.
Following~\cite{SchottWall2014}, for the CIP-stabilization terms
\eqref{eq:cip-s_beta}--\eqref{eq:cip-s_p} we take $\gamma_\beta =
\gamma_p=0.05$ and set $\gamma_u=0.05\gamma_\beta$, as suggested in
\cite{Burman2007}.  Same parameters are used for the related
ghost-penalty terms
\eqref{eq:ghost-penalty-beta}--\eqref{eq:ghost-penalty-p}.  As studied
in \cite{SchottWall2014}, we choose $\gamma = 30.0$ for the
Nitsche-penalty terms and $\gamma_\mu=0.05$ for the viscous
ghost-penalty term \eqref{eq:ghost-penalty-nu}.  The parameter for the
(pseudo-)reactive ghost-penalty term \eqref{eq:ghost-penalty-sigma},
however, is set to a considerably smaller value $\gamma_\sigma = 0.001$.
Moreover, the different flow regimes appearing in $\phi_u,\phi_{\beta},\phi_p$ are
weighted as $\mu + c_u (\|\bfbeta\|_{0,\infty,T}h) + c_{\sigma} (\sigma h^2)$ with
$c_u = 1/6$ and $c_\sigma= 1/12$ as suggested in \cite{SchottRasthoferGravemeierWall2015}.

All simulations presented in this publication have been performed
using the parallel finite element software environment “Bavarian
Advanced Computational Initiative” (BACI), see \cite{WallGeeBaci2012}.

\subsection{Convergence Study - 2D Taylor Problem}
\label{ssec:numexamples:taylor_problem}

To confirm the optimal order \apriori~error estimate stated in \eqref{eq:apriori-estimate-cutfem} and 
Theorem~\ref{thm:apriori-estimate}, we study error convergence for the
two-dimensional Taylor problem, see also~\cite{Kim1985,Chorin1968,Pearson1965}.
Periodic steady velocity and pressure fields $(\bfu,p)$ are given as
\begin{align}
 u_{1}(x_1,x_2) &= -\cos(2\pi x_1)\sin(2\pi x_2) \label{eq:kimmoin_function_u_x},\\
 u_{2}(x_1,x_2) &= \sin(2\pi x_1)\cos(2\pi x_2)  \label{eq:kimmoin_function_u_y},\\
 p(x_1,x_2)     &= -0.25(\cos(4\pi x_1)+\cos(4\pi x_2)), \label{eq:kimmoin_function_p}
\end{align}
such that $\div \bfu$ = 0.  We compute the numerical solution on a
circular fluid domain
\begin{equation} \Omega=\left\{ \bfx =(x_1,x_2)\in
\mathbb{R}^2~|~\phi(x_1,x_2) =
\sqrt{(x_1-0.5)^2+(x_2-0.5)^2}-0.45<0\right\},
\end{equation} where the boundary $\Gamma$ is represented implicitly
by the zero-level set of the function $\phi$.  This level-set field is
defined on a background square domain $[0,1]^2$ and approximated on a
background mesh $\widehat{\mcT}_h$ consisting of linear right-angled
triangular elements
$\mcP^k(T),~k\in\{1,2\}$.
The right-hand side $\boldsymbol{f}$
and the boundary condition $\bfg$ are adapted such that
\eqref{eq:kimmoin_function_u_x}--\eqref{eq:kimmoin_function_p} are
solution to the Oseen problem
\eqref{eq:oseen-problem-momentum}--\eqref{eq:oseen-problem-boundary}.
Boundary conditions on $\Gamma$ are imposed using our unfitted
Nitsche-type formulation, as introduced in
Section~\ref{ssec:cutfem-oseen}.  The constant pressure mode is
filtered out in the iterative solver, such that
$\int_{\Omega}{p_h-p}\dx = 0$.  The resulting Oseen system can be
interpreted as one time step of a backward Euler time-discretization
scheme for the linearized Navier-Stokes equations, where $\sigma =
1/\Delta t$ is the inverse of the time-step length.  The advective
velocity is given by the exact solution $\bfbeta = \bfu$ and its
discrete counterpart $\bfbeta_h$ by its nodal interpolation.

For a series of mesh sizes $h = 1/N$ with $N\in[10; 240]$, each
generated background mesh $\widehat{\mcT}_h$ consists of equal-sized triangles.
It has to be noted that the set of active elements
$\mcT_h$ used for approximating $\bfu_h$ and $p_h$ varies with mesh
refinement due to the unfitted boundary within the background mesh.
In the following, linear
and quadratic
equal-order approximations,
i.e. $\mcV_h^k\times\mcQ_h^k,~k\in\{1,2\}$,
for velocity and pressure are investigated.
We would like to point out that for all simulations with higher-order approximations, i.e. $k>1$,
the convective and divergence ghost penalty terms $g_\beta,g_u$ (see \eqref{eq:ghost-penalty-beta} and \eqref{eq:ghost-penalty-u})
and the related continuous interior penalty stabilizations $s_\beta,s_u$ (see \eqref{eq:cip-s_beta} and \eqref{eq:cip-s_u})
are replaced by the easier to implement (order-preserving) term $\overline{g}_{\beta}$ \eqref{eq:gbeta-simple-def};
see also Remark~\ref{rem:gbeta-simple} and Lemma~\ref{lem:gbeta-simple}.

Related to the triple norm $\tnorm{\cdot}$ defined in
\eqref{eq:oseen-norm-u}--\eqref{eq:oseen-norm-up}, we compute $L^2$-
and $H^1$-semi-norms to measure velocity and pressure approximation
errors $(\bfu_h-\bfu)$ and $(p_h-p)$ in the bulk $\Omega$ and on the
boundary $\Gamma$, respectively.  To examine convergence rates for
different Reynolds-number regimes, all errors are computed for two
different viscosities of $\mu = 0.1$ and $\mu=0.0001$.  Furthermore,
to investigate the effect of possibly dominating $\sigma$-scalings in
fluid-stabilization and boundary mass conservation terms, but also to
demonstrate the importance of the (pseudo-)reactive ghost-penalty
term, all studies are carried out for varying $\sigma$.

\subsubsection{Viscous dominant flow}

In the viscous case with $\mu=0.1$ the element
Reynolds numbers are low for all meshes, i.e., $\RE_T = \norm{\bfbeta}_{0,\infty,T} h/\mu
\leqslant 1$ since $\|\bfbeta\|_{0,\infty,\Omega}\leqslant 1$.  While the viscous
scalings in the Nitsche boundary terms and the pressure stabilization
terms are highly important to guarantee stability near the boundary as
well as to ensure inf-sup stability, all advective contributions to
the scalings are not required in this case.  Furthermore, the CIP
terms $s_\beta, s_u$ as well as related ghost-penalty terms $g_\beta,
g_u$ are not essential to guarantee stability.  However, the applied
scalings~\eqref{eq:cip-s_scalings} ensure sufficiently small
stabilization contributions from these terms to not deteriorate
convergence rates or to not lead to significantly increased error levels.
In \Figref{fig:kimmoin_oseen:spatial_convergence_visc}
and \Figref{fig:kimmoin_oseen:spatial_convergence_visc_quadratic},
errors computed
for our stabilized unfitted method~\eqref{eq:oseen-discrete-unfitted}
are presented
for $k=1,2$,
respectively.
As desired, optimal convergence is obtained for all
considered velocity norms, while for the pressure a superconvergent
rate of order
$k+1/2$
can be observed in the asymptotic range; this is
due to the high regularity of the solution as frequently
reported in literature before, see, e.g., in \cite{BurmanFernandezHansbo2006}.
Moreover, the optimality
$\mcO(h^{k+1})$
for the velocity $L^2$-norm error in the low
Reynolds number regime
(see Remark~\ref{rem:oseen:low-order-L2-velocity-optimality})
could be confirmed.
To further investigate the effect of
large values of $\sigma \gg 1$, which corresponds to the choice of
small time steps when $\sigma$ results from temporal discretizations, we
show the error behavior for different
$\sigma\in\{1,100,10000\}$.
While the velocity errors are robust when $\sigma$ becomes large,
the pressure $L^2$ error shows deteriorating convergence behavior.
This is most likely due to the effect of not properly chosen initial data.
Even though the right hand side is adapted being solution to the strong form of the Oseen problem,
the right hand side contains a discrete initial velocity field which is not discrete divergence free
due to the presence of the symmetric pressure stabilization terms.
The effect of a polluted incompressibility rendering in an unstable problem for the pressure has been analyzed in \cite{Burman2009}
for the transient Stokes problem.
The numerical results presented in the latter work are quite similar to the behavior observed in
\Figref{fig:kimmoin_oseen:spatial_convergence_visc}--\Figref{fig:kimmoin_oseen:spatial_convergence_conv_quadratic}.
Note that for practical flow problems, for which
the transient incompressible Navier-Stokes equations are solved and the simulation usually starts from a quiescent flow, i.e. $\bfu = \bfzero$,
it is expected that this effect does not occur.
\begin{figure}[t]
  \centering
  \subfloat{\includegraphics[trim=0 0 0 0, clip, width=0.33\textwidth]{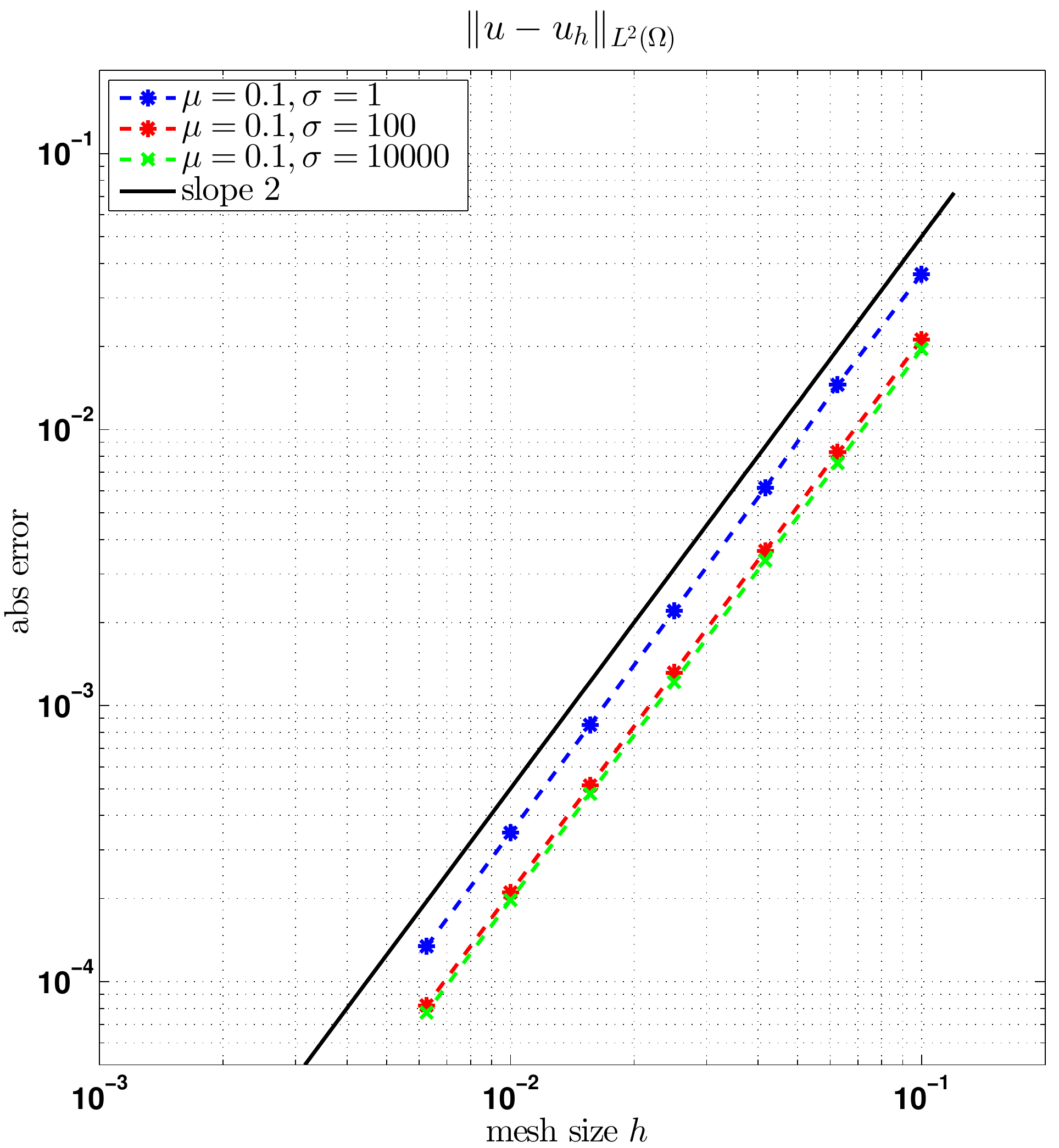}}
  \subfloat{\includegraphics[trim=0 0 0 0, clip, width=0.33\textwidth]{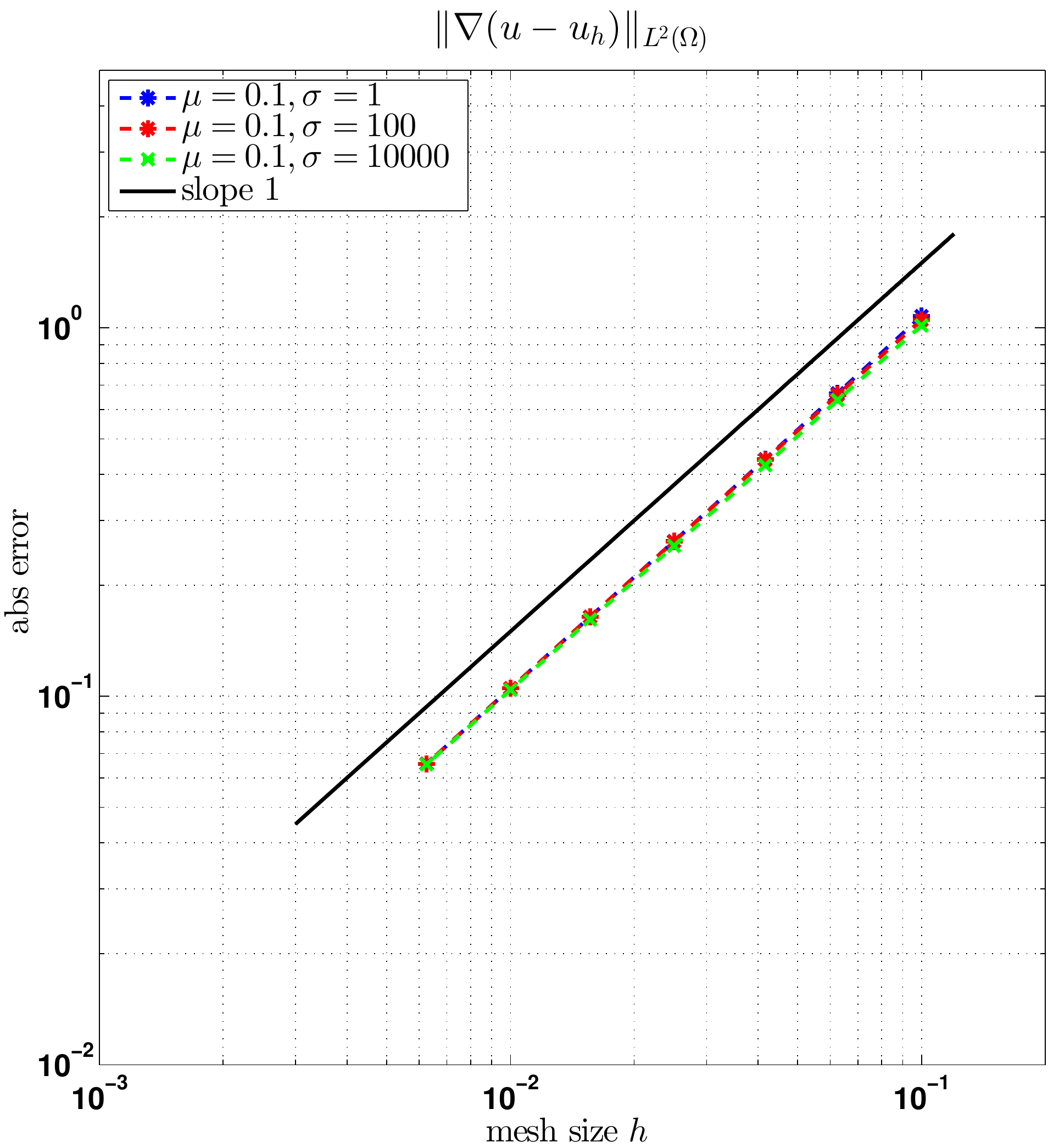}}
  \subfloat{\includegraphics[trim=0 0 0 0, clip, width=0.33\textwidth]{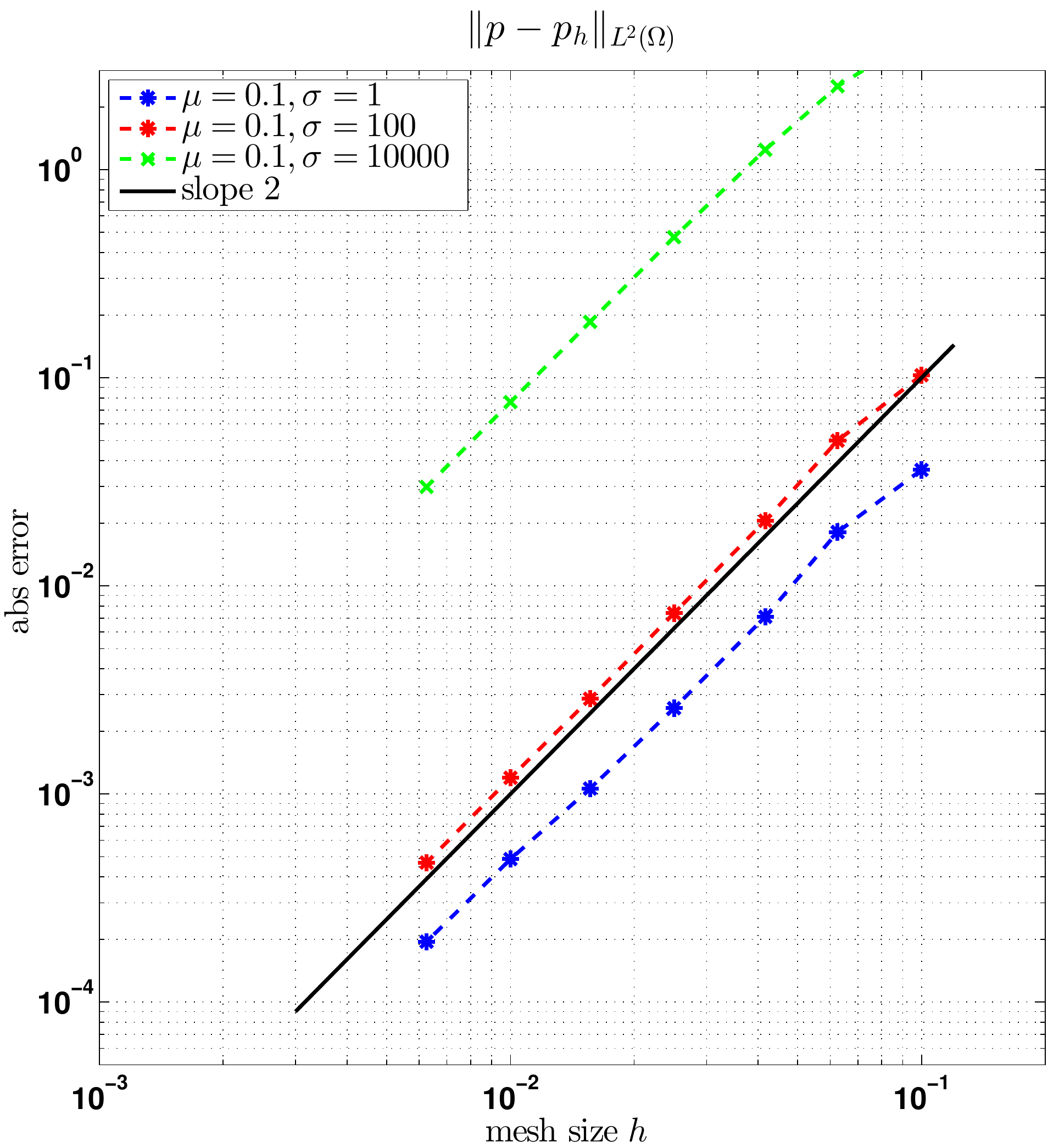}}\\
  \subfloat{\includegraphics[trim=0 0 0 0, clip, width=0.33\textwidth]{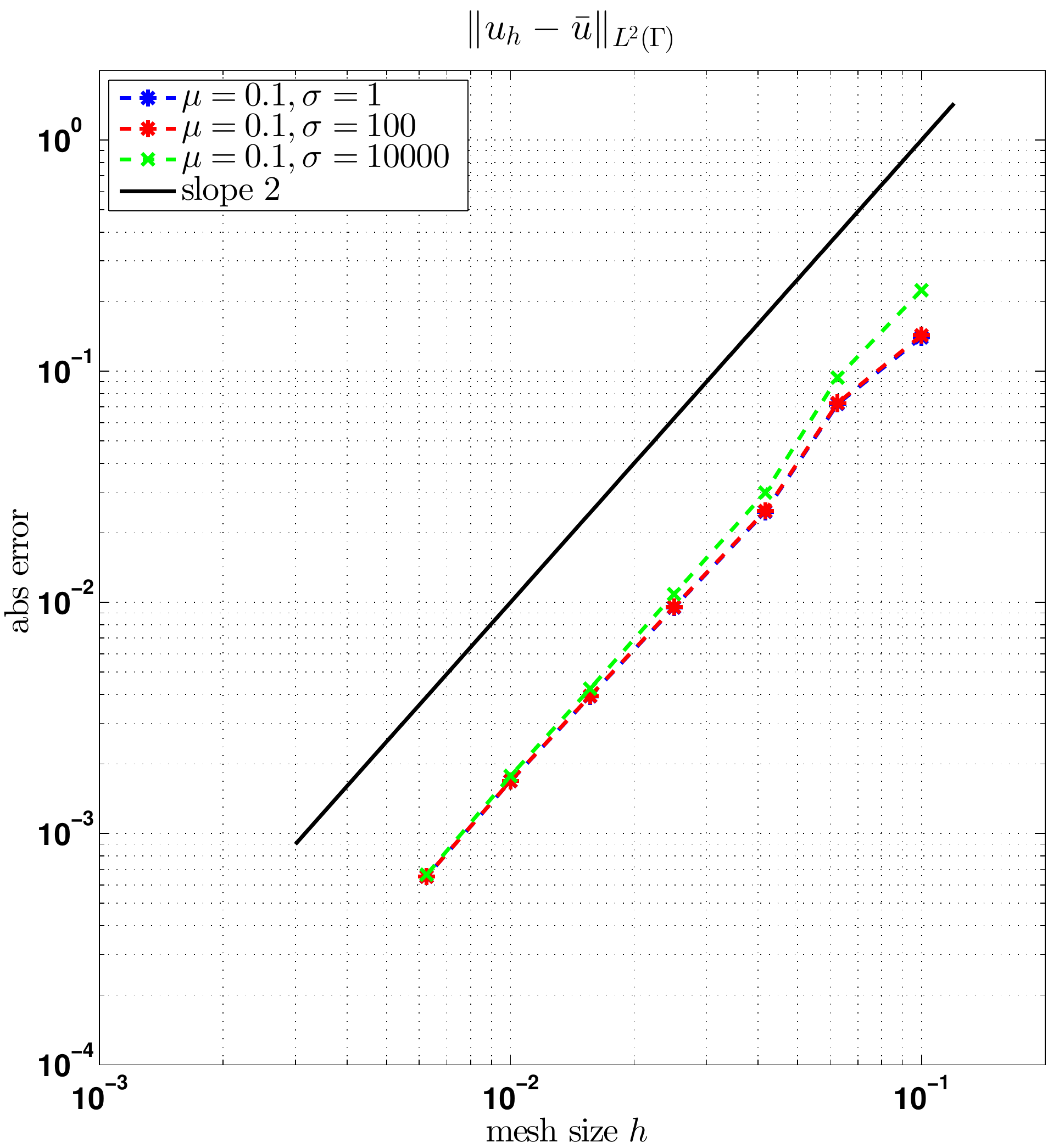}}
  \subfloat{\includegraphics[trim=0 0 0 0, clip, width=0.33\textwidth]{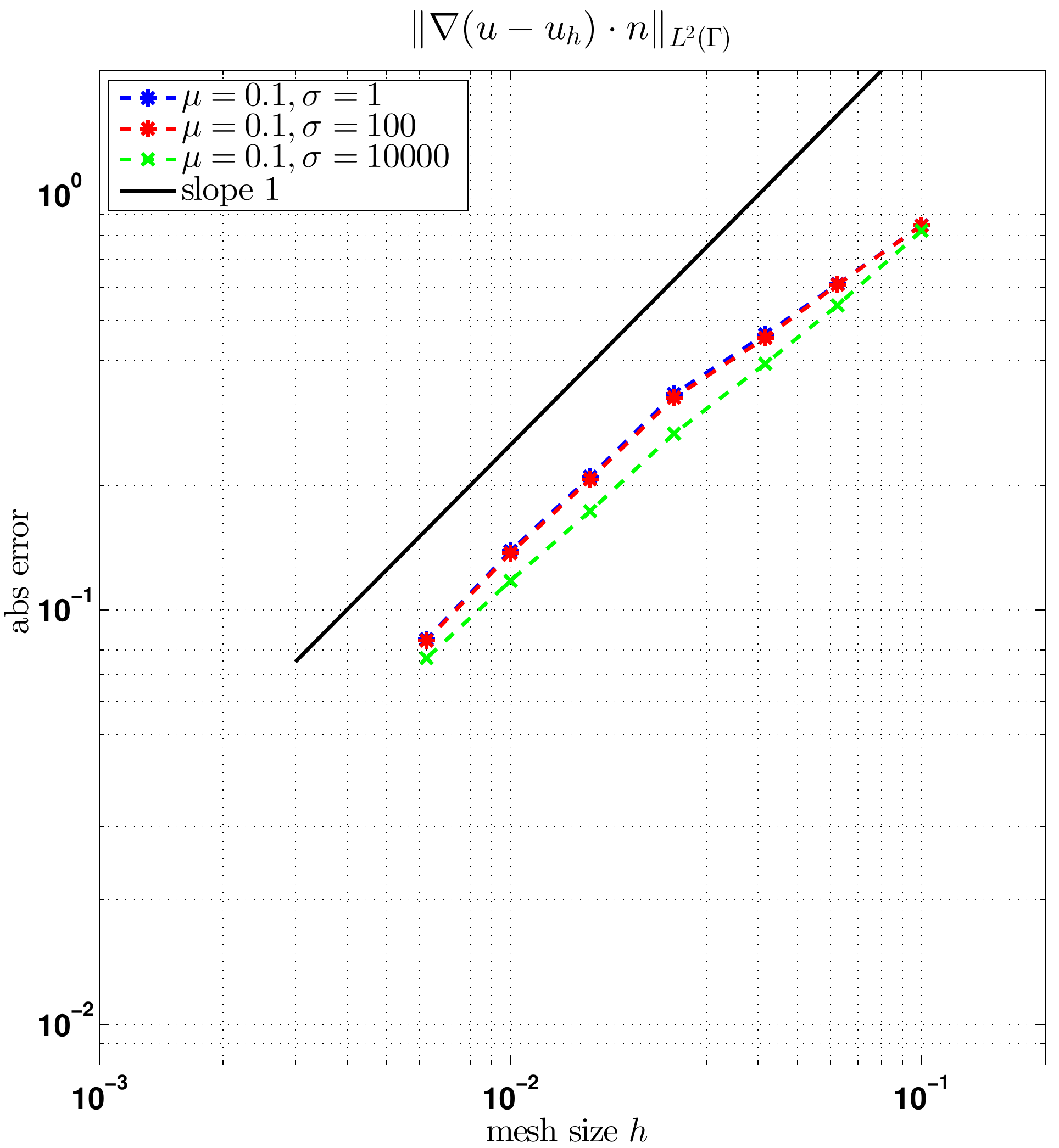}}
  \subfloat{\includegraphics[trim=0 0 0 0, clip, width=0.33\textwidth]{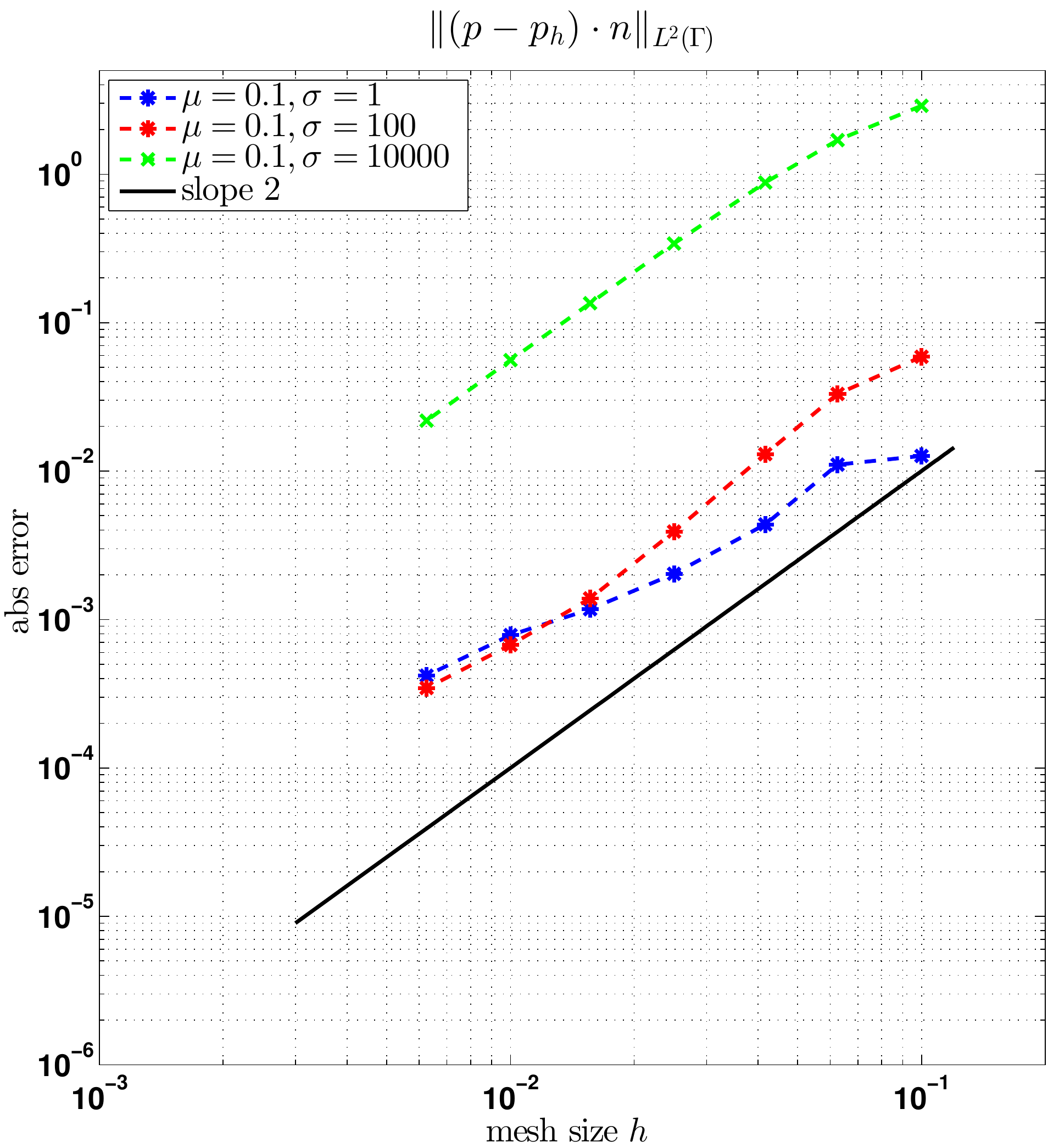}}
  \caption{Low-Reynolds-number 2D Taylor Problem with $\mu=0.1$ and $\mcP^1$ approximations: Convergence rates in $L^2$-norms for velocity, velocity gradient and pressure in the domain (top row) and on the boundary (bottom row).}
  \label{fig:kimmoin_oseen:spatial_convergence_visc}
\end{figure}

\begin{figure}[t]
  \centering
  \subfloat{\includegraphics[trim=0 0 0 0, clip, width=0.33\textwidth]{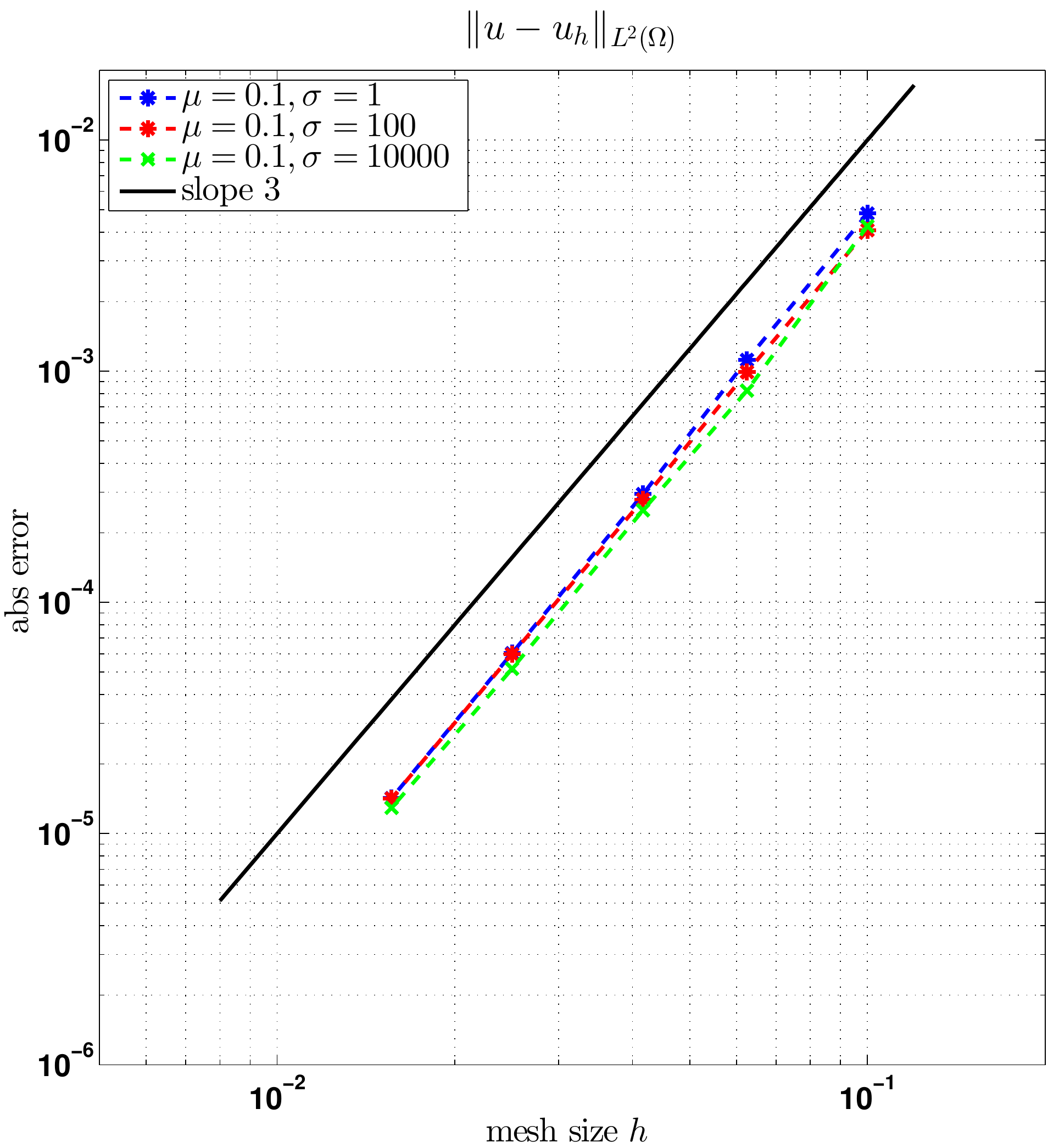}}
  \subfloat{\includegraphics[trim=0 0 0 0, clip, width=0.33\textwidth]{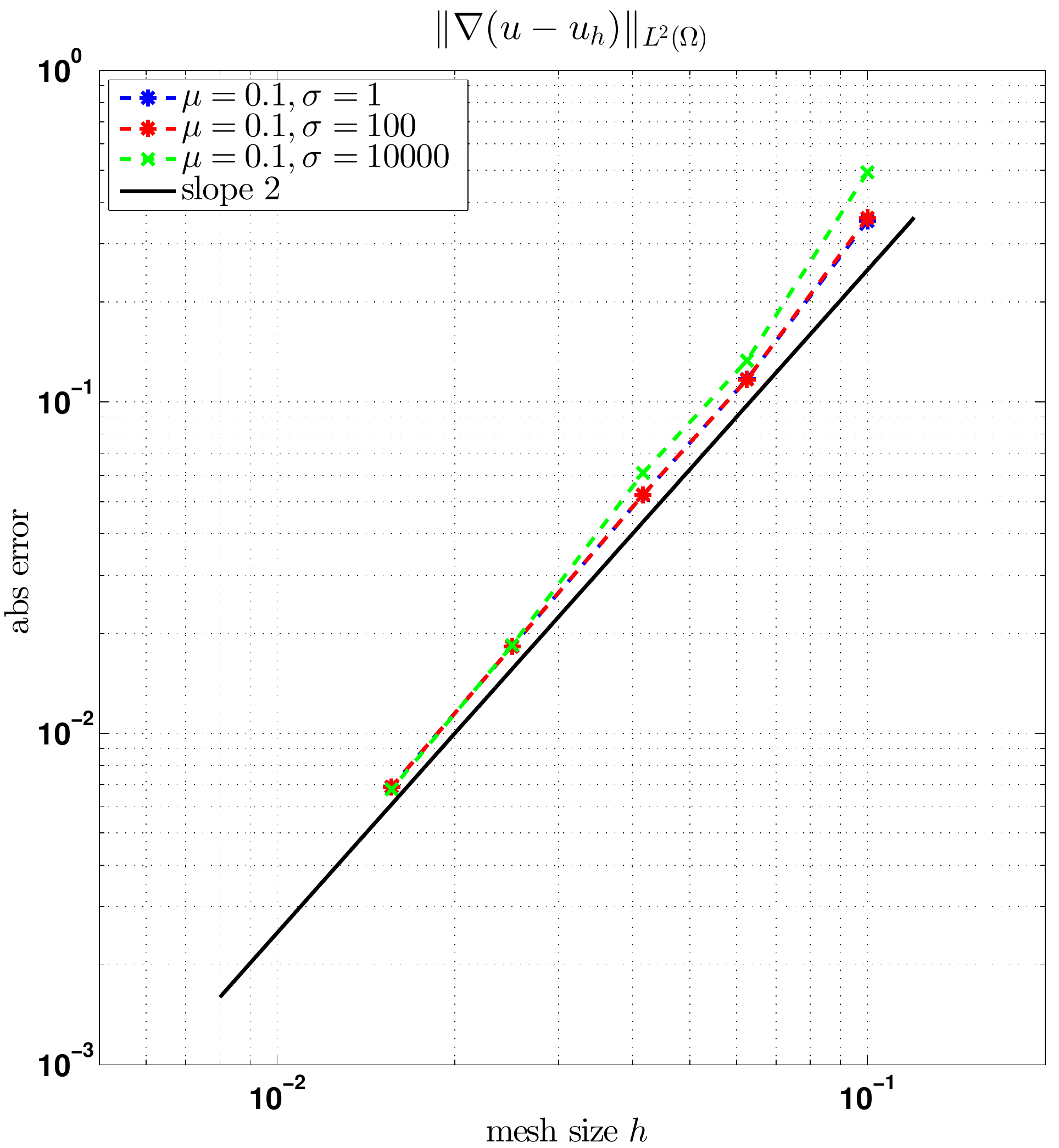}}
  \subfloat{\includegraphics[trim=0 0 0 0, clip, width=0.33\textwidth]{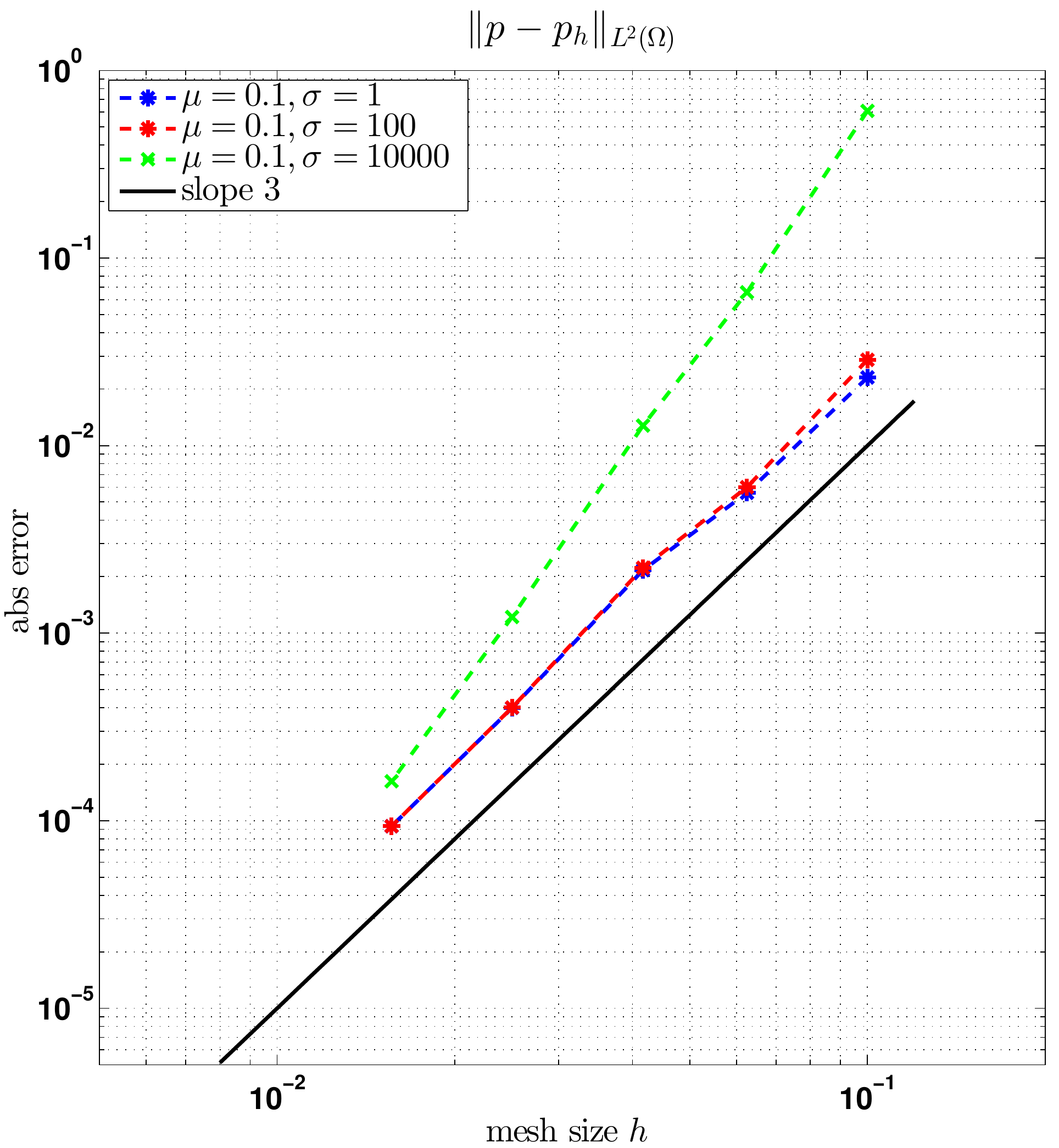}}\\
  \subfloat{\includegraphics[trim=0 0 0 0, clip, width=0.33\textwidth]{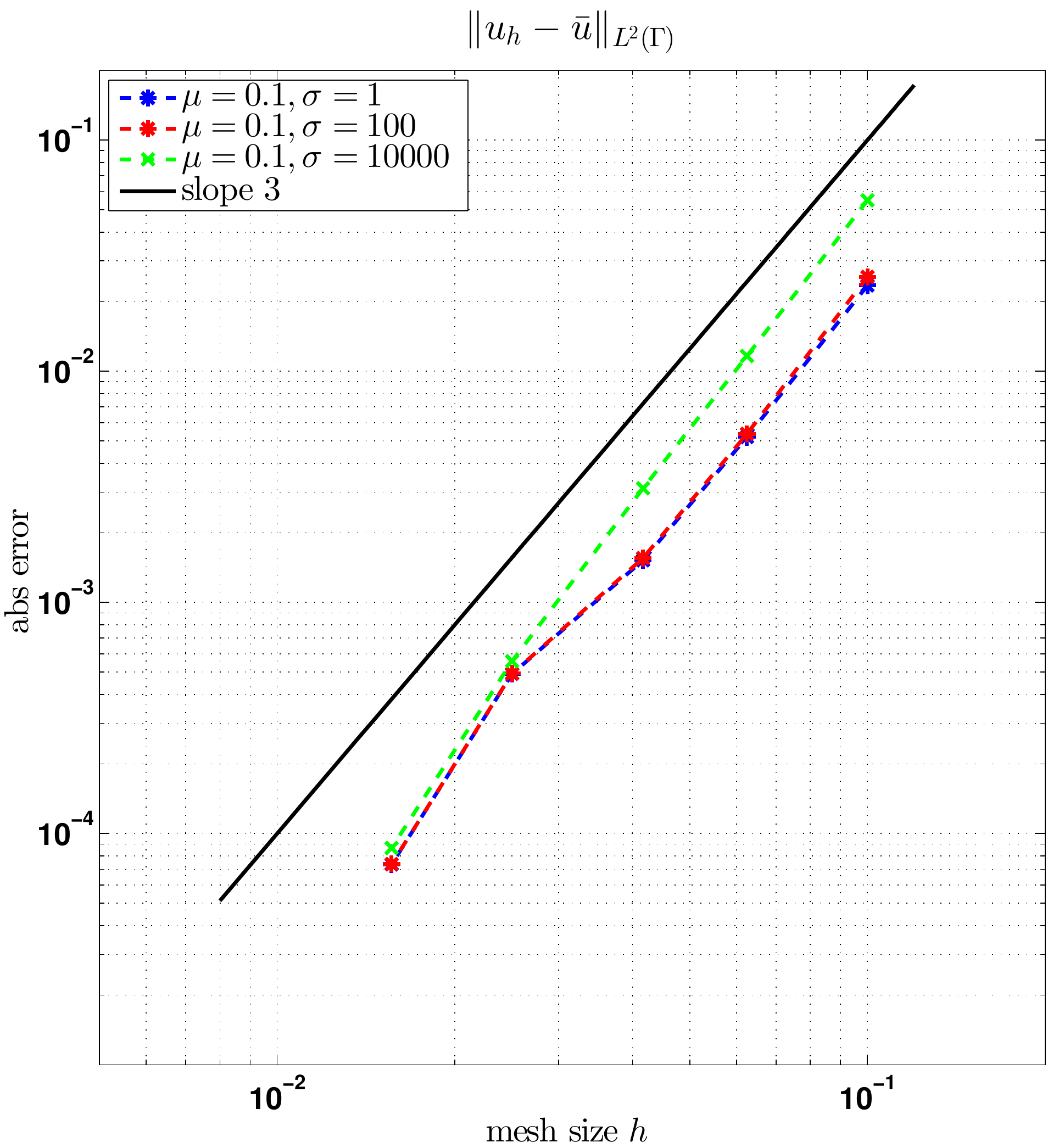}}
  \subfloat{\includegraphics[trim=0 0 0 0, clip, width=0.33\textwidth]{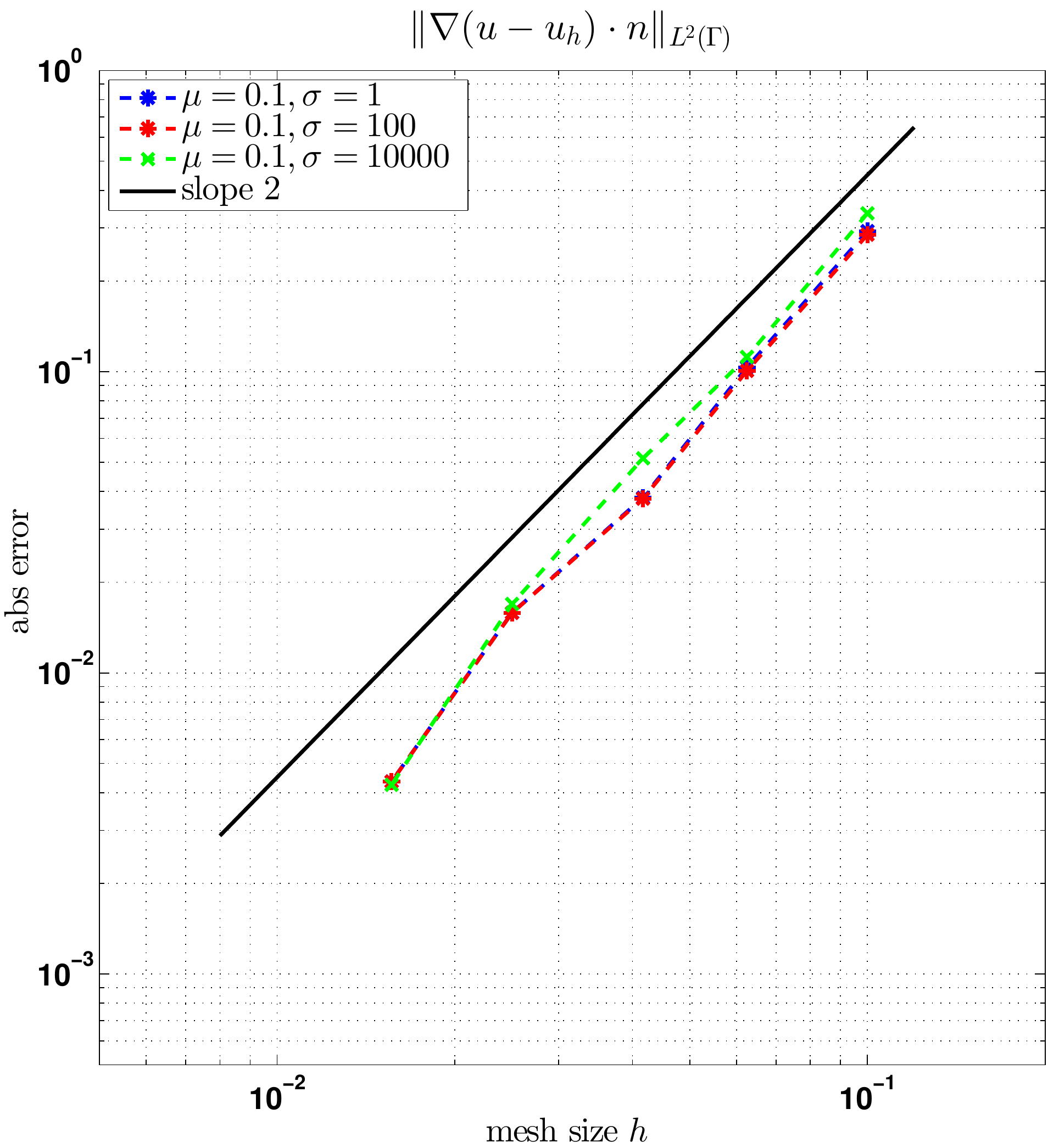}}
  \subfloat{\includegraphics[trim=0 0 0 0, clip, width=0.33\textwidth]{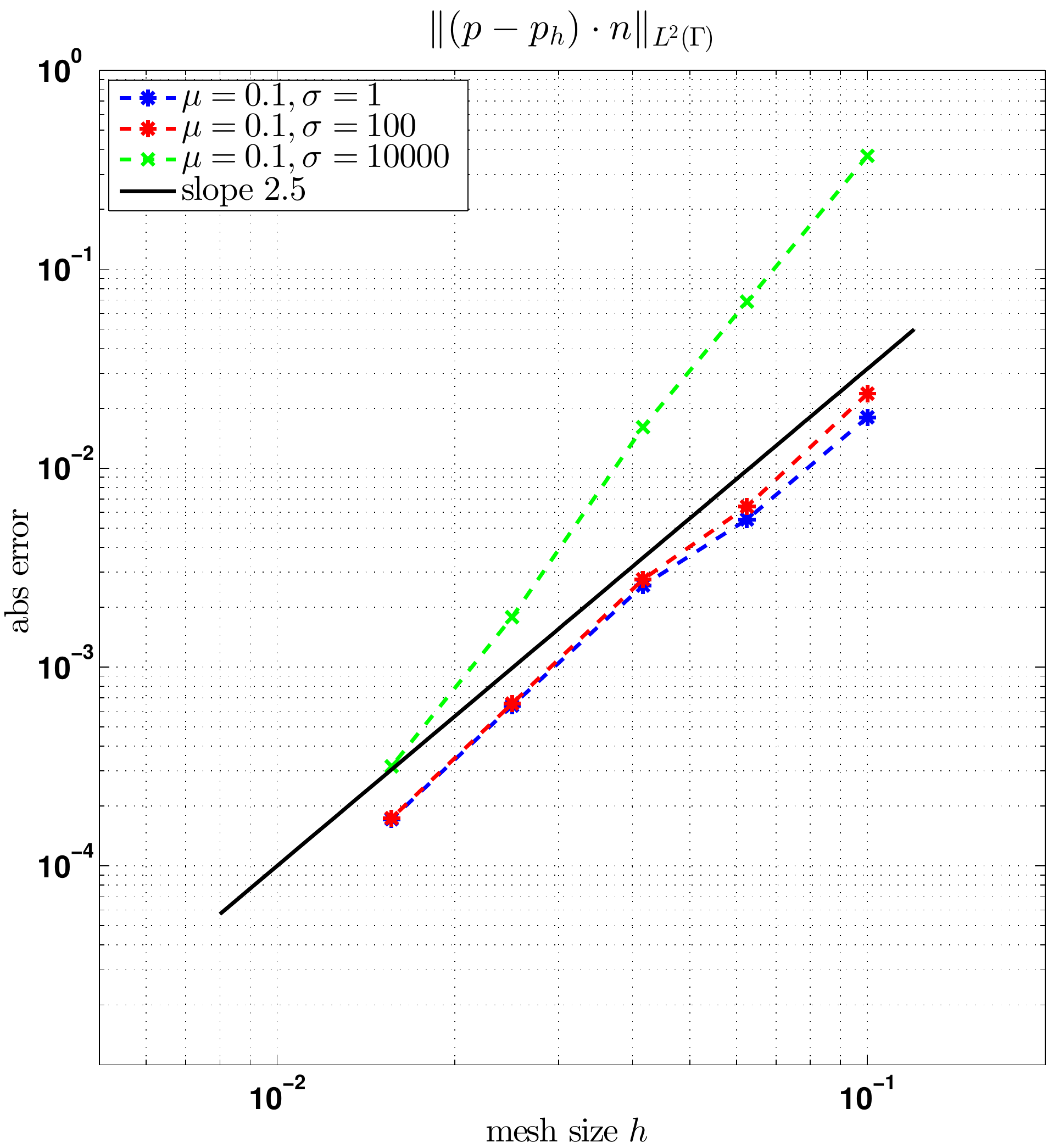}}
  \caption{
Low-Reynolds-number 2D Taylor Problem with $\mu=0.1$ and $\mcP^2$ approximations: Convergence rates in $L^2$-norms for velocity, velocity gradient and pressure in the domain (top row) and on the boundary (bottom row).
}
  \label{fig:kimmoin_oseen:spatial_convergence_visc_quadratic}
\end{figure}

\subsubsection{Convection dominant flow}

The same studies are carried out for convection dominant flow with $\mu=0.0001$.
In this setting, the resulting Oseen system exhibits highly varying element Reynolds numbers $\RE_T$ due to the locally dominating advective term $\bfbeta \cdot\nabla \bfu$.
In contrast to the previous studies, the full stabilization parameter scalings in $\phi_u,\phi_\beta,\phi_p$
including advective and reactive contributions, see \eqref{eq:cip-s_scalings},
are now required for all continuous interior penalty and related ghost-penalty stabilizations as well as for the mass conservation boundary term
to ensure inf-sup stability and optimality of the error convergence.
In~\Figref{fig:kimmoin_oseen:spatial_convergence_conv}
and \Figref{fig:kimmoin_oseen:spatial_convergence_conv_quadratic}
errors are reported for the same family of triangulations as for the viscous setting.
While the velocity approximations again show optimality in the domain and on the boundary, for the pressure we obtain convergence
of order $k+1$
, which
confirms the potential gain of half an order compared to the viscous flow regime,
similar to observations made in \cite{BurmanFernandezHansbo2006}.
\begin{figure}[t]
  \centering
  \subfloat{\includegraphics[trim=0 0 0 0, clip, width=0.33\textwidth]{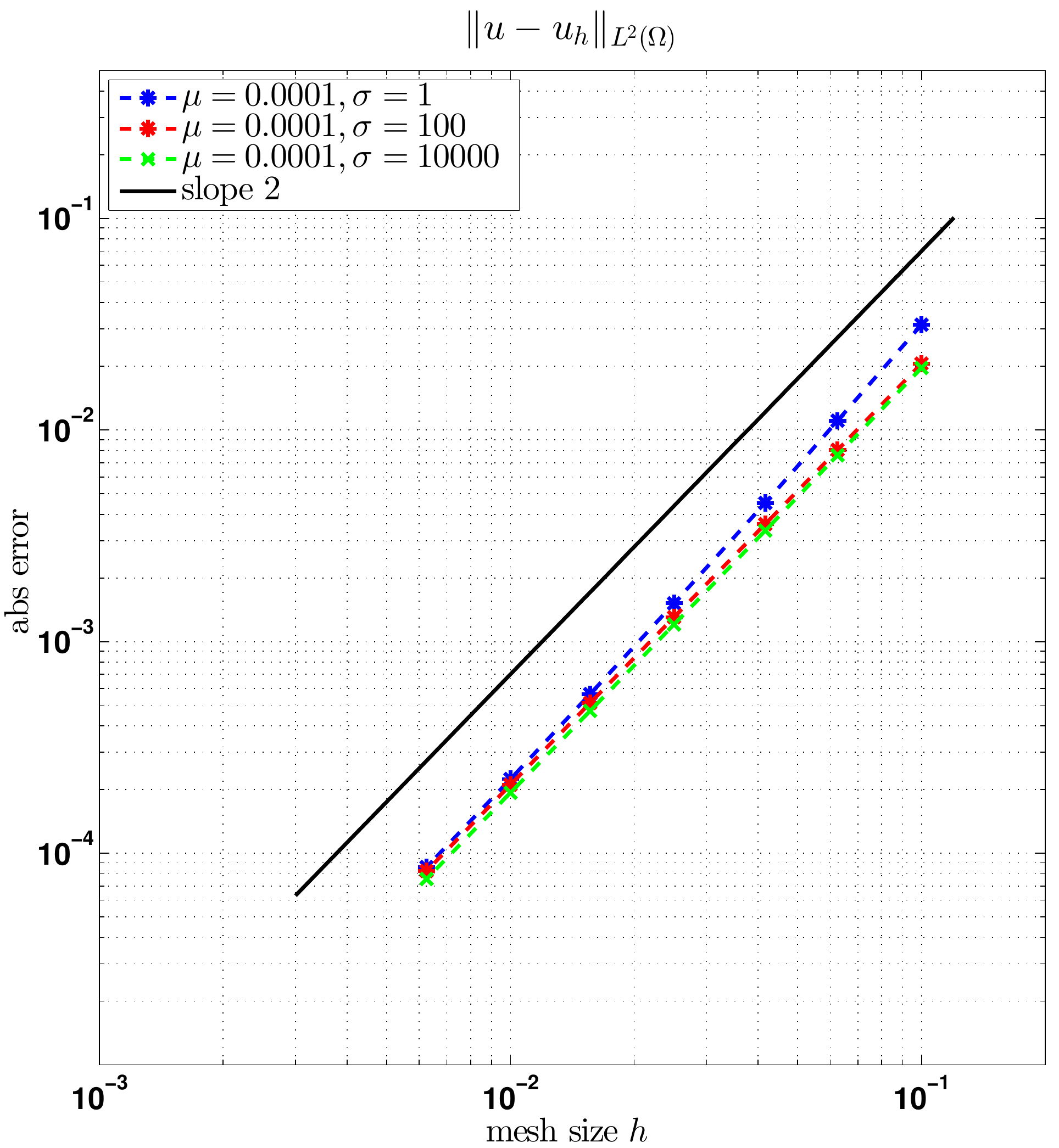}}
  \subfloat{\includegraphics[trim=0 0 0 0, clip, width=0.33\textwidth]{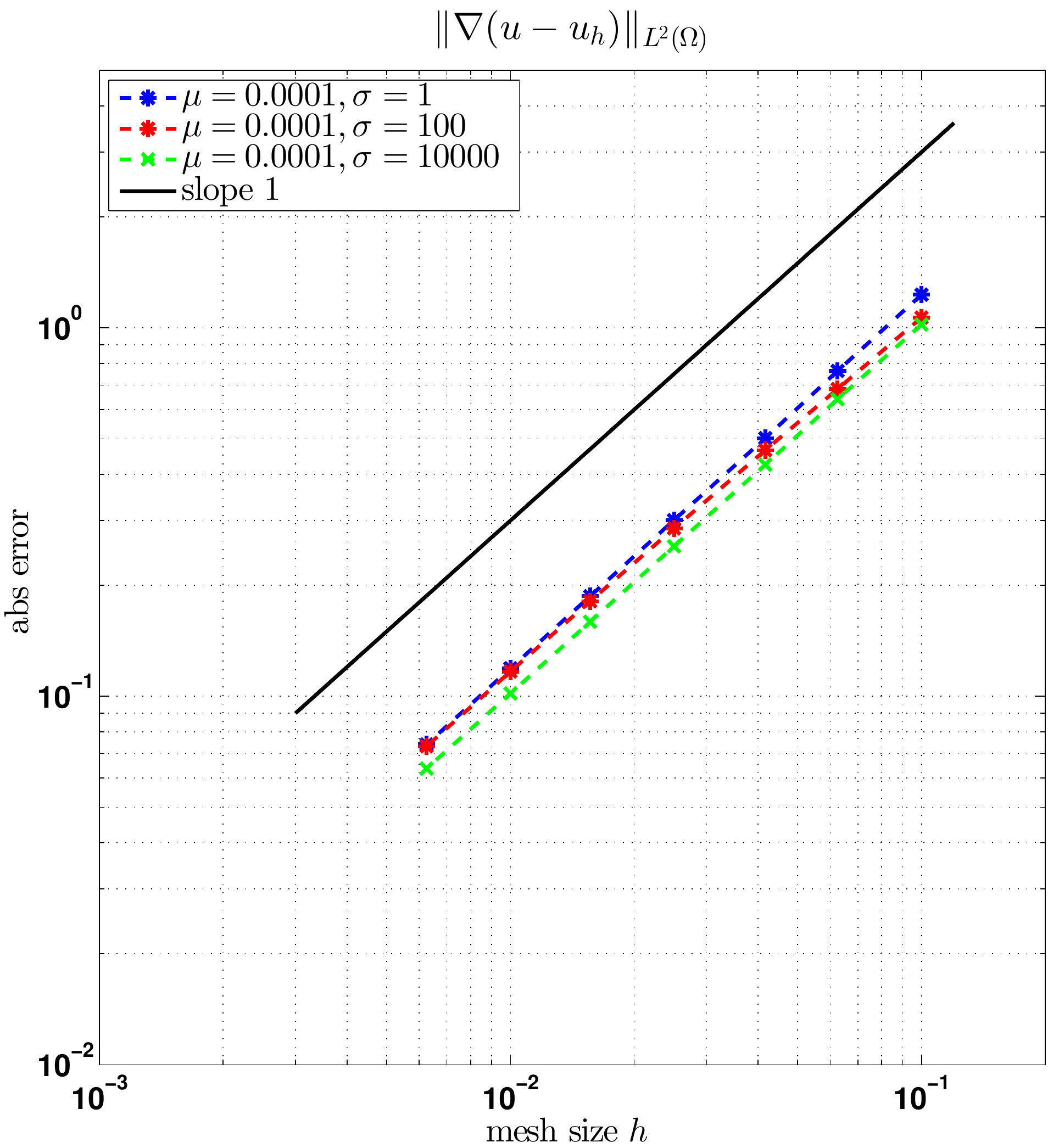}}
  \subfloat{\includegraphics[trim=0 0 0 0, clip, width=0.33\textwidth]{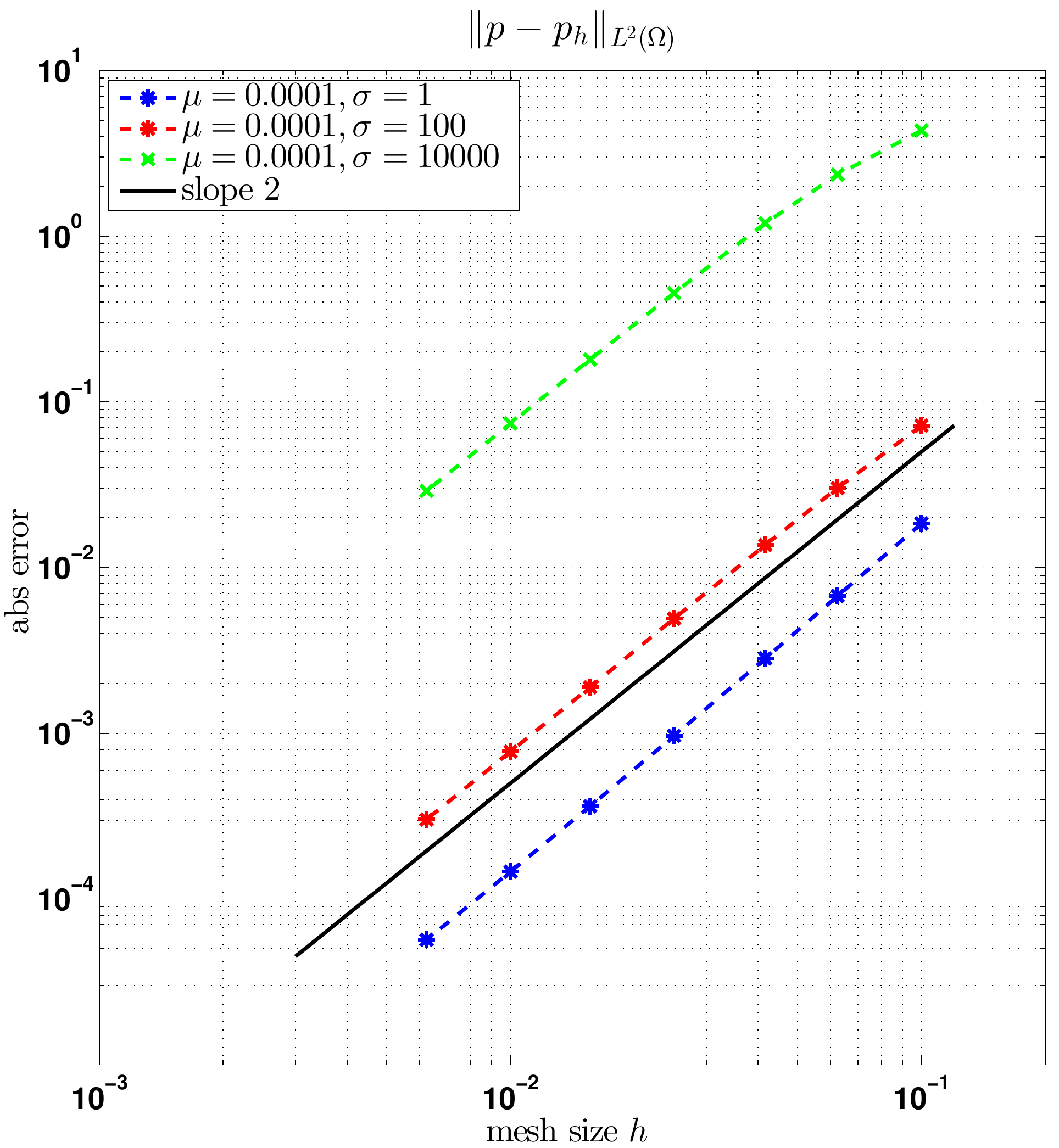}}\\
  \subfloat{\includegraphics[trim=0 0 0 0, clip, width=0.33\textwidth]{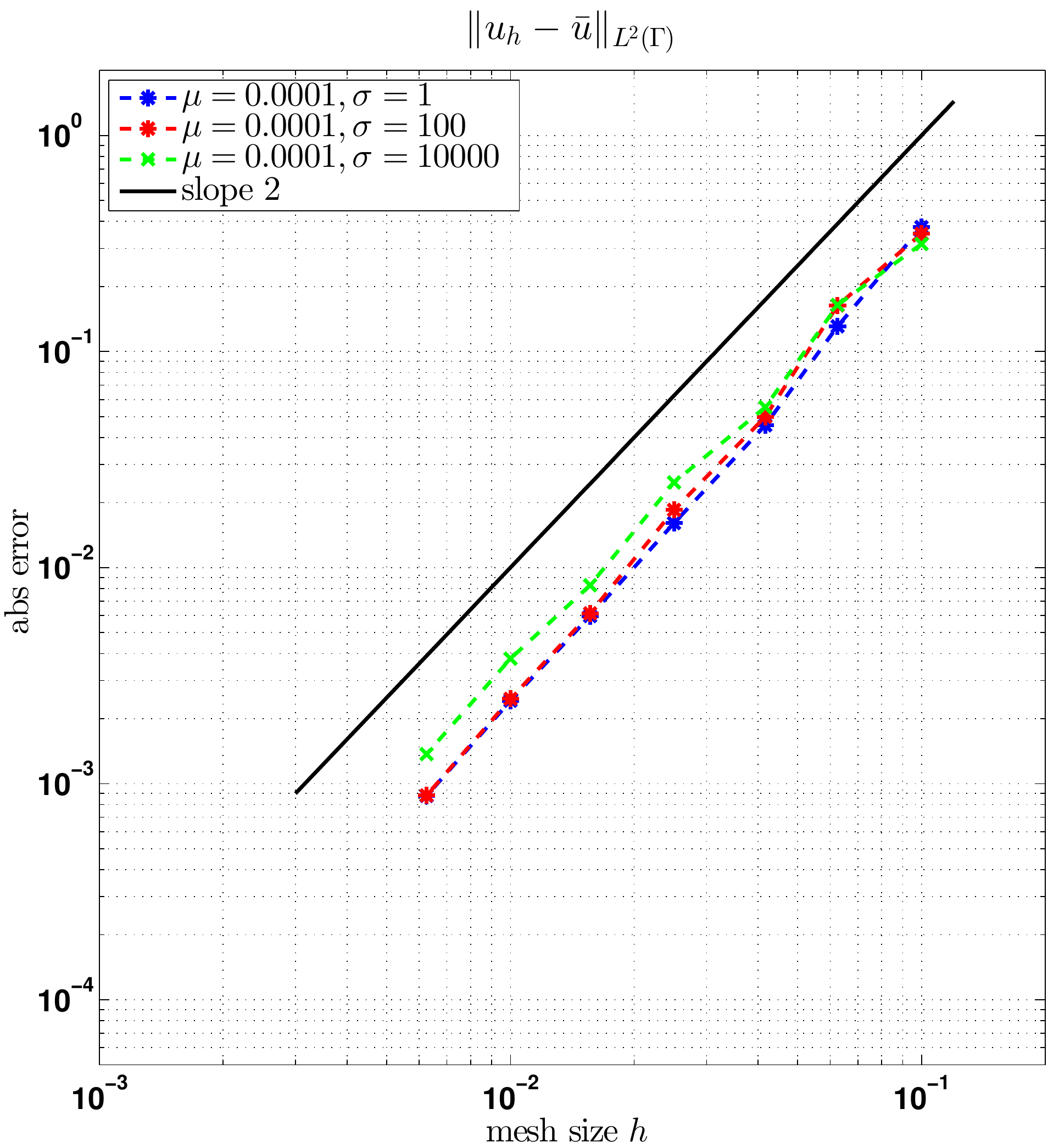}}
  \subfloat{\includegraphics[trim=0 0 0 0, clip, width=0.33\textwidth]{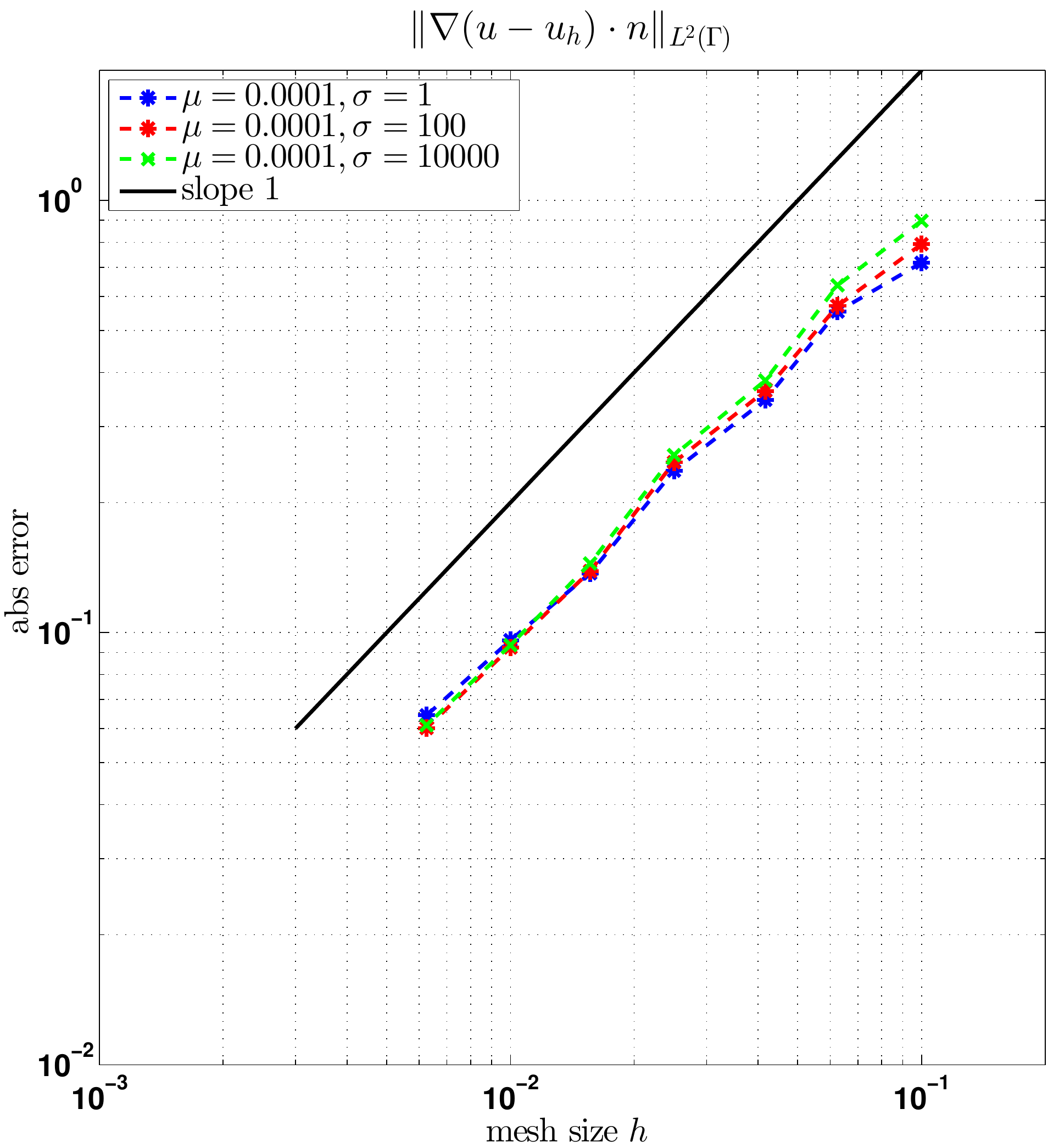}}
  \subfloat{\includegraphics[trim=0 0 0 0, clip, width=0.33\textwidth]{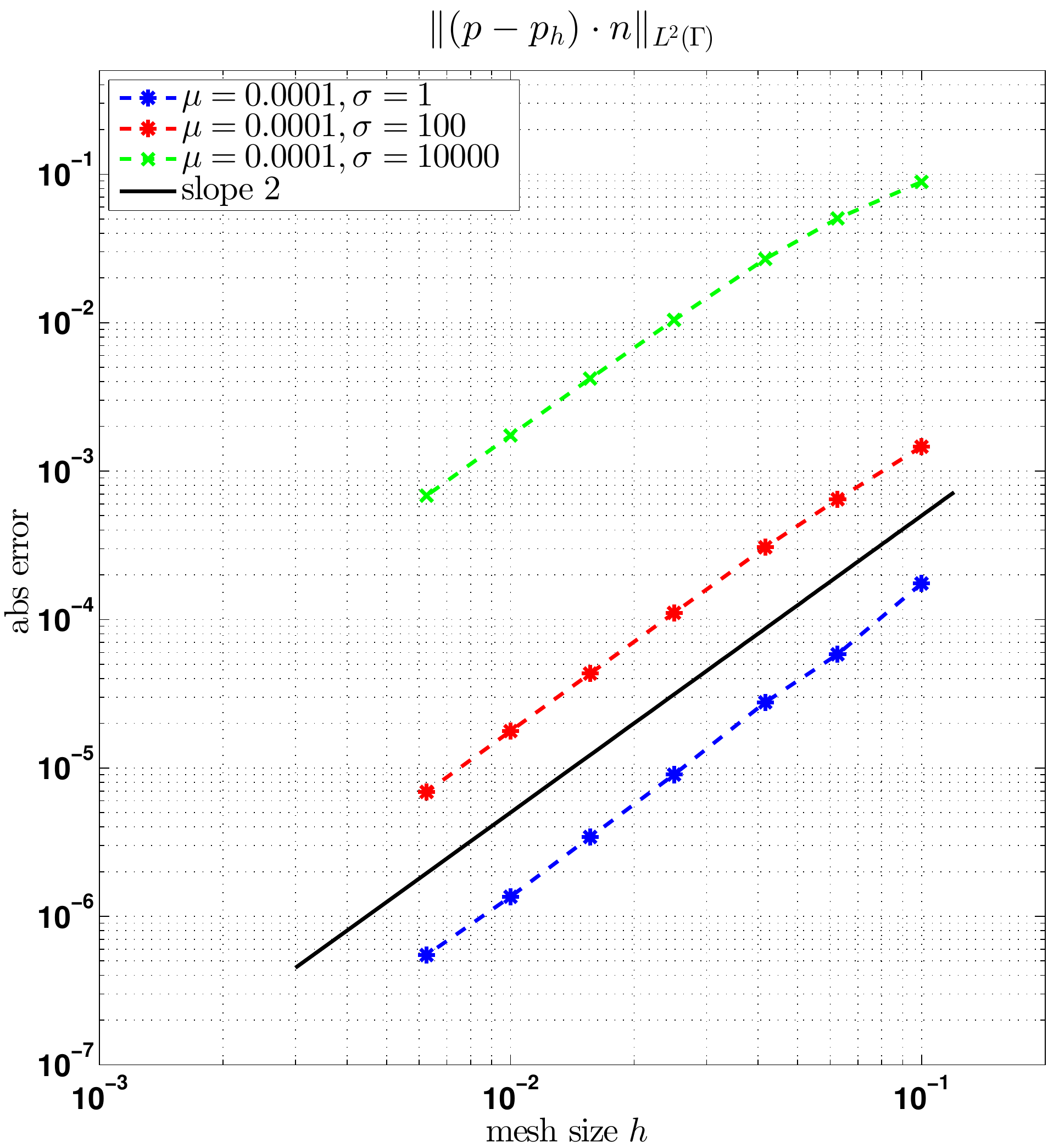}}
  \caption{High-Reynolds-number 2D Taylor Problem with $\mu=0.0001$ and $\mcP^1$ approximations: Convergence rates in $L^2$-norms for velocity, velocity gradient and pressure in the domain (top row) and on the boundary (bottom row).}
  \label{fig:kimmoin_oseen:spatial_convergence_conv}
\end{figure}

\begin{figure}[t]
  \centering
  \subfloat{\includegraphics[trim=0 0 0 0, clip, width=0.33\textwidth]{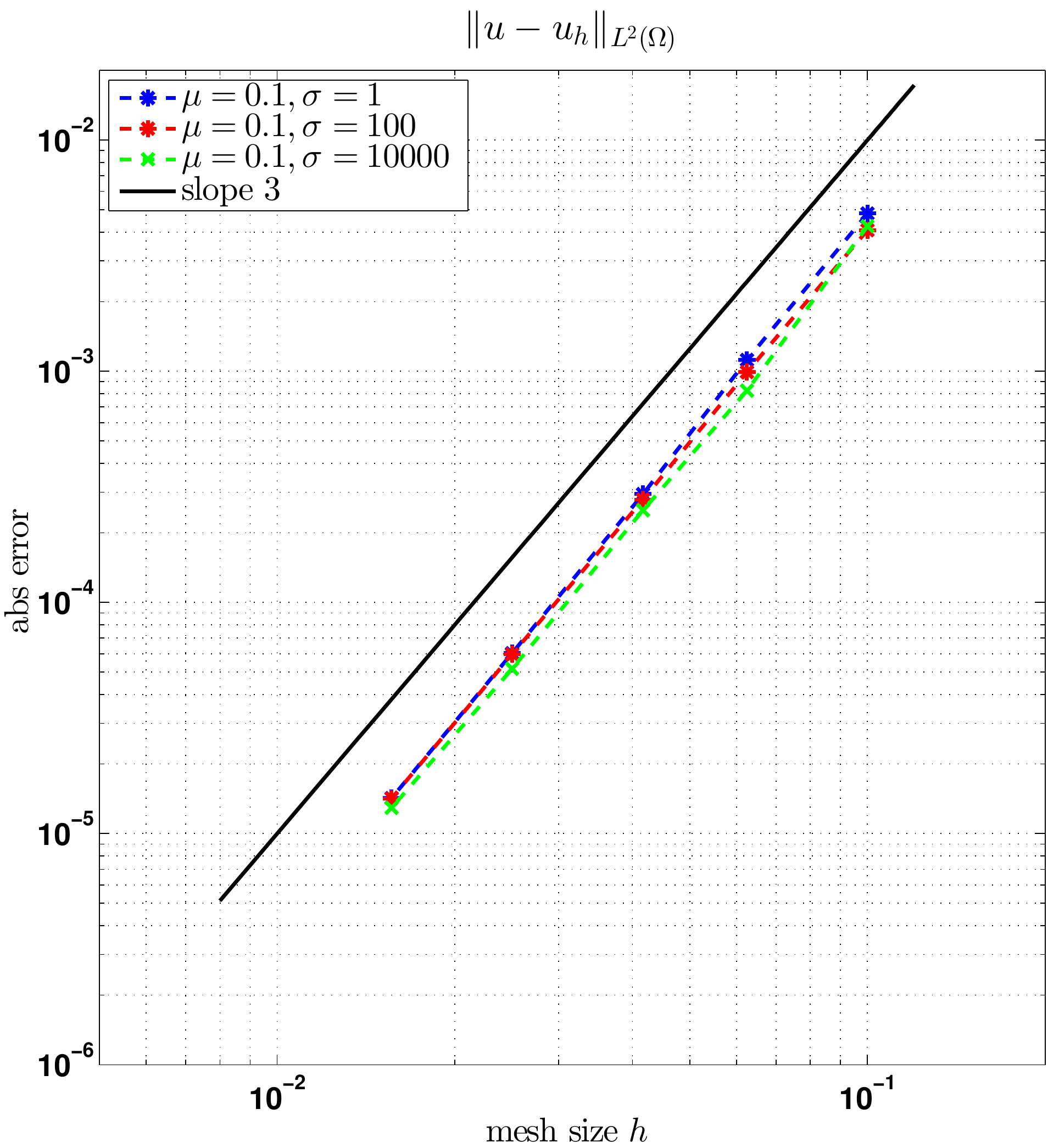}}
  \subfloat{\includegraphics[trim=0 0 0 0, clip, width=0.33\textwidth]{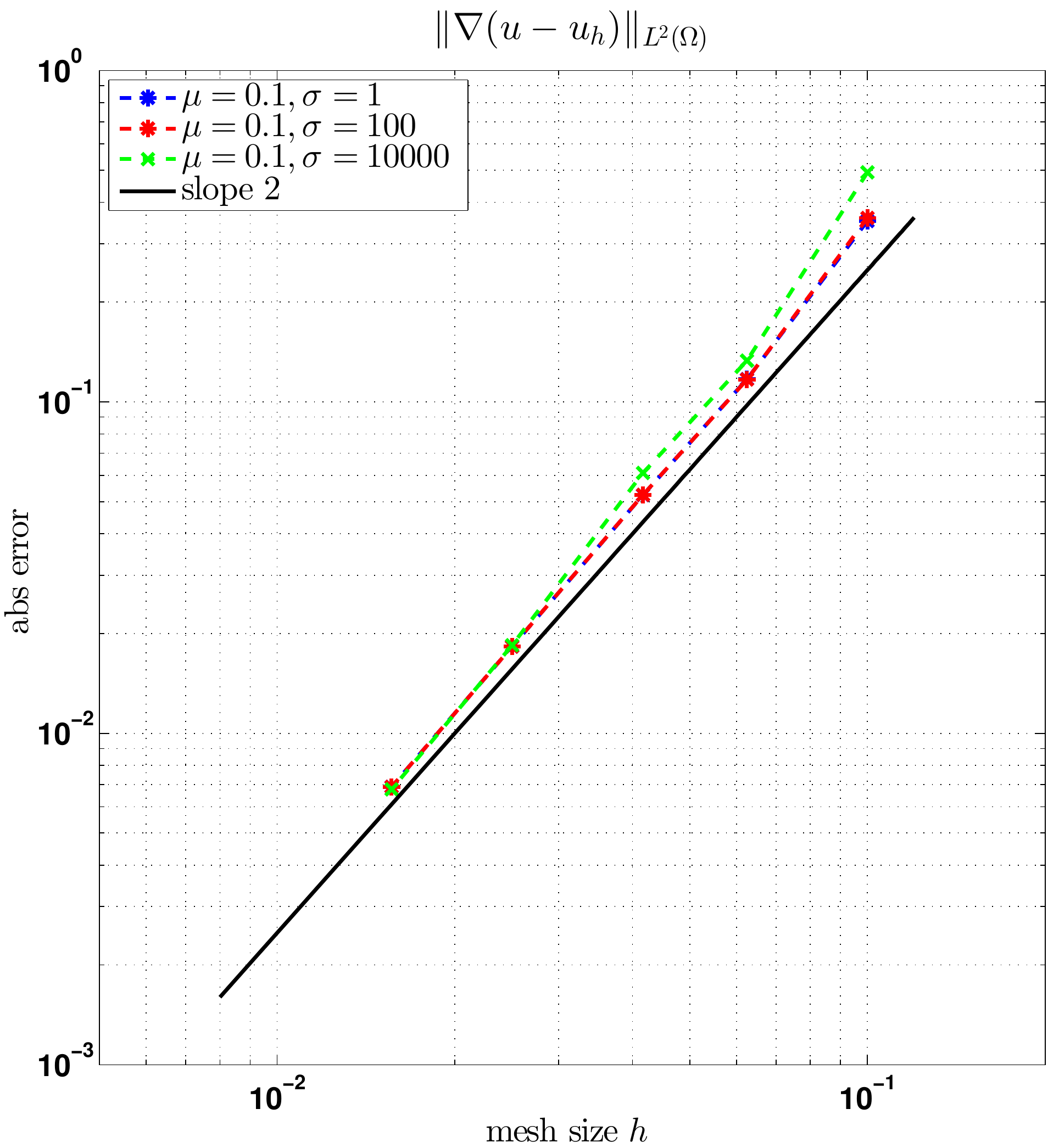}}
  \subfloat{\includegraphics[trim=0 0 0 0, clip, width=0.33\textwidth]{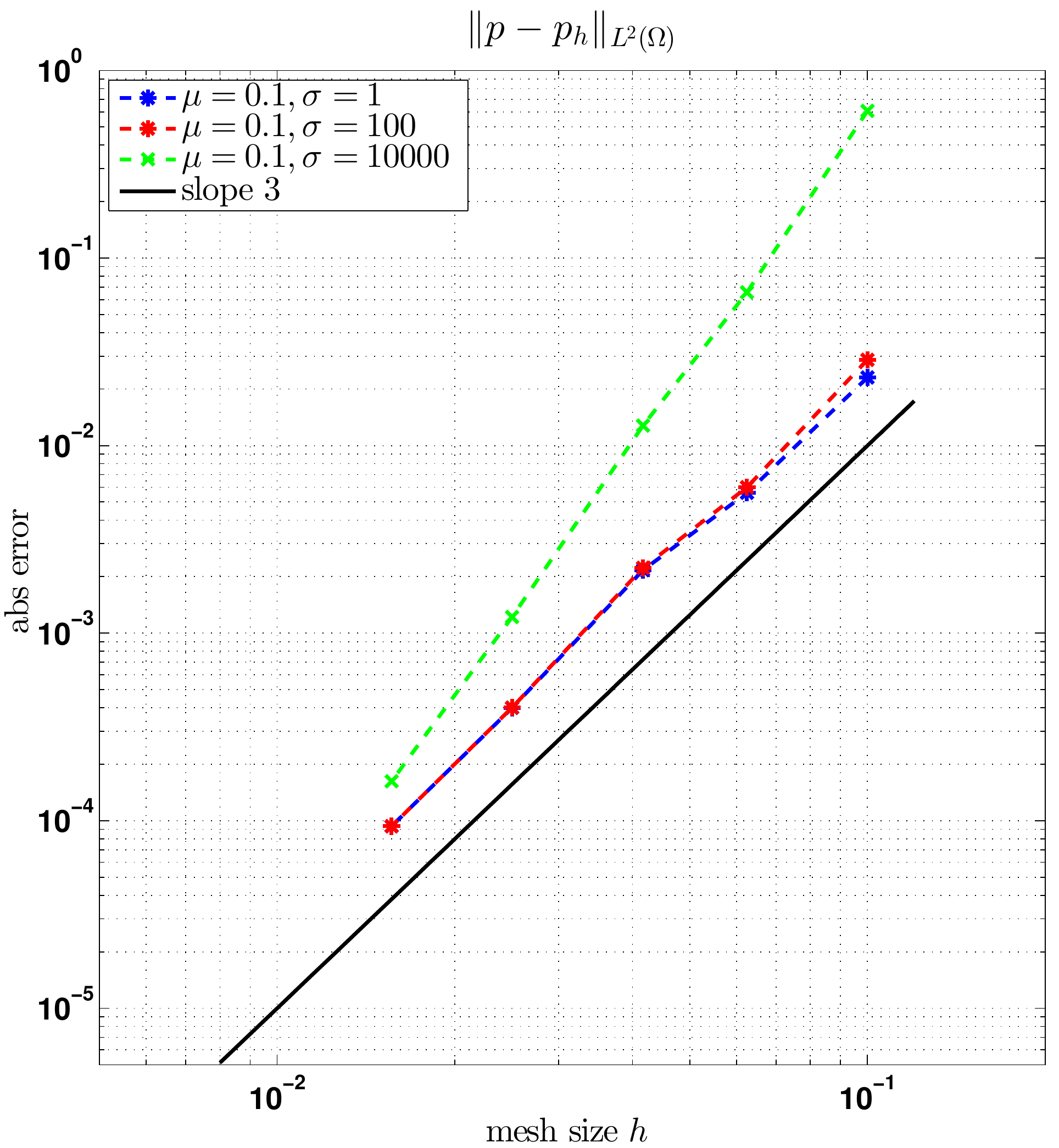}}\\
  \subfloat{\includegraphics[trim=0 0 0 0, clip, width=0.33\textwidth]{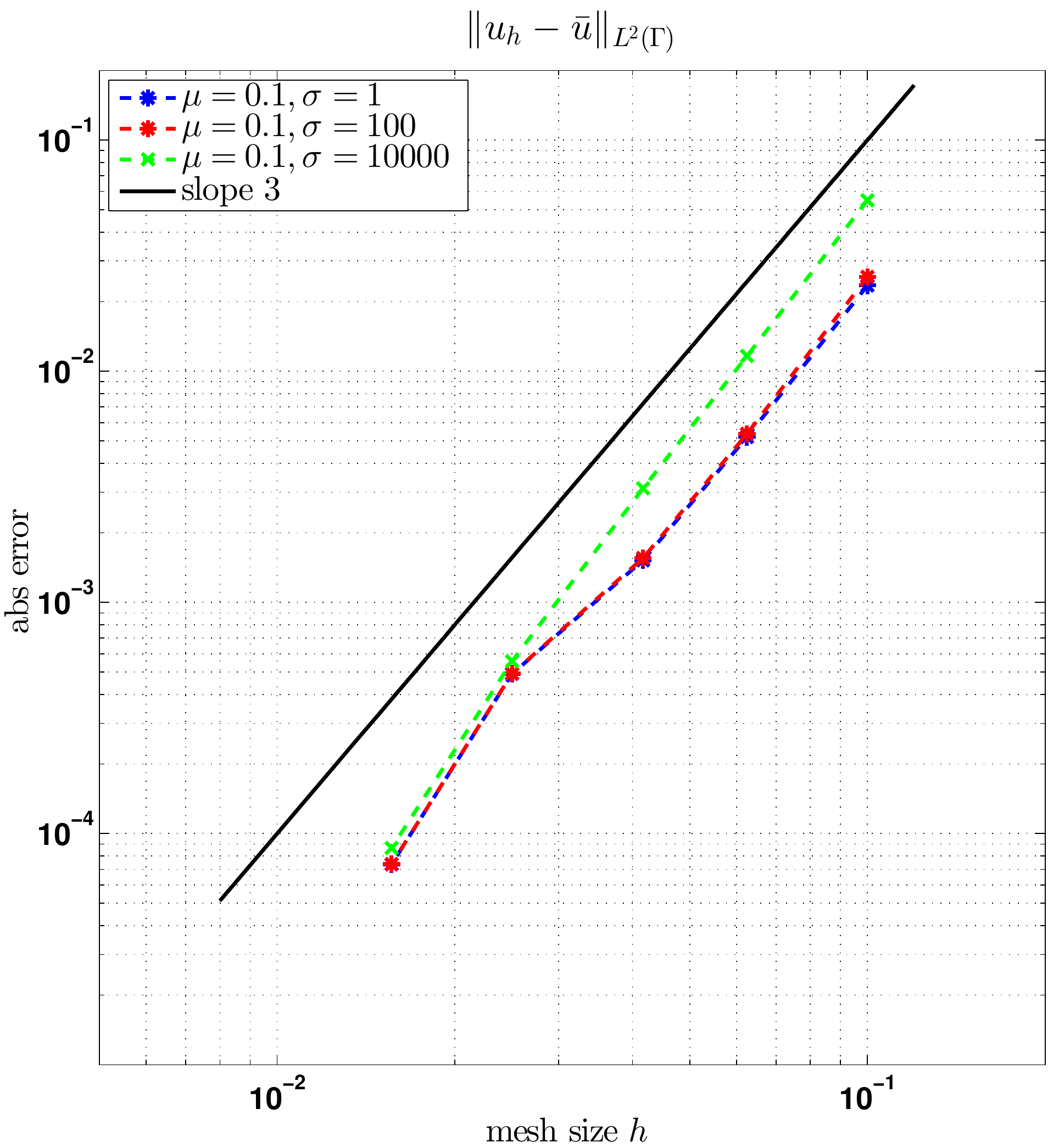}}
  \subfloat{\includegraphics[trim=0 0 0 0, clip, width=0.33\textwidth]{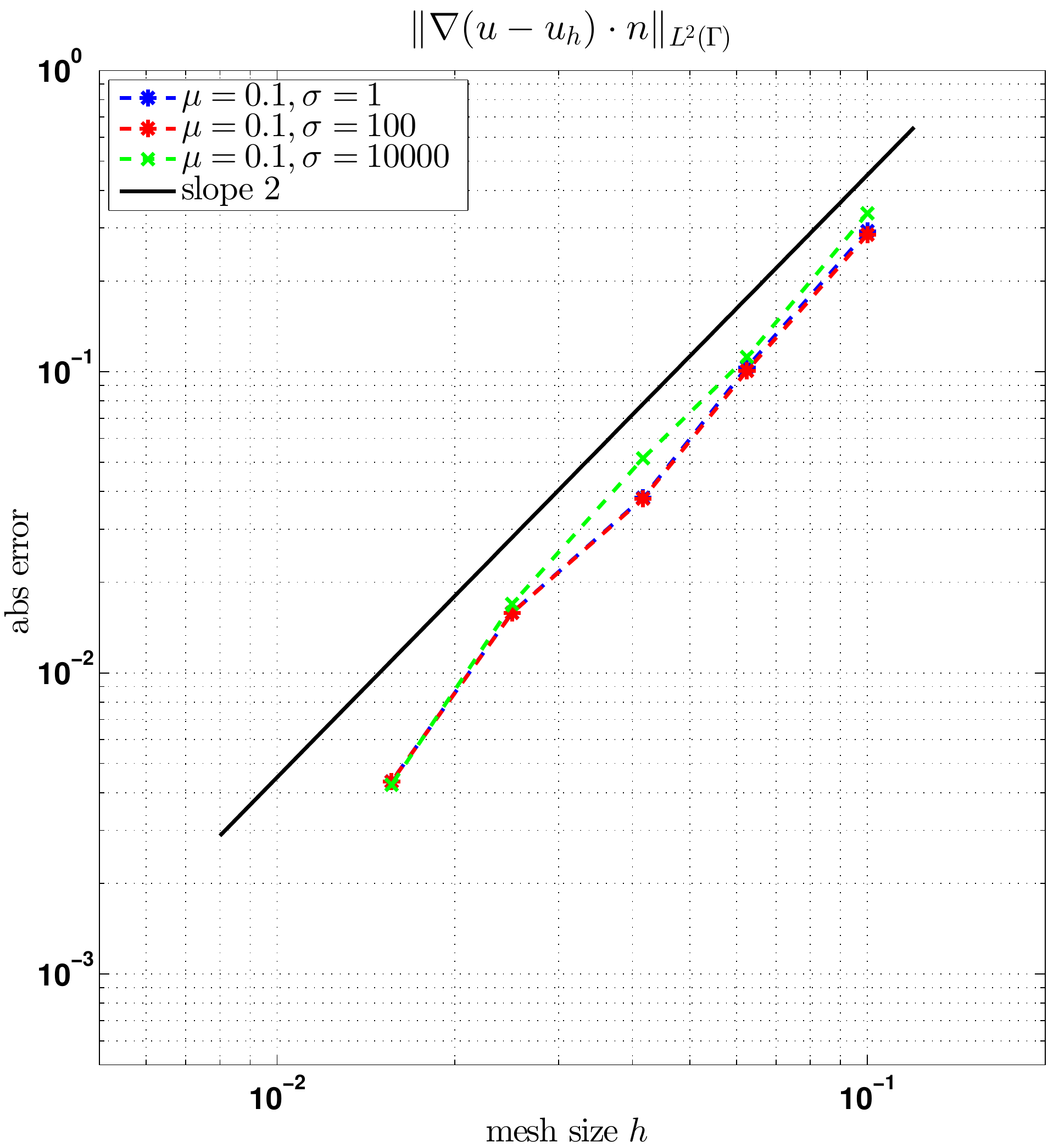}}
  \subfloat{\includegraphics[trim=0 0 0 0, clip, width=0.33\textwidth]{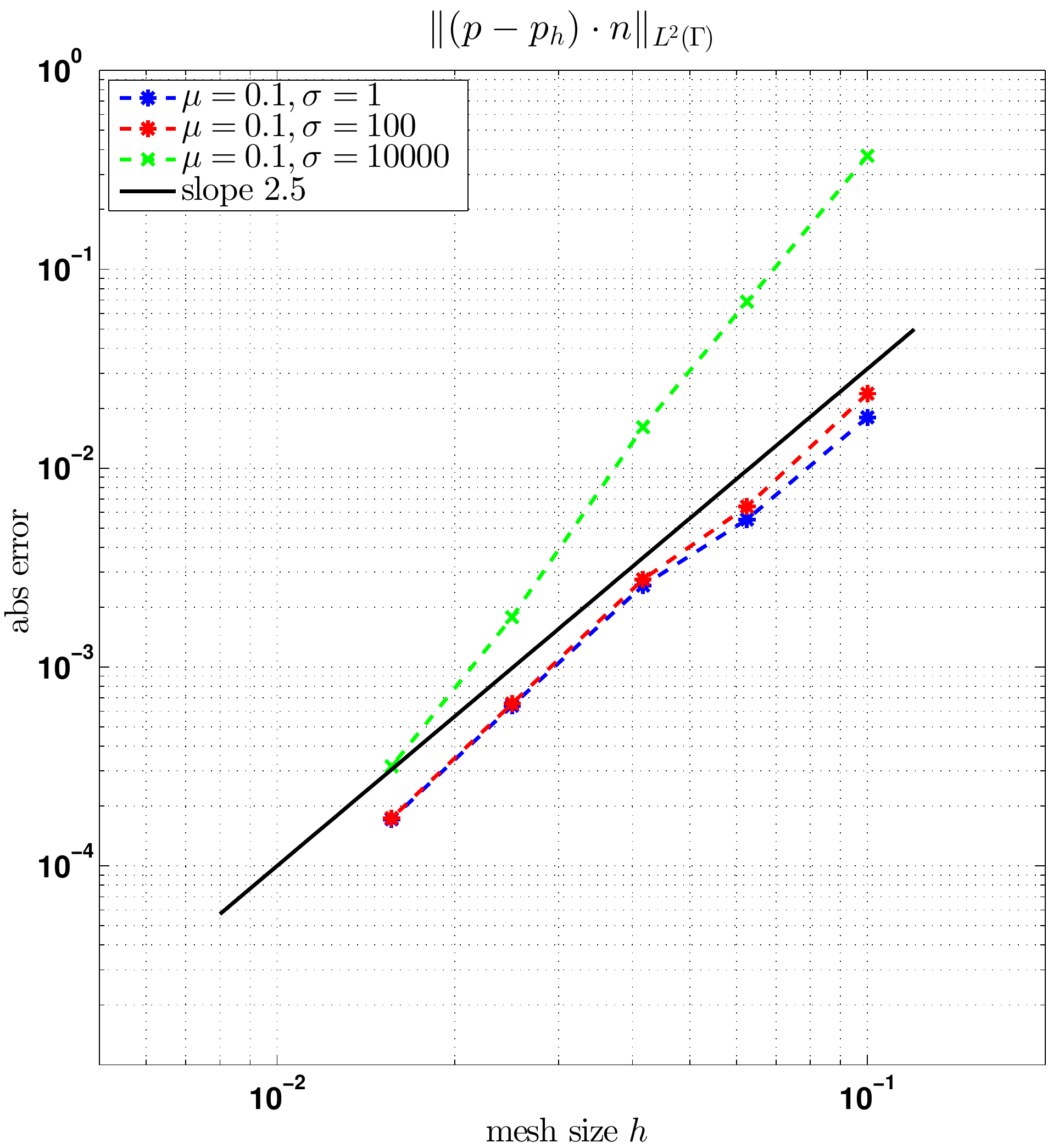}}
  \caption{
High-Reynolds-number 2D Taylor Problem with $\mu=0.0001$ and $\mcP^2$ approximations: Convergence rates in $L^2$-norms for velocity, velocity gradient and pressure in the domain (top row) and on the boundary (bottom row).
}
  \label{fig:kimmoin_oseen:spatial_convergence_conv_quadratic}
\end{figure}

\subsection{Convergence Study - 3D Beltrami Flow}
\label{ssec:numexamples:beltrami_problem}

To support the theoretical results also in three spatial dimensions, we consider the well-studied Beltrami-flow example, see, e.g, in \cite{EthierSteinman1994, BurmanFernandezHansbo2006}.
The steady Beltrami flow is analytically given as
\begin{align}
 u_{1}(x_1,x_2,x_3) &= b e^{a(x_1-x_3) + b(x_2-x_3)} - a e^{a(x_3-x_2) + b(x_1-x_2)} \label{eq:beltrami_function_u_x},\\
 u_{2}(x_1,x_2,x_3) &= b e^{a(x_2-x_1) + b(x_3-x_1)} - a e^{a(x_1-x_3) + b(x_2-x_3)} \label{eq:beltrami_function_u_y},\\
 u_{3}(x_1,x_2,x_3) &= b e^{a(x_3-x_2) + b(x_1-x_2)} - a e^{a(x_2-x_1) + b(x_3-x_1)} \label{eq:beltrami_function_u_z},\\
 p(x_1,x_2,x_3)     &= (a^2+b^2+ab) [ e^{a(x_1-x_2)+b(x_1-x_3)} + e^{a(x_2-x_3)+b(x_2-x_1)} + e^{a(x_3-x_1)+b(x_3-x_2)}]  \label{eq:beltrami_function_p}
\end{align}
with $a=b=\pi/4$.
The velocity field $\bfu$ is solenoidal by construction.
The right-hand side $\bff$ and the boundary data $\bfg$ are adapted to the Oseen problem \eqref{eq:oseen-problem-momentum}--\eqref{eq:oseen-problem-boundary} accordingly.

Numerical solutions are computed on a spherical fluid domain with radius $r=0.45$,
given implicitly as
\begin{equation}
  \Omega=\left\{ \bfx \in \mathbb{R}^3~|~\phi(x_1,x_2,x_3) = \sqrt{(x_1-1.0)^2+(x_2-0.5)^2+(x_3-0.5)^2}-0.45<0\right\},
\end{equation}
where its center is located at $(1.0,0.5,0.5)$.
The level-set field $\phi$ and the solutions are approximated on the active parts of a family of background meshes $\widehat{\mcT}_h$ covering a background cube $[0.5,1.5]\times[0,1]^2$ .
Thereby, each mesh is constructed of $N^3$~cubes, where each cube is subdivided into six tetrahedra $\mcP^1(T)$.
Here, $N$ denotes the number of cubes in each coordinate direction and $h=1/N$ is the short length of each tetrahedra~$T$.
Similar to the two-dimensional setting, a low- and a high-Reynolds-number setting is considered,
characterized by two different viscosities $\mu = 0.1$ and $\mu=0.0001$.
Computed velocity and pressure approximation errors $(\bfu_h-\bfu)$ and $(p_h-p)$ are shown in~\Figref{fig:beltrami_oseen:spatial_convergence_visc}
for the low-Reynolds-number case and in \Figref{fig:beltrami_oseen:spatial_convergence_conv} for the high-Reynolds-number case.
The same optimal rates for velocity errors as well as super-convergence for the pressure can be observed similar to the two-dimensional example.
\begin{figure}[ht]
  \centering
  \subfloat{\includegraphics[trim=0 0 0 0, clip, width=0.33\textwidth]{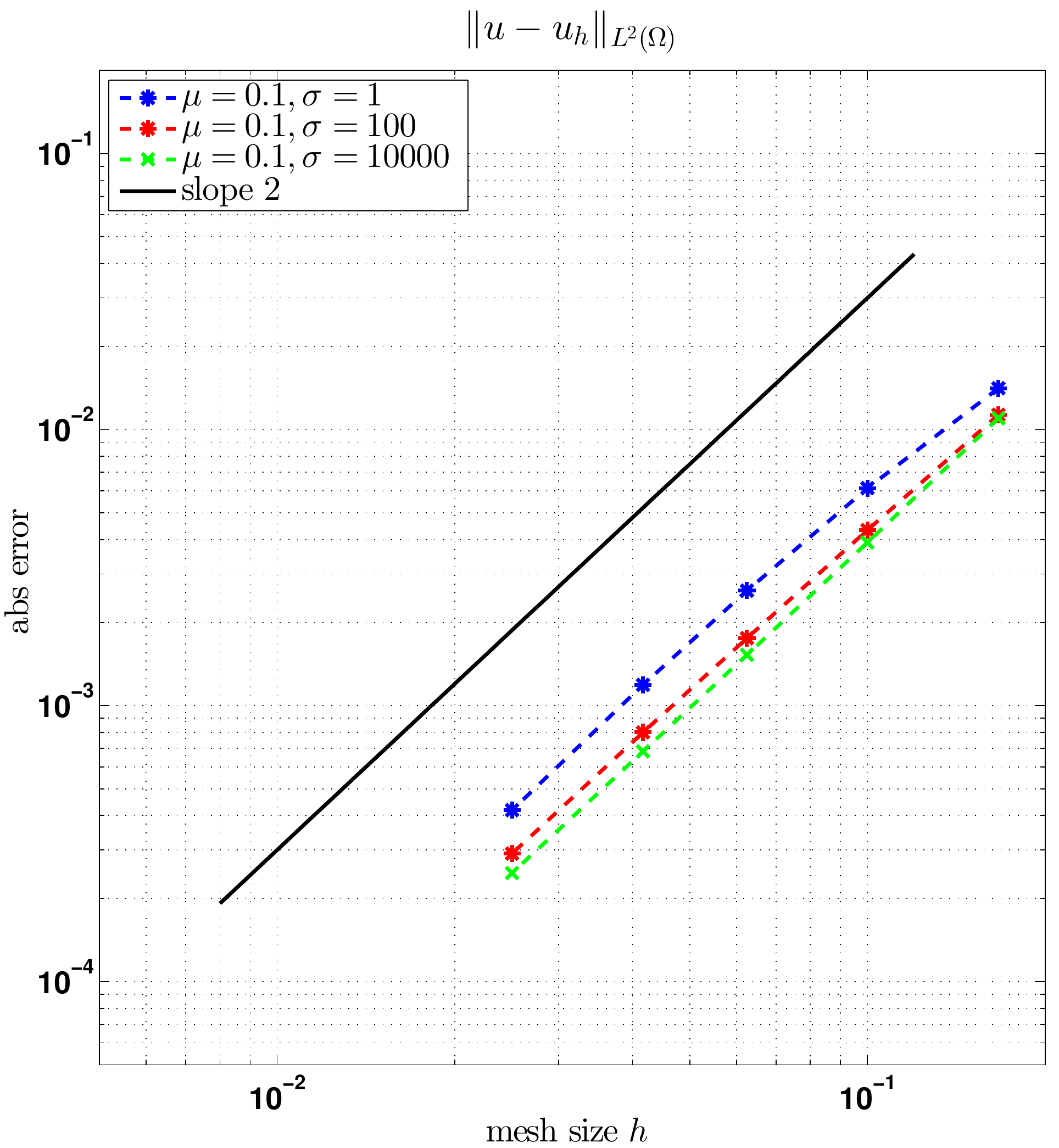}}
  \subfloat{\includegraphics[trim=0 0 0 0, clip, width=0.33\textwidth]{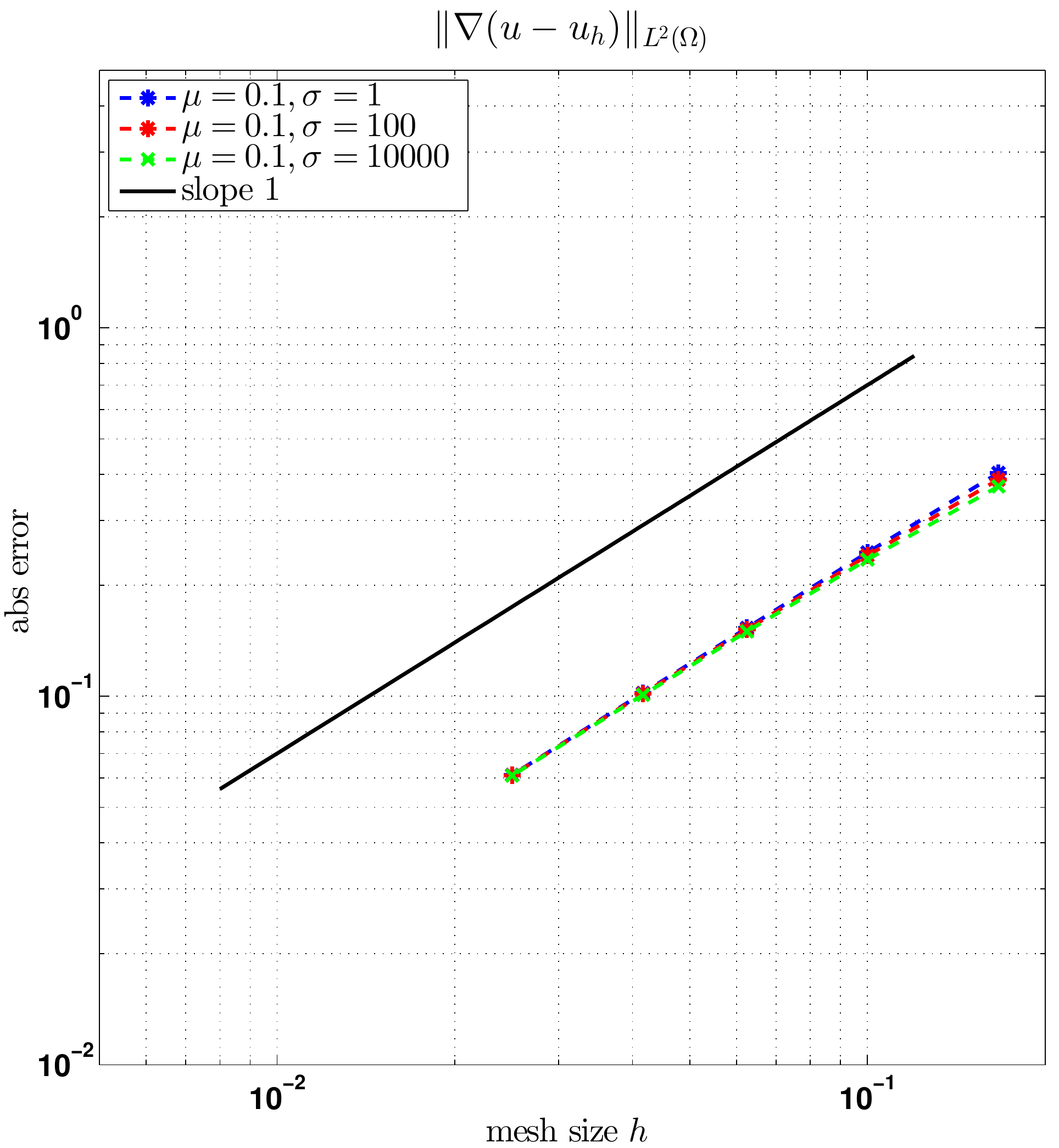}}
  \subfloat{\includegraphics[trim=0 0 0 0, clip, width=0.33\textwidth]{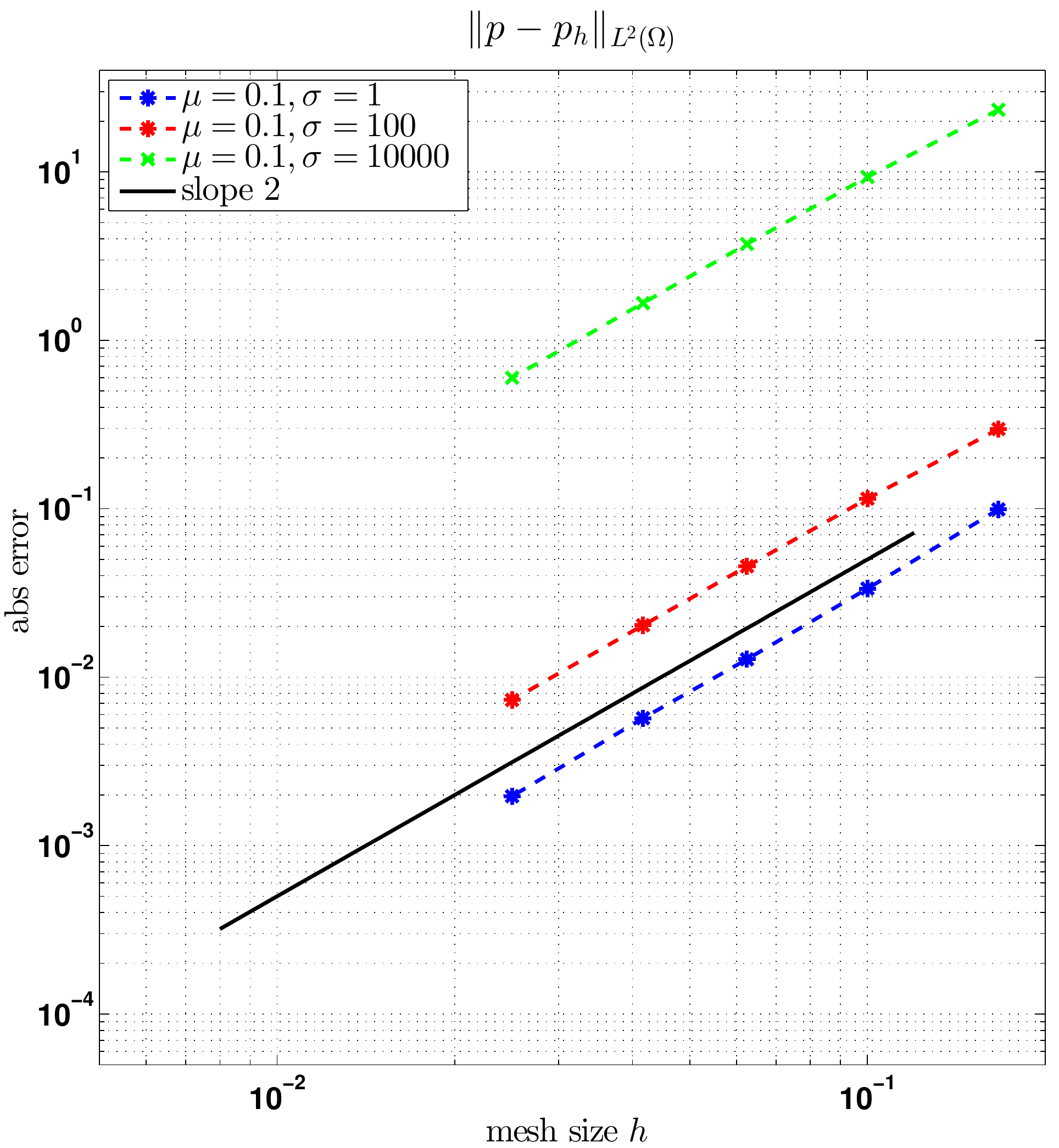}}\\
  \subfloat{\includegraphics[trim=0 0 0 0, clip, width=0.33\textwidth]{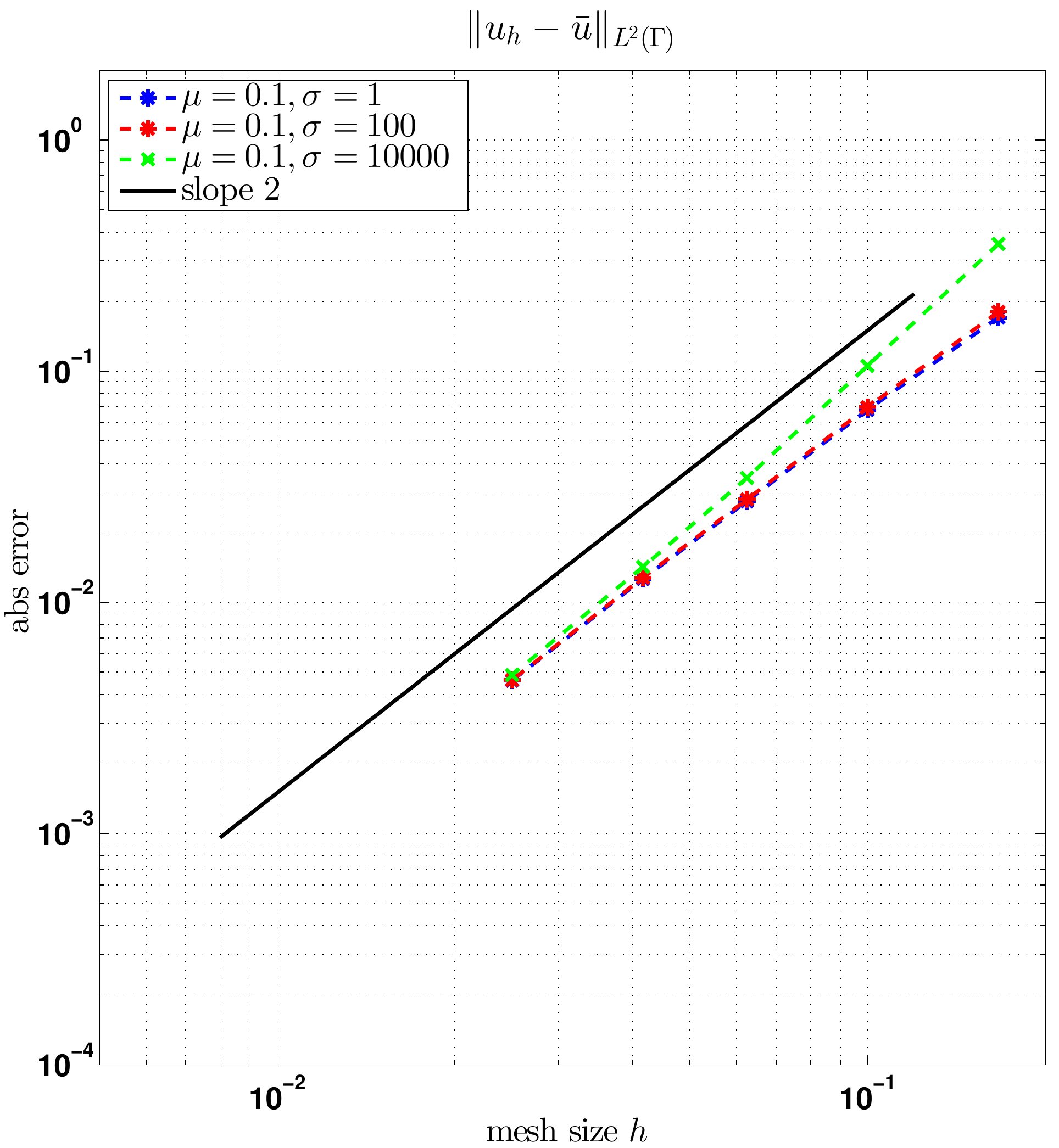}}
  \subfloat{\includegraphics[trim=0 0 0 0, clip, width=0.33\textwidth]{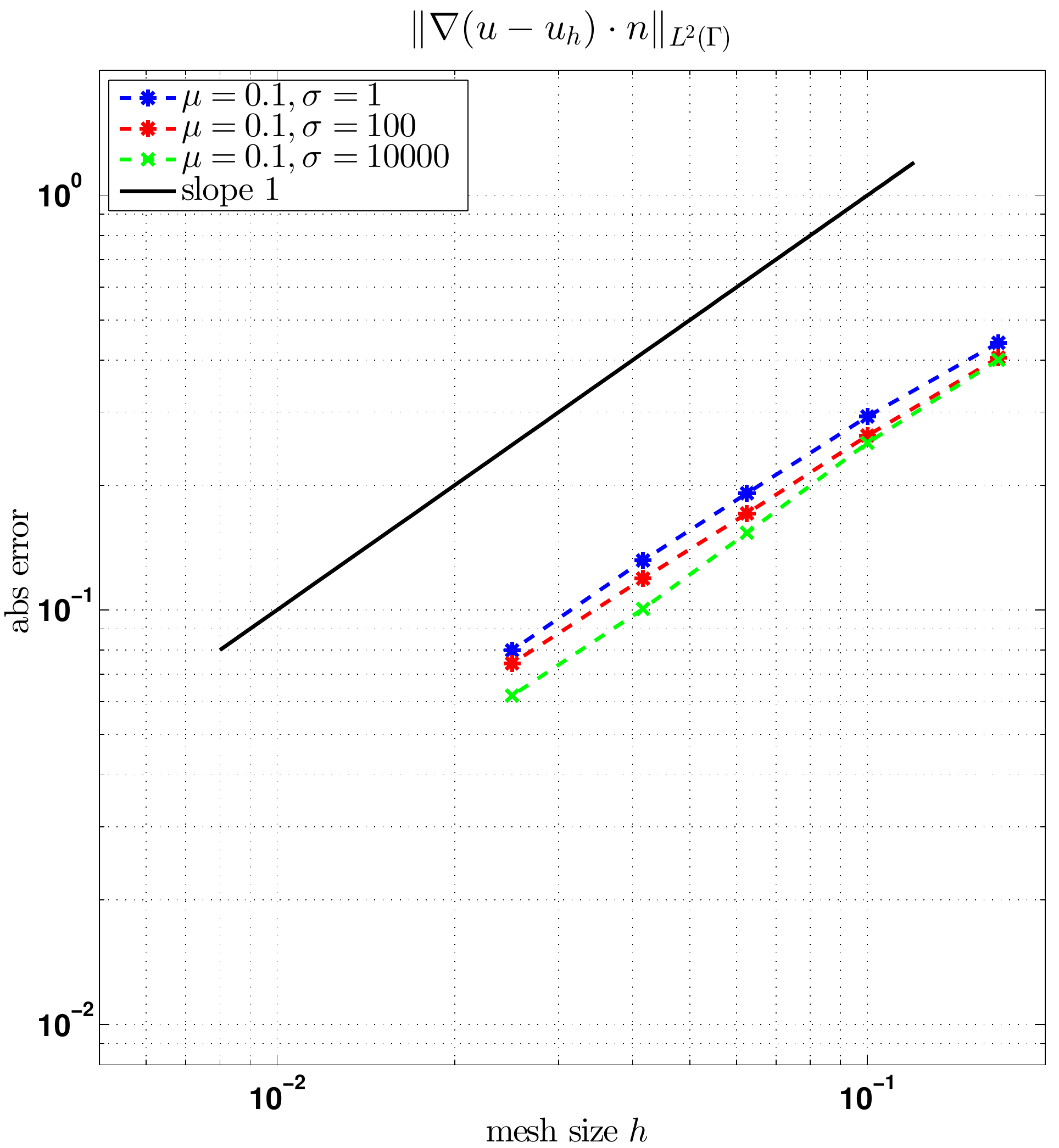}}
  \subfloat{\includegraphics[trim=0 0 0 0, clip, width=0.33\textwidth]{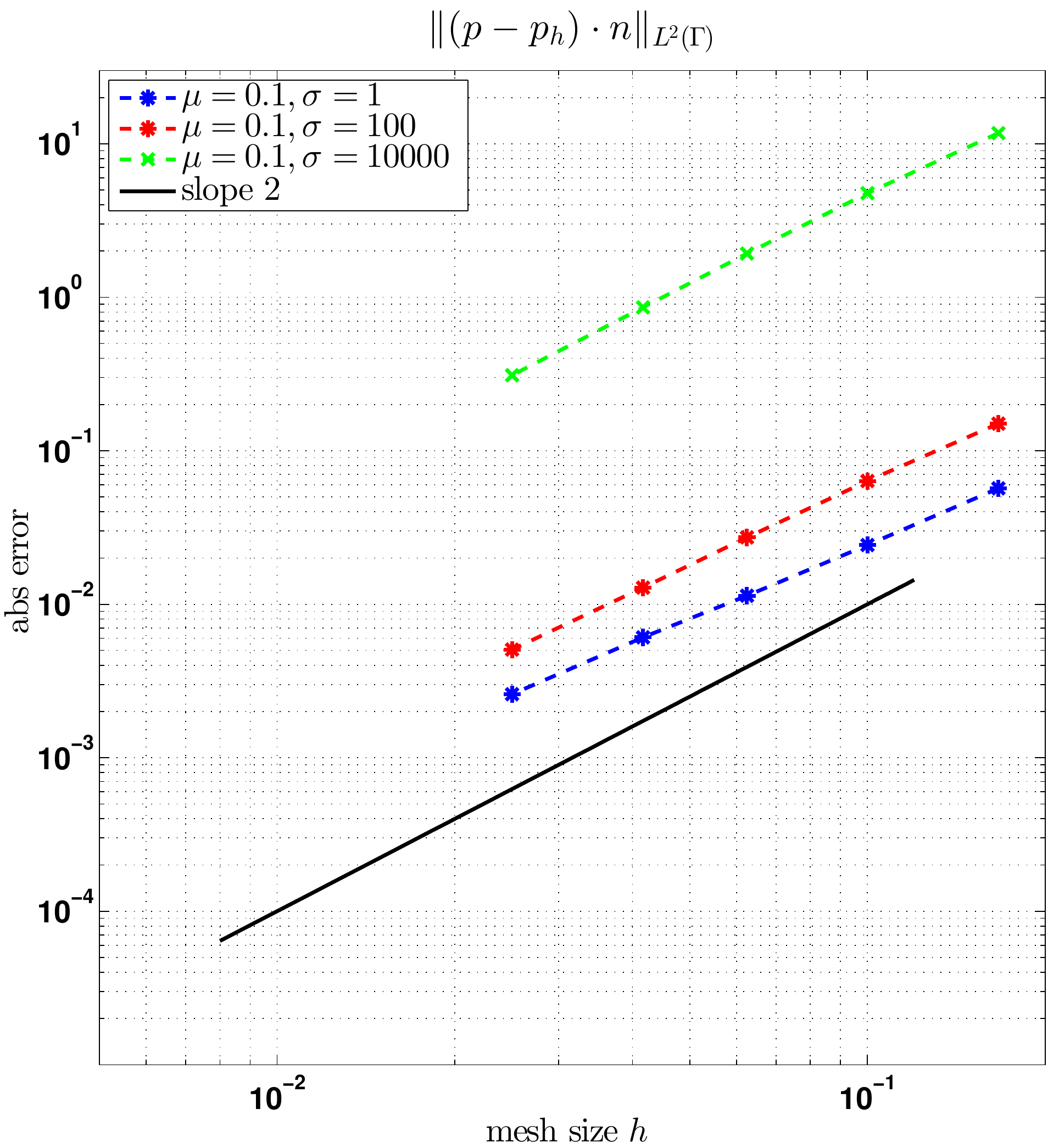}}
  \caption{Low-Reynolds-number 3D Beltrami Flow with $\mu=0.1$: Convergence rates in $L^2$-norms for velocity, velocity gradient and pressure in the domain (top row) and on the boundary (bottom row).}
  \label{fig:beltrami_oseen:spatial_convergence_visc}
\end{figure}

\begin{figure}[ht]
  \centering
  \subfloat{\includegraphics[trim=0 0 0 0, clip, width=0.33\textwidth]{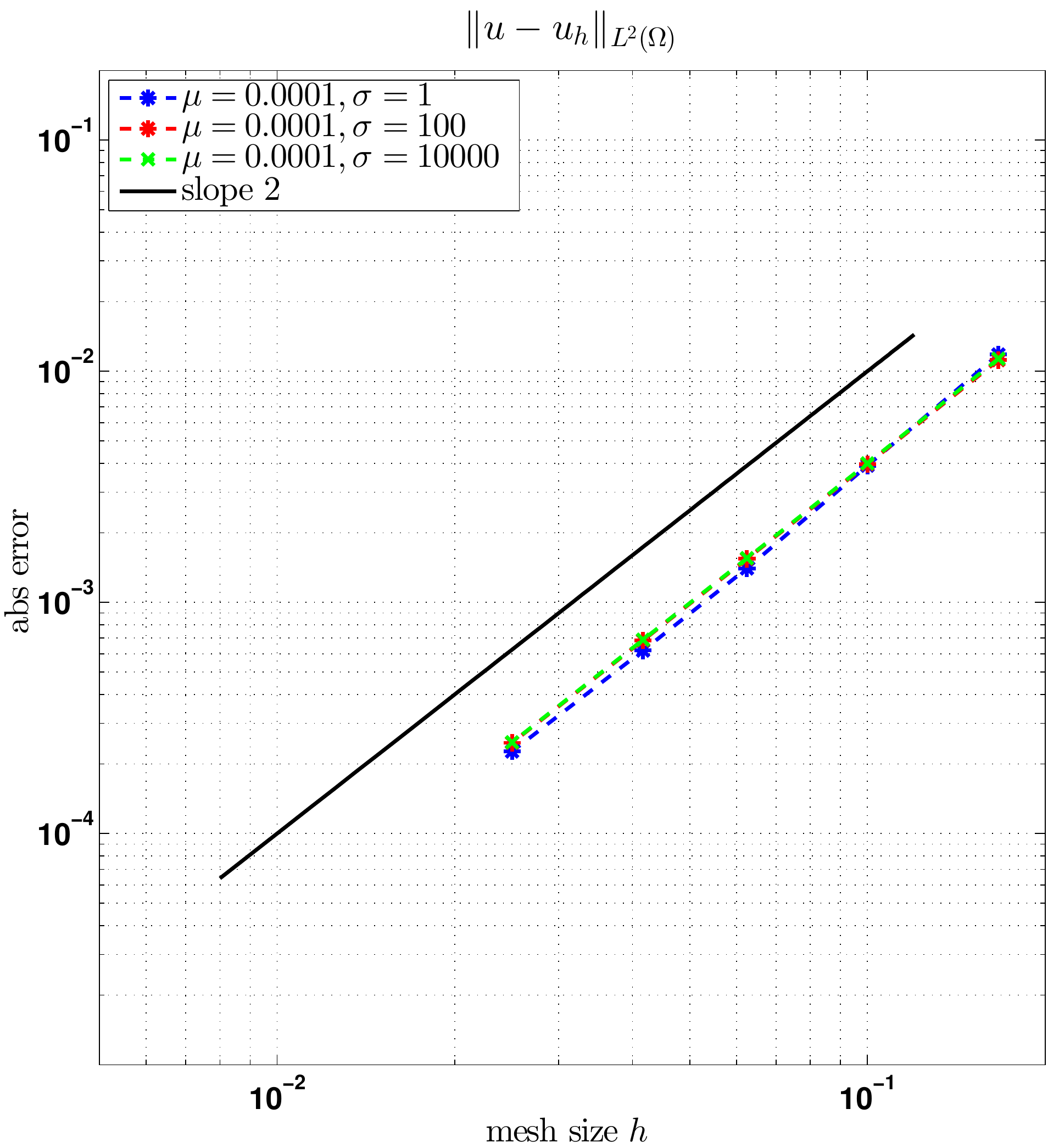}}
  \subfloat{\includegraphics[trim=0 0 0 0, clip, width=0.33\textwidth]{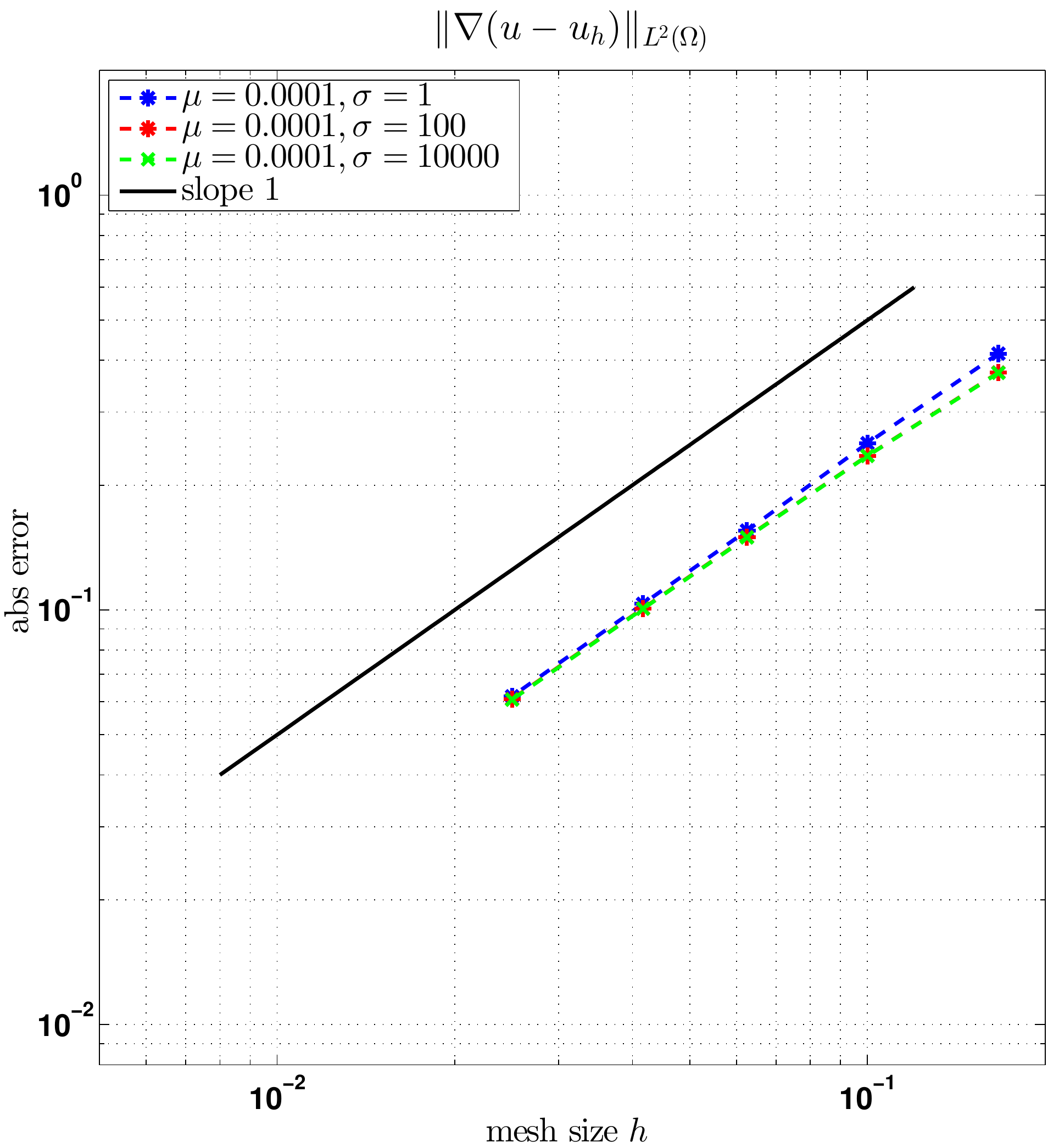}}
  \subfloat{\includegraphics[trim=0 0 0 0, clip, width=0.33\textwidth]{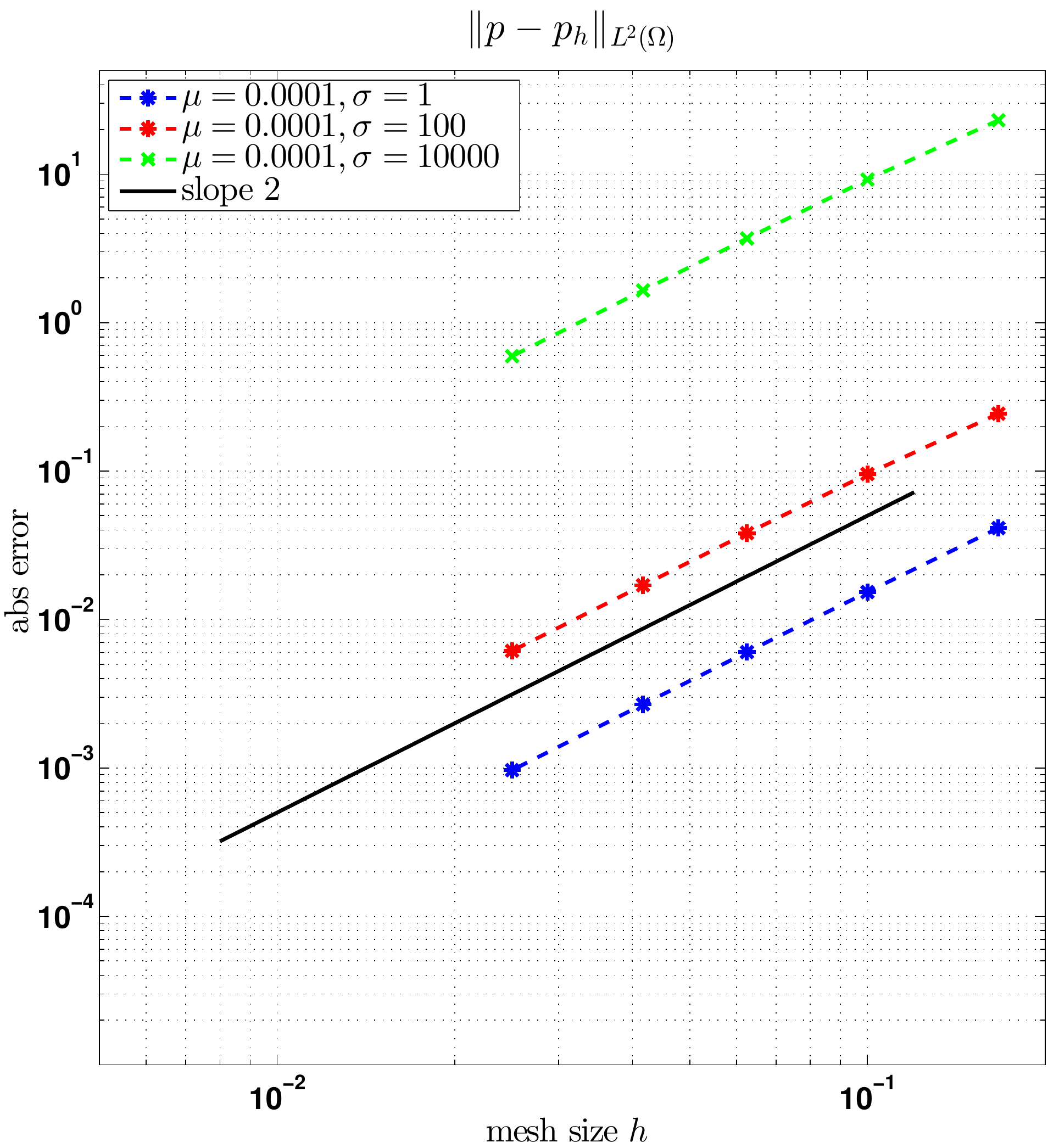}}\\
  \subfloat{\includegraphics[trim=0 0 0 0, clip, width=0.33\textwidth]{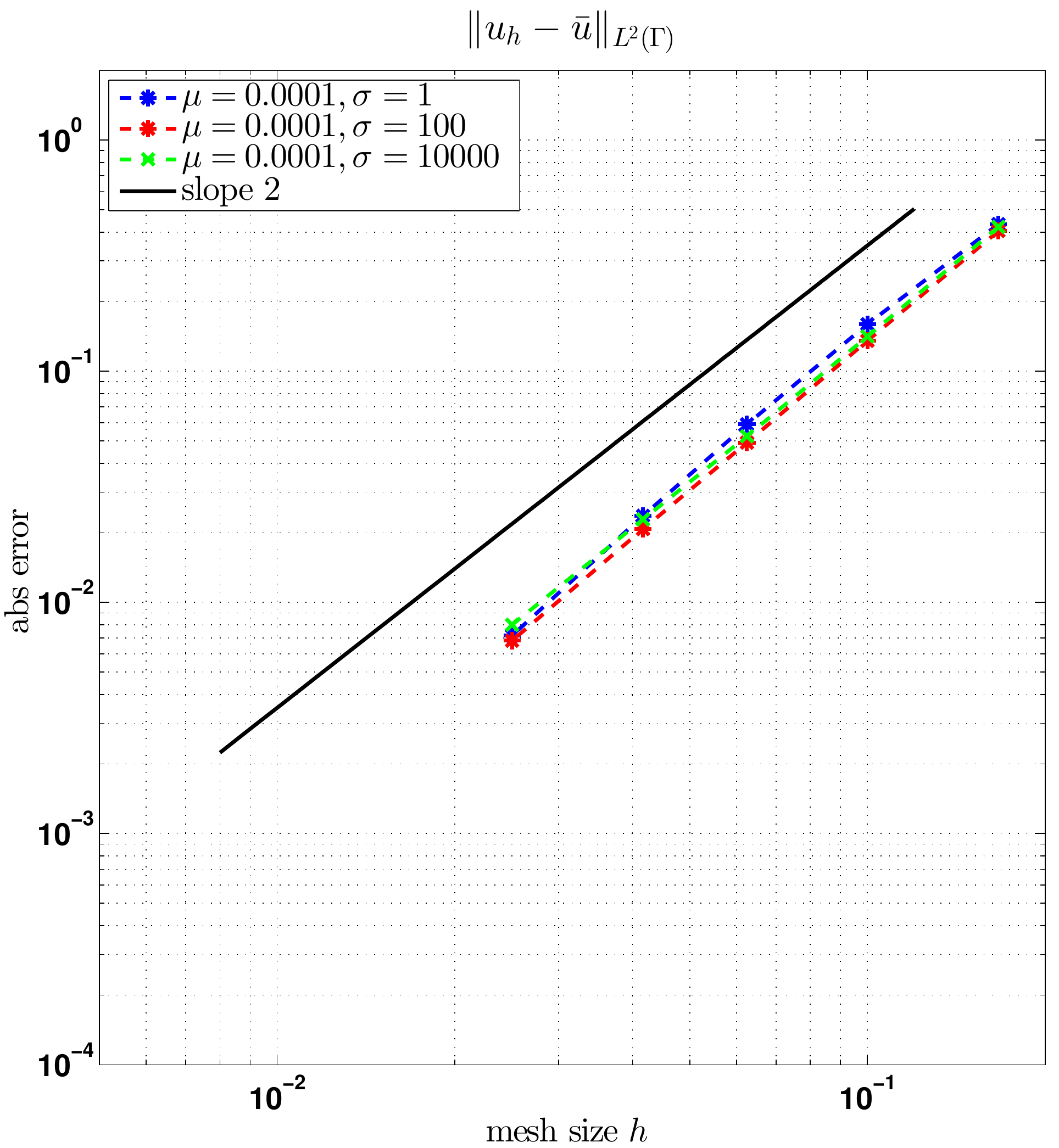}}
  \subfloat{\includegraphics[trim=0 0 0 0, clip, width=0.33\textwidth]{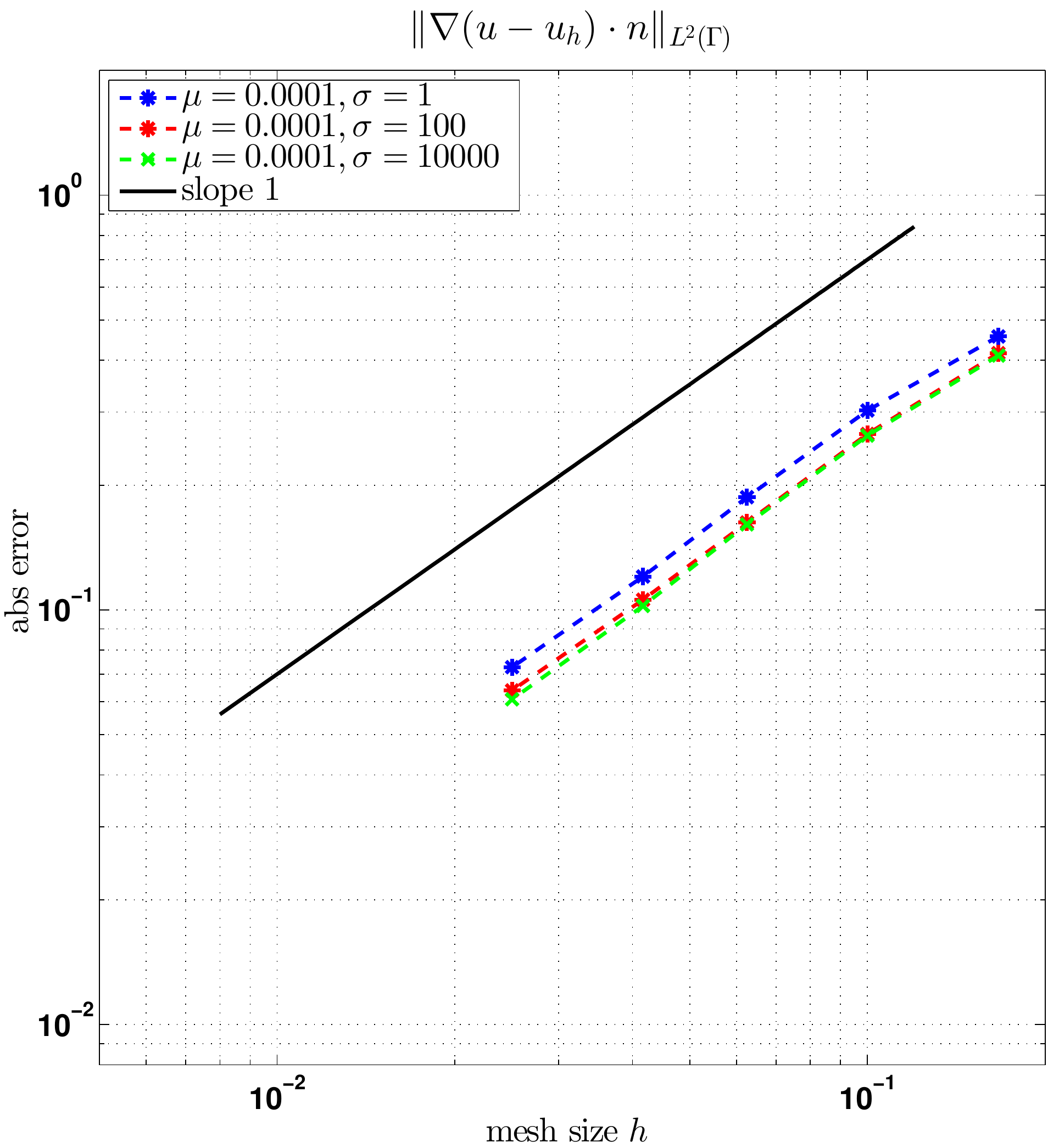}}
  \subfloat{\includegraphics[trim=0 0 0 0, clip, width=0.33\textwidth]{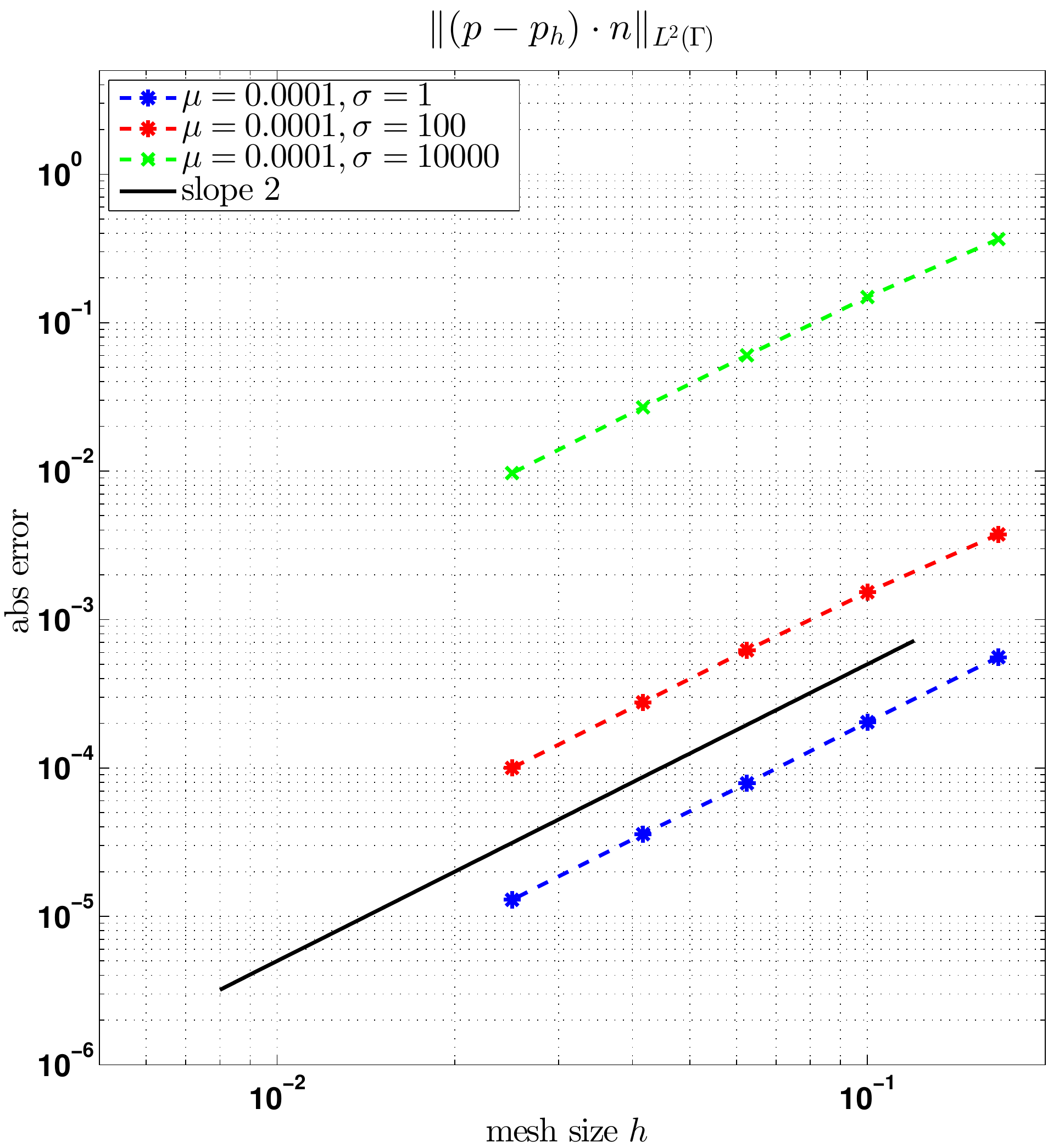}}
  \caption{High-Reynolds-number 3D Beltrami Flow with $\mu=0.0001$: Convergence rates in $L^2$-norms for velocity, velocity gradient and pressure in the domain (top row) and on the boundary (bottom row).}
  \label{fig:beltrami_oseen:spatial_convergence_conv}
\end{figure}

To underline the stability of the velocity and pressure solutions for the high-Reynolds-number setting,
in \Figref{fig:beltrami_oseen:solution_cross_section}, velocity streamlines and the pressure solution are visualized
along cross-sections computed on a coarse non-boundary-fitted mesh.
It is clearly visible that the solutions do not exhibit any oscillatory behavior, neither in the interior of the domain, nor at the boundary.
This is due to the different proposed continuous interior penalty and ghost-penalty stabilizations.
This fact underlines the stability of our proposed formulation even though the solution exhibits highly varying local element Reynolds numbers in the computational domain.
\begin{figure}
  \centering
\includegraphics[trim=0 0 0 50, clip, width=0.48\textwidth]{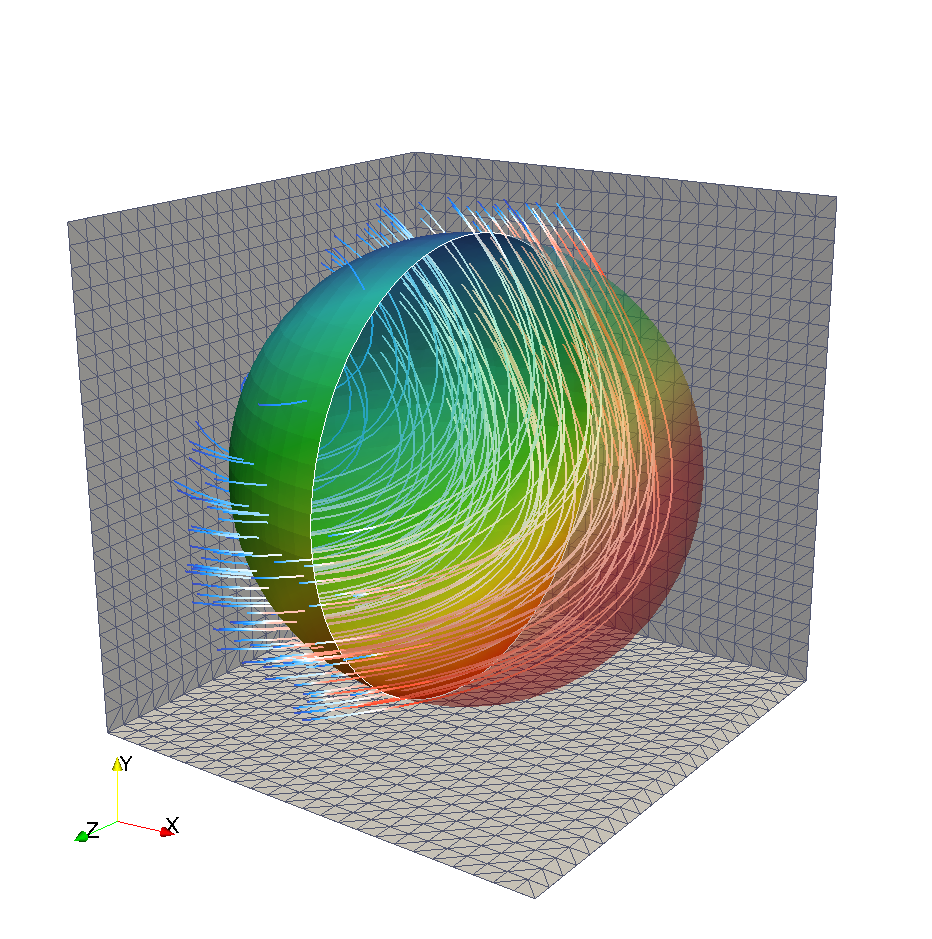}
\quad
\includegraphics[trim=0 0 0 50, clip, width=0.48\textwidth]{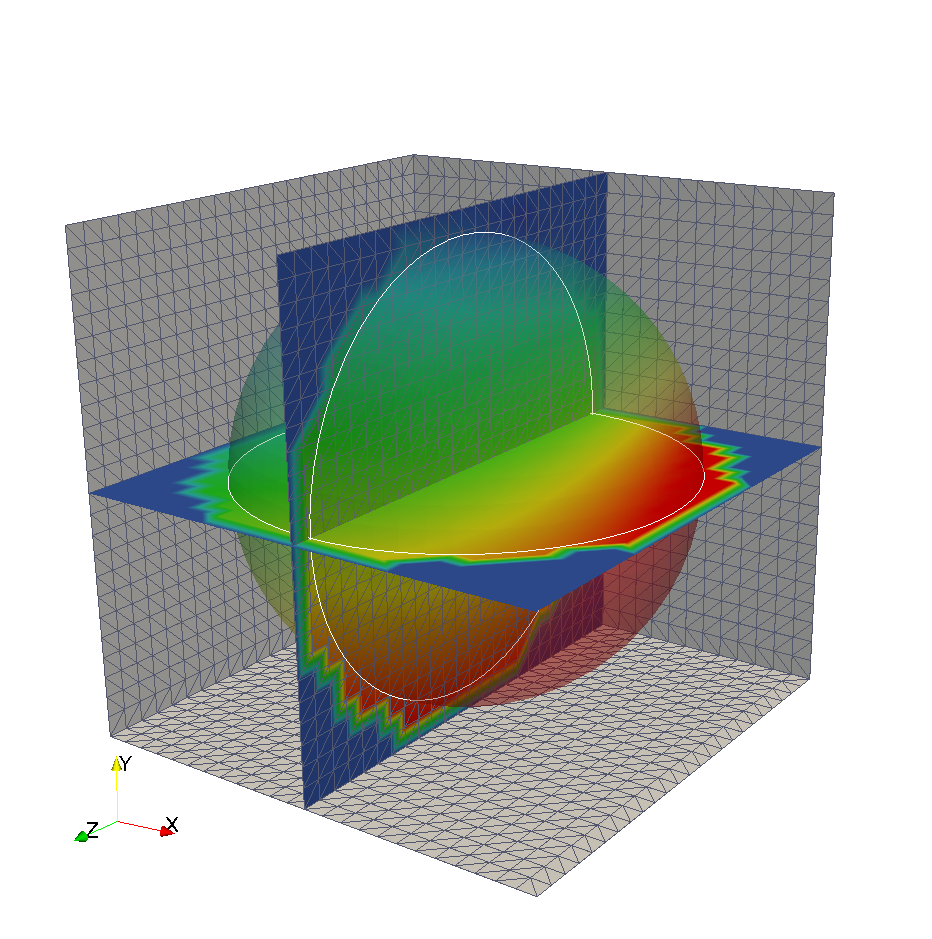}
  \caption{High-Reynolds-number 3D Beltrami flow: Computed velocity and pressure solutions on unfitted mesh with~$h=1/24$.
Stable solutions in the interior of the domain and in the boundary zone due to sufficient control ensured by different CIP and GP stabilization terms.
Left:~Streamlines colored by velocity magnitude and half-sphere ($x_1<1.0$) colored by pressure distribution. Right:~Pressure solution at cross-sections defined by $x_1=1.0$ and $x_2=0.5$.}
  \label{fig:beltrami_oseen:solution_cross_section}
\end{figure}

\subsection{Flow through a Helical Pipe}
\label{ssec:numexamples:pipe}

In the final numerical experiment, 
we demonstrate the applicability of the proposed cut finite element
formulation to solve the full time-dependent incompressible
Navier-Stokes equations in a complicated, implicitly described domain.
Temporal discretization is based on a one-step-$\theta$ scheme and the non-linear convective term
is approximated by fixed-point iterations. The resulting series of linear Oseen type systems are solved
by the proposed cut finite element method.

In this example, we consider the incompressible flow through a helical pipe.
The relevance of curved pipe flows ranges from basic industrial applications like chemical reactors, heat exchangers and pipelines to
medical applications considering physiological flows in the human body.
Helical pipe flows have been extensively studied in literature, see, e.g., in~\cite{Wang1981,Germano1982, Zabielski1998}.

The geometric setup of the pipe considered in this work is depicted in~\Figref{fig:helical_pipe} and is described as follows:
\begin{figure}[ht!]
  \centering
  \subfloat[\label{fig:helical_pipe:a}]{\includegraphics[trim=200 0 400 0, clip, width=0.38\textwidth]{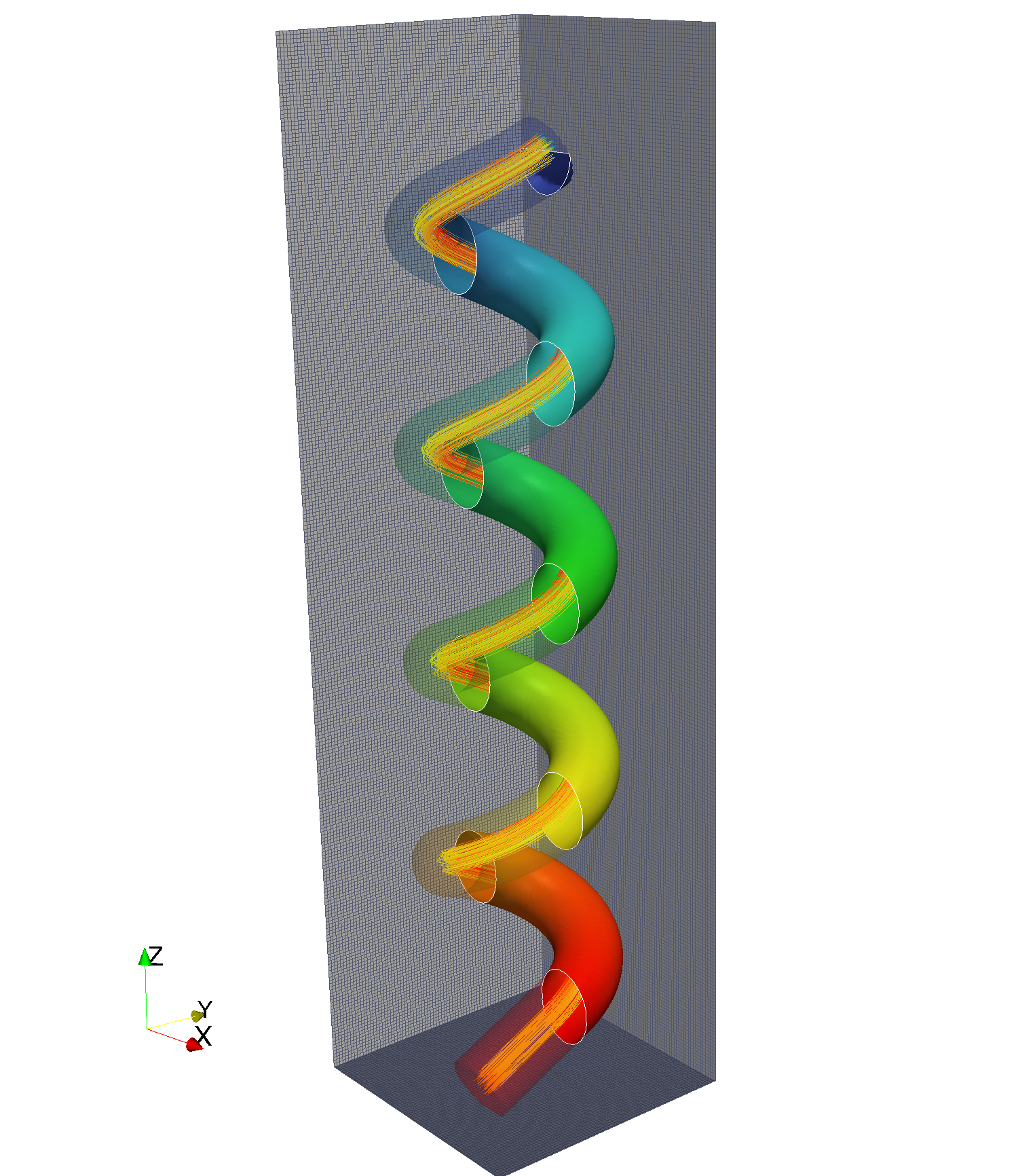}}
  \qquad\qquad
  \subfloat[\label{fig:helical_pipe:b}]{\includegraphics[trim=200 0 400 0, clip, width=0.38\textwidth]{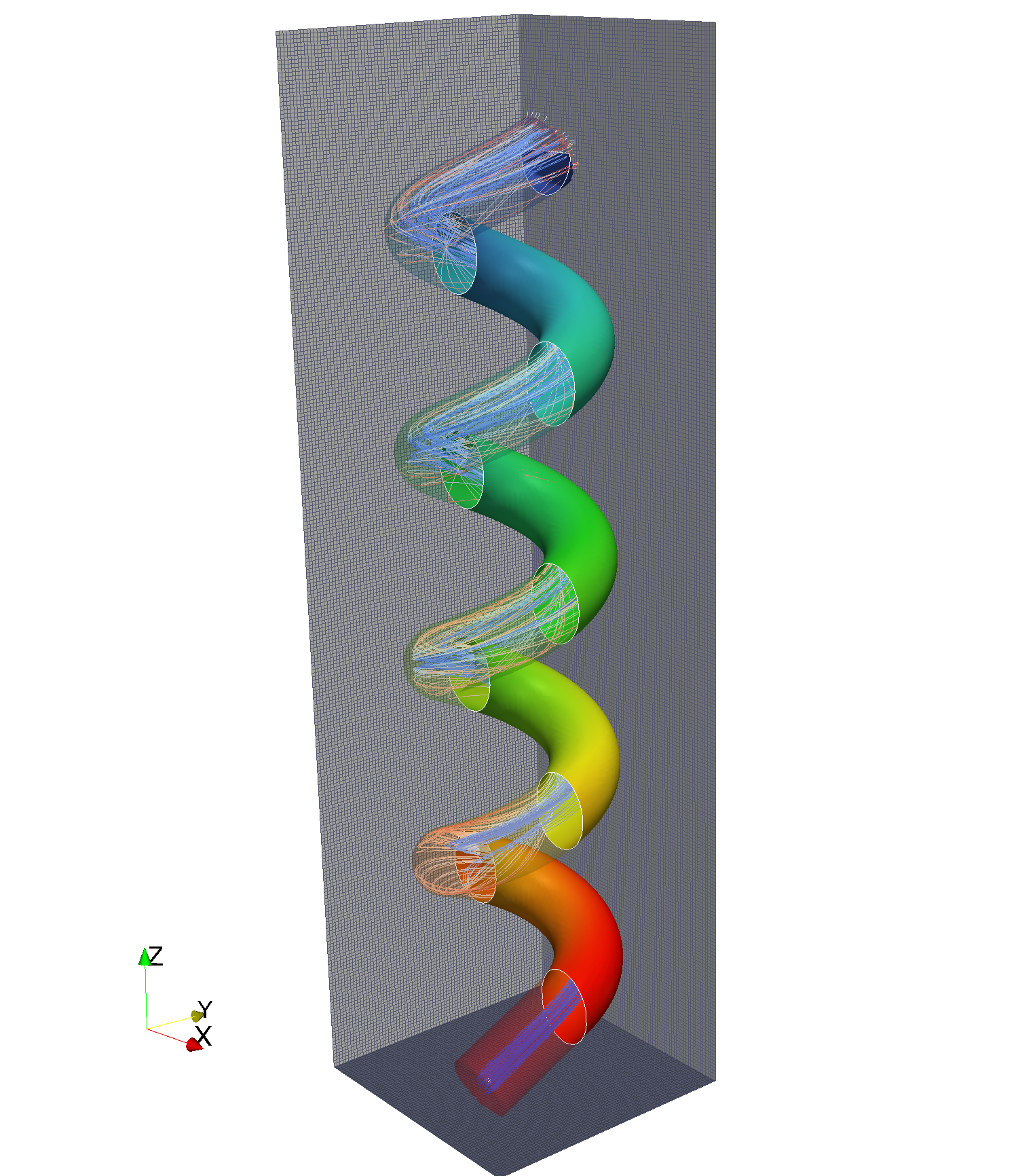}} 
  \caption{Helical Pipe Flow at $\RE=100$:
(a)~Flow during the ramp phase at $t=0.5$.
(b)~Laminar flow including twisted streamlines at $t=3.0$
when flow reaches steady state.
Streamlines start at inflow boundary located in a radius of $0.04$ around pipe centerline and are colored by velocity magnitude at $t=0.5$ and by vorticity at $t=3.0$.
Pressure distribution is visualized along pipe surface.}
  \label{fig:helical_pipe}
\end{figure}
The cross-section of the pipe defines a circle with a radius $r=0.1$
which is expanded along a helical curve
parametrized by $\bfx_{\textrm{curv}}(s)=(R\cos(2\pi s), R\sin(2\pi s), \alpha s)$.
This three-dimensional curve turns around the $x_3$-axis at a constant distance of $R=0.2$ and a constant thread pitch of $\alpha=0.6$.
The spiral twists four times parametrized by $s\in [-2,2]$.
In addition, a cylinder with a radius equal to the cross-section radius $r$ and a length of $h=0.35$ is put on the lower end of the spiral.
Its orientation is aligned to the tangential vector $\bft(s)=\partial \bfx_{\textrm{curv}}(s)/ \partial s$ of the helix at $s=-2$.

The front end of the cylinder defines the inflow boundary $\GammaIn$, where a velocity
$\bfg = u_{\textrm{in}} \bfn$
is imposed where $\bfn=\bft(s=-2)$ denotes the unit vector which is normal to the cylinder cross section.
Along the cylindrical and helical pipe surfaces, zero boundary conditions $\bfg=\bfzero$ are imposed,
while at the back end of the helical pipe a zero-traction Neumann boundary condition is set.
As before, all boundary conditions are enforced weakly.

The geometry as well as the flow solution are approximated on a cut background mesh~$\widehat{\mcT}_h$
covering a background cuboid $[-0.4,0.4]\times[-0.4,0.4]\times[-1.5,1.5]$
with $ 76\times 76\times 285$ trilinearly interpolated hexahedral elements.
In total, the number of active velocity and pressure degrees of freedom is~$859612$.

In the following, we consider a laminar pipe flow at $\RE=100$,
where the characteristic Reynolds number is defined as $\RE = u_{\textrm{eff}} r / \mu$
with an effective cross-section averaged velocity $u_{\textrm{eff}}$, the pipe radius $r$ and the viscosity of the fluid $\mu$.
At the inflow a constant velocity of $u_{\textrm{in}} = u_{\textrm{eff}} = 3.8$ is imposed which drives the mass flow.
Thereby, the velocity is chosen according to a wall Reynolds number of $\RE_{\tau}=180$ for pipe flows,
see, e.g., in~\cite{Hartel1994} for further explanations.
For this setup the viscosity is $\mu=1.9\cdot 10^{-3}$.
The pipe flow is investigated for a total simulation time of $T_{\textrm{end}}=3$ which is the approximated time needed for $3$ runs through the whole pipe
along its centerline.
For the temporal discretization a one-step-$\theta$ scheme with $\theta=0.5$ is applied and 
the time-step length is set to $\Delta t=0.001$, which ensures a
maximum $\CFL$-number $< 0.5$.
The inflow velocity is increased within $t\in[0,T_1]$ by a ramp function $1/2(1-\cos(\pi t/ T_1))$ with $T_1=0.1$.

In~\Figref{fig:helical_pipe}, the solution to the pipe flow is shown during the ramp phase at $t=0.5$
and when the flow is fully developed and reached steady state, as expected for this laminar setting.
During the ramp phase,
when the flow enters the helical pipe, streamlines follow the helical main curve through the pipe.
During the first turn the distance of the line of highest velocities to the $x_3$-axis decreases it's initial value~$R$,
however, remains unchanged for all following turns.
A snapshot of the ramp phase including streamlines and pressure distribution along the pipe surface is shown in~\Figref{fig:helical_pipe:a}.
Contour lines of the velocity magnitude at a cross-section clip plane $x_2=0.0$ is shown in~\Figref{fig:helical_pipe_low_RE_contour:a}
for the pipe range of $s\in[0,1.5]$.
Pressure isocontours along the pipe surface are visualized in~\Figref{fig:helical_pipe_low_RE_contour:b}.
Accurate enforcement of the zero boundary condition along the pipe surface as well as stability of velocity and pressure solutions
are clearly visible.
\begin{figure}[ht!]
  \centering
  \subfloat[\label{fig:helical_pipe_low_RE_contour:a}]{\includegraphics[trim=0 0 200 0, clip, width=0.35\textwidth]{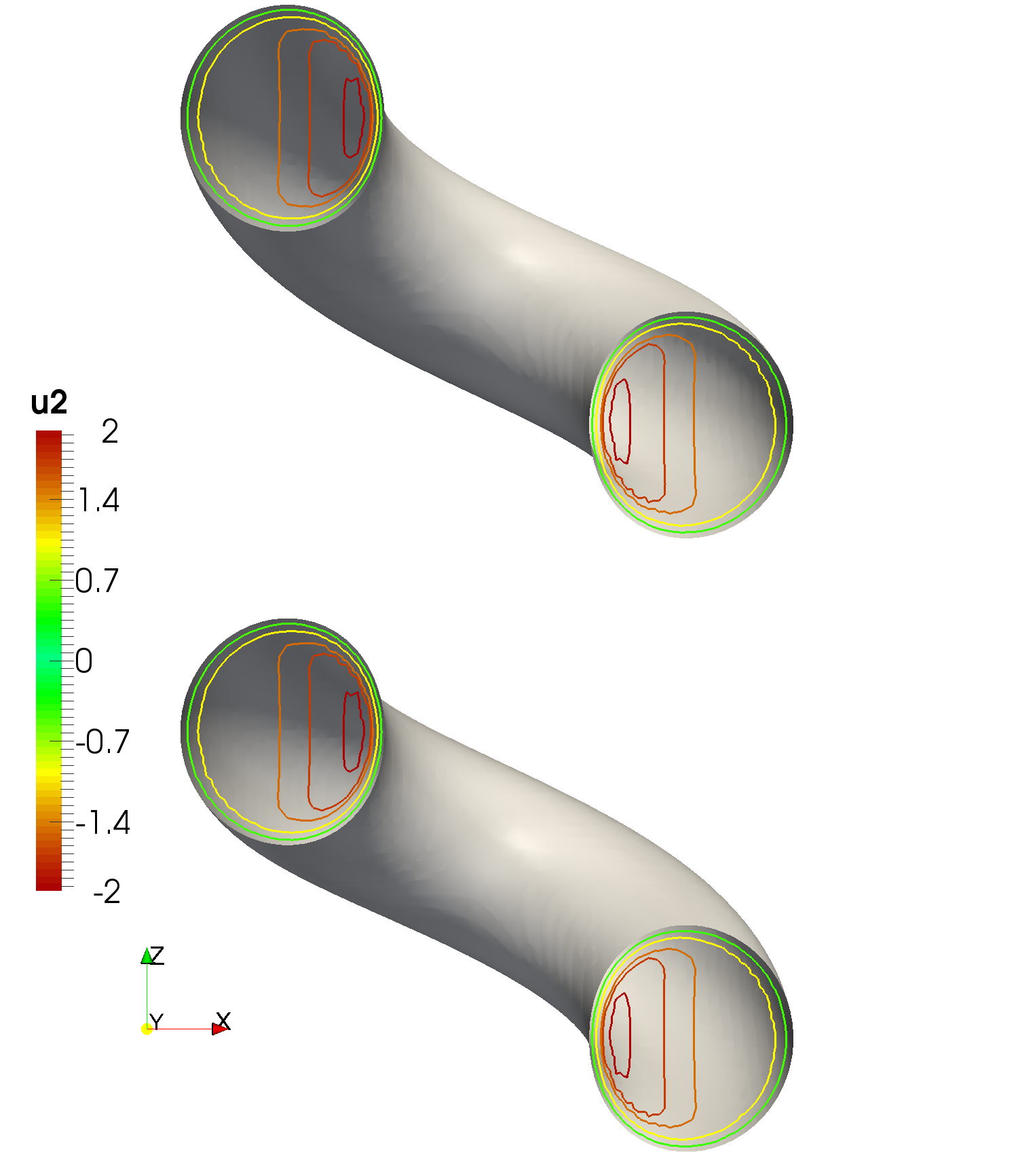}}
  \qquad
  \subfloat[\label{fig:helical_pipe_low_RE_contour:b}]{\includegraphics[trim=0 0 200 0, clip, width=0.35\textwidth]{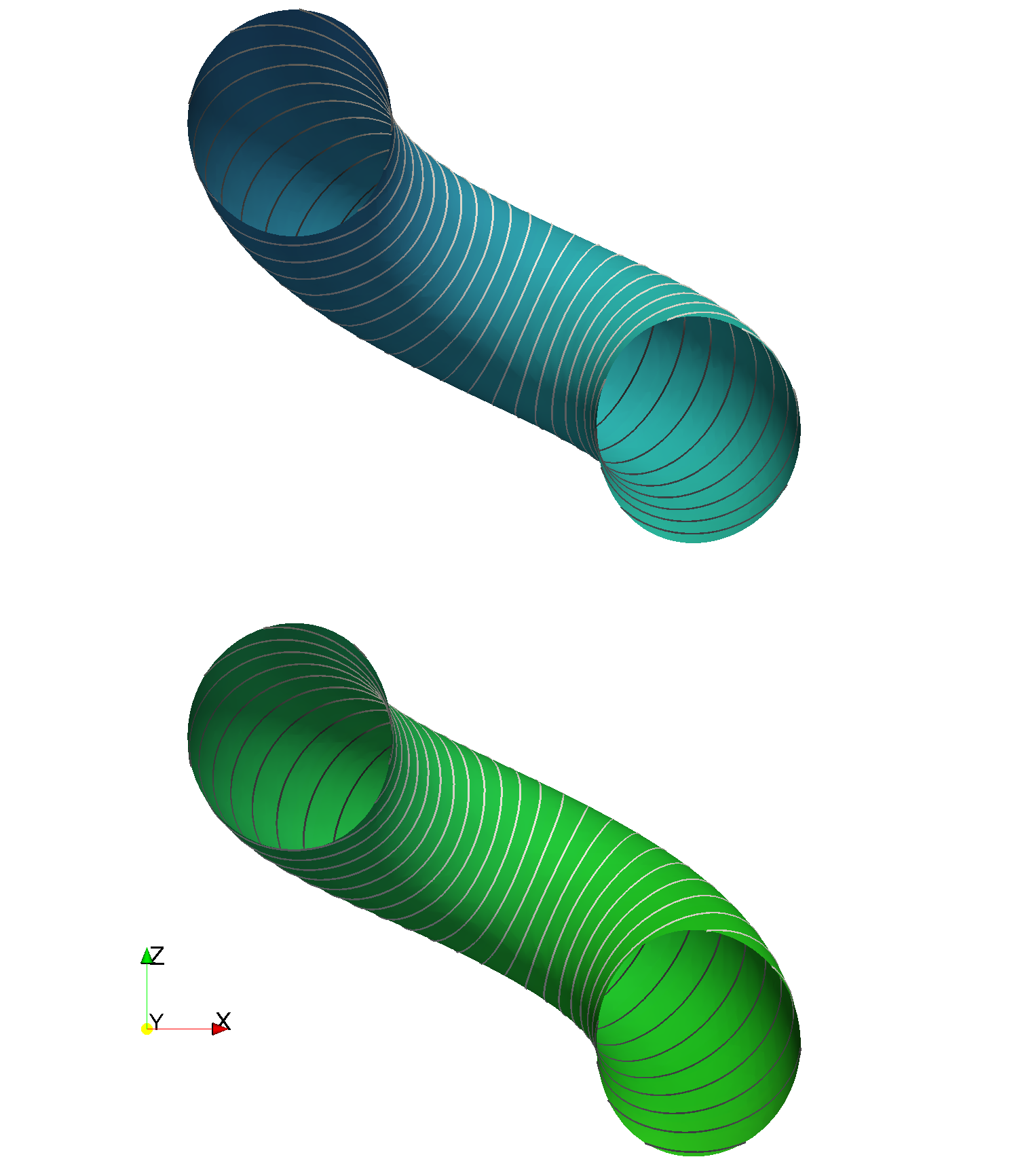}}
  \caption{Helical Pipe Flow at $\RE=100$ at $t=0.5$:
(a)~Isocontours of velocity magnitude at a cross section defined by the clip plane $x_2=0.0$ for a helix range of $s\in[0,1.5]$
show higher velocities with decreasing distance to the $x_3$-axis due to higher mass flow rate.
Zero boundary condition at the helix surface is accurately enforced and velocity solution is stable in the interior of the fluid domain as well as near the boundary.
(b)~Pressure contour lines at the helix surface show a stable pressure solution near the boundary.
}
  \label{fig:helical_pipe_low_RE_contour}
\end{figure}

When the flow is fully developed the flow pattern clearly changes and streamlines are twisted.
However, the flow reaches steady state after some time, which is shown in a snapshot in~\Figref{fig:helical_pipe:b} at $t=3.0$
including streamlines colored by vorticity and the pressure distribution is depicted along the helical surface.
It is inherently linked to the weak enforcement technique that the enforcement of boundary conditions in wall tangential direction
gets relaxed for higher local Reynolds numbers near the boundary.
However, the non-penetration condition in wall normal direction of the pipe is still sufficiently enforced.
Moreover, the solution exhibits non-oscillatory stable velocity and pressure due to stabilizing effects
of the different proposed stabilization operators in the interior of the fluid domain as well as near the boundary zone.
Contour lines are visualized in~\Figref{fig:helical_pipe_high_RE_contour:a} for the velocity and in~\Figref{fig:helical_pipe_high_RE_contour:b}
for the pressure.
\begin{figure}[ht!]
  \centering
  \subfloat[\label{fig:helical_pipe_high_RE_contour:a}]{\includegraphics[trim=0 0 200 0, clip, width=0.35\textwidth]{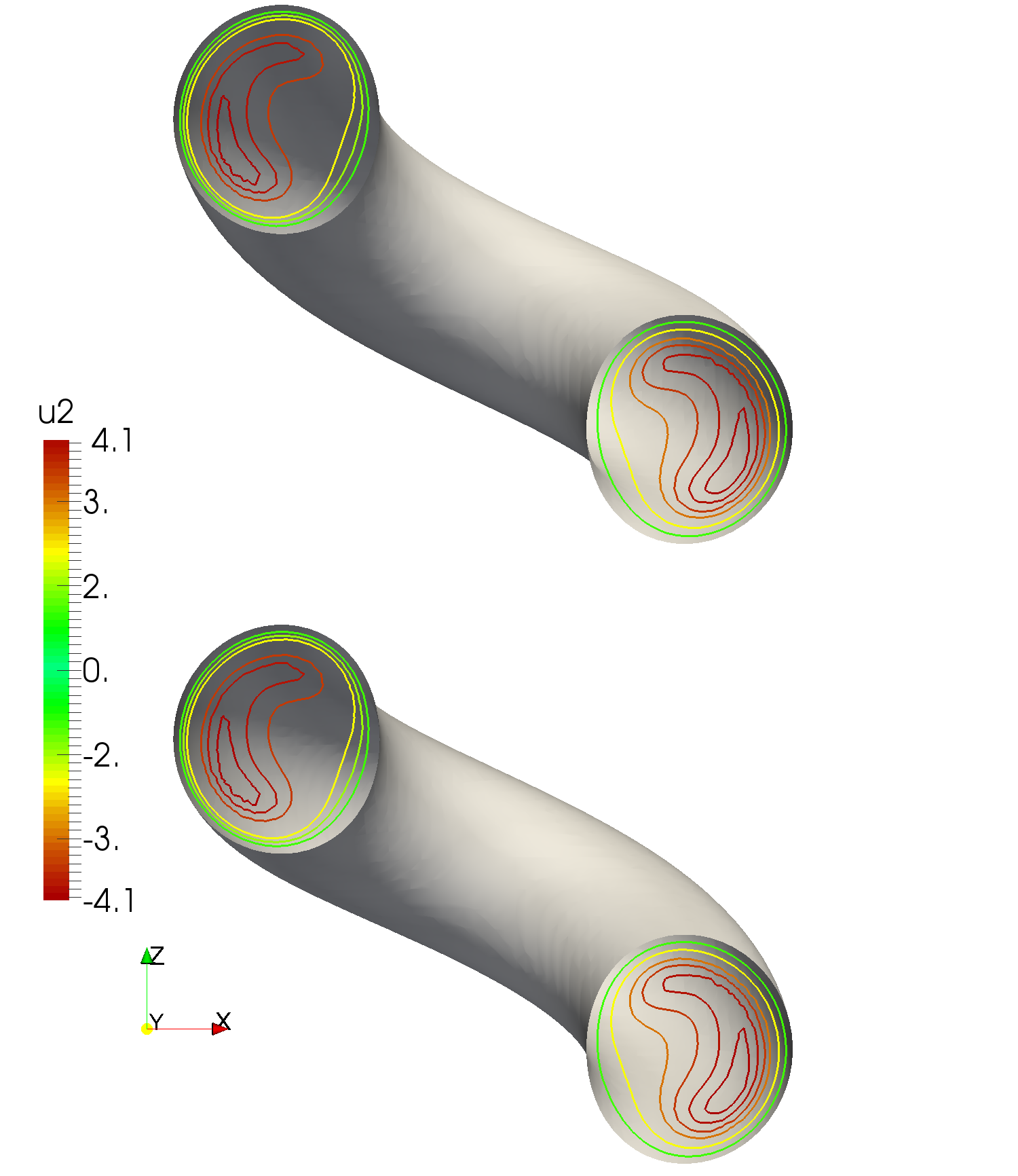}}
  \qquad
  \subfloat[\label{fig:helical_pipe_high_RE_contour:b}]{\includegraphics[trim=0 0 200 0, clip, width=0.35\textwidth]{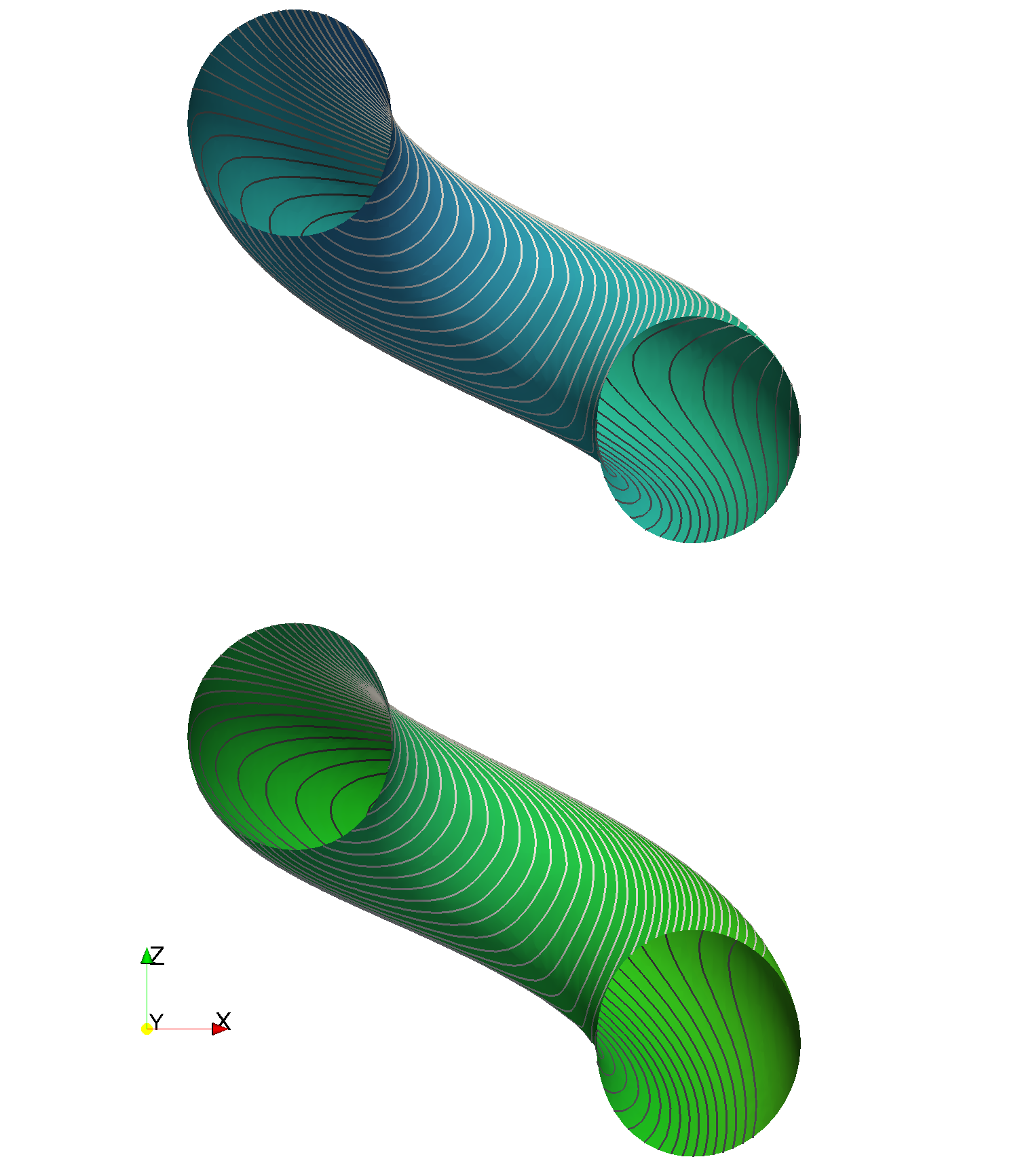}}
  \caption{Helical Pipe Flow at $\RE=100$ at $t=3.0$ when flow reaches steady state:
(a)~Isocontours of velocity magnitude at a cross section defined by the clip plane $x_2=0.0$ for a helix range of $s\in[0,1.5]$
exhibits stable velocity solution in the interior of the fluid domain as well as near the boundary.
(b)~Pressure contour lines at the helix surface show a stable pressure solution near the boundary.
}
  \label{fig:helical_pipe_high_RE_contour}
\end{figure}

\section{Useful Inequalities and Interpolation Estimates}
\label{sec:interpolation-est}
In this section, we start with the numerical analysis of the proposed
cut finite element method \eqref{eq:oseen-discrete-unfitted}.
We collect a number of useful inequalities and
elucidate the role of the ghost-penalties for the cut finite element method.
Moreover, we introduce suitable interpolation operators which will be instrumental
in deriving stability bounds in Section~\ref{sec:stability-properties} and a priori error estimates
in Section~\ref{sec:apriori-analysis}.

\subsection{Assumption on the Mesh and the Velocity Field}
\label{sec:assumpt-mesh-veloc}
To simplify the presentation of the numerical analysis of our cut
finite element method, we assume quasi-uniform meshes.  However, the
subsequent analysis can be also adapted to the case of locally
quasi-uniform meshes, as considered for boundary-fitted meshes in
\citet{BurmanFernandezHansbo2006}.
Recall that the velocity field $\bfbeta$ satisfies $\bfbeta \in
[W^{1,\infty}(\Omega)]^d \subseteq [C^{0,1}(\Omega)]^d$
by assumption and thus there is a piecewise constant discrete vector field $\bfbeta_h$
satisfying
\begin{equation}
h^{\onehalf}\|\bfbeta - \bfbeta_h\|_{0,\infty,\partial T}
+ \|\bfbeta - \bfbeta_h\|_{0,\infty,T} 
\lesssim h
\|\bfbeta\|_{1,\infty,T}
\quad
\text{and}
\quad
\| \bfbeta_h \|_{0,\infty,T} \lesssim  \| \bfbeta \|_{0,\infty,T}
\quad \foralls T\in\mcT_h.
\label{eq:beta-approximation}
\end{equation}
Such an approximative vector field $\bfbeta_h$ will be used at several occasions
in the forthcoming numerical analysis.
As usual, the notation $a\lesssim b$
means that $a\leqslant C b$ for a generic positive constant~$C$ which
is independent of~$h$.
Similar to~\cite{BurmanFernandezHansbo2006}, we additionally require 
that the flow field $\bfbeta$ is sufficiently resolved by the mesh in the sense
that
for some constant $c_{\beta} > 1$ and
$\foralls T\in\mcT_h$
\begin{align}
c_{\beta}^{-1}\| \bfbeta \|_{0,\infty,T'} 
\leqslant
\| \bfbeta \|_{0,\infty,T} 
\leqslant
c_{\beta}\| \bfbeta \|_{0,\infty,T'}
\quad \foralls T'\in\omega(T),
  \label{eq:beta-resolution-I}
\end{align}
where $\omega(T)$ denotes a local patch of elements neighboring $T$.
Assumption~(\ref{eq:beta-resolution-I}) can be ensured if, e.g.,
\begin{align}
  \|\bfbeta\|_{1,\infty,\omega(T)} \leqslant d_{\beta} h^{-1} \|\bfbeta\|_{0,\infty,\omega(T)} \quad \foralls T\in\mcT_h
  \label{eq:beta-resolution-II}
\end{align}
is satisfied for some constant $0 < d_{\beta} < h/\diam(\omega(T)) =: c_{\omega}$.
Then 
\begin{align}
 \|\bfbeta\|_{0,\infty,\omega(T)} &\leqslant 
                         \|\bfbeta\|_{0,\infty,T'}
                         + \diam(\omega(T)) \|\bfbeta\|_{1,\infty,\omega(T)}
\leqslant
                         \|\bfbeta\|_{0,\infty,T'}
                         + c_{\omega}^{-1} d_{\beta}\|\bfbeta\|_{0,\infty,\omega(T)}
\end{align}
and consequently, assumption~(\ref{eq:beta-resolution-I}) holds with $c_{\beta}^{-1} = 1 - c_{\omega}^{-1} d_{\beta}$
since
for any $T,T'\in\omega(T)$
\begin{align}
(1 - c_{\omega}^{-1} d_{\beta}) \|\bfbeta\|_{0,\infty,T} 
\leqslant 
(1 - c_{\omega}^{-1} d_{\beta})\|\bfbeta\|_{0,\infty,\omega(T)} 
\leqslant 
  \|\bfbeta\|_{0,\infty,T'}.
\end{align}
Due to assumption~(\ref{eq:beta-resolution-I}),
the piecewise constant stabilization parameters
are comparable in the sense that for $\phi \in \{\phi_u, \phi_{\beta}, \phi_p \}$
\begin{align}
  (c_{\beta} c_{\mcT_h})^{-1} \phi_{T'} \leqslant \
  \phi_{T} \leqslant \
  (c_{\beta} c_{\mcT_h}) \phi_{T'} 
\quad \foralls T' \in \omega(T),
\end{align}
where $c_{\mcT_h}$ characterizes the quasi-uniformness of $\mcT_h$.
With this in mind, we will simply write
\begin{align}
\label{eq:local_equivalence_stab_param_patch}
 \phi_T \sim \phi_{T'} \quad\forall T'  \in \omega(T),
\qquad 
 \phi_F \sim \phi_{T} \quad \forall T \in \omega(F).
\end{align}
\subsection{Trace Inequalities and Inverse Estimates}
\label{ssec:trace-inequalities}
Throughout our analysis, we will make heavy use of
the following well-known generalized inverse and trace inequalities for discrete functions $v_h
\in \mcX_h$:
\begin{alignat}{3}
  \norm{D^j v_h}_{T}
  & \lesssim h^{i-j} \norm{D^i v_h}_{T} & & \quad \foralls T \in
  \mcT_h, \quad 0\leqslant i\leqslant j,
  \label{eq:inverse-estimates_T}
  \\
  \norm{\partial_{\bfn}^j v_h}_{\partial T}
  & \lesssim h^{i-j-1/2} \norm{D^i v_h}_{T} & & \quad \foralls T \in \mcT_h, 
  \quad 0\leqslant i\leqslant j,
  \label{eq:inverse-estimates_F}
  \\
  \intertext{and their counterpart for elements $T$ which are arbitrarily intersected by the boundary $\Gamma$}
  \norm{\partial_{\bfn}^j v_h}_{\Gamma \cap T} & \lesssim
  h^{i-j-1/2} \norm{D^i v_h}_{T} & & \quad \foralls T \in
  \mcT_h, 
  \quad 0\leqslant i\leqslant j,
  \label{eq:inverse-estimates_Gamma}
\end{alignat}
proven in \cite{HansboHansbo2002,BurmanHansboLarsonEtAl2016}.
Here and throughout this work, we use the notation $a \lesssim b$ for
$a\leqslant C b$ for some generic
positive
constant $C$ which varies with the context
but is always independent of the mesh size $h$ and the position of
$\Gamma$ relative to~$\mcT_h$.
For $v \in H^1(\Oast_h)$, we will make use of trace inequalities of
the form
  \begin{align}
    \label{eq:trace-inequality}
    \norm{v}_{\partial T}
    &\lesssim
    h^{-1/2} \norm{v}_{T} +
    h^{1/2}  \norm{\nabla v}_{T}
    \quad \foralls T \in \mcT_h,
    \\
    \norm{v}_{\Gamma \cap T}
    &\lesssim
      h^{-1/2} \norm{v}_{T}
      + h^{1/2} \norm{\nabla v}_{T}
    \quad \foralls T \in \mcT_h,
    \label{eq:trace-inequality-cut}
  \end{align}
  see~\cite{HansboHansbo2002,BurmanHansboLarsonEtAl2016}
for a proof of the second one.
  Finally, we recall the well-known Poincar\'e and Korn
  inequalities~\cite{BrennerScott2008}, stating that 
  $\foralls \bfv \in
    [H_0^1(\Omega)]^d$,
  \begin{align}
    \| \bfv \|_{0,\Omega} 
    &\lesssim C_P \| \nabla \bfv \|_{0,\Omega},
    \label{eq:Poincare-I}
    \\
    \| \nabla \bfv \|_{0,\Omega}
    &\lesssim
    \| \bfeps(\bfv) \|_{0,\Omega},
    \label{eq:korn-I}
  \end{align}
  and the following variants if the boundary trace of $\bfv$ is not
  vanishing, that is  $\foralls \bfv \in [H^1(\Omega)]^d$:
  \begin{alignat}{2}
    \| \bfv \|_{0,\Omega} 
    &\lesssim C_P (\| \nabla \bfv \|_{0,\Omega} + \|\bfv\|_{\Gamma}),
    \label{eq:Poincare-II}
    \\
    \| \bfv \|_{1,\Omega}
    & \lesssim
    \| \bfeps(\bfv) \|_{0,\Omega}
    + \| \bfv \|_{0,\Omega},
    \label{eq:korn-II}
    \\
    \| \bfv \|_{1,\Omega}
    & \lesssim
    \| \bfeps(\bfv) \|_{0,\Omega}
    + \| \bfv \|_{0,\Gamma}.
    \label{eq:korn-III}
  \end{alignat}

\subsection{Interpolation Operators}
\label{ssec:interpolation}
To construct an appropriate interpolation operator $L^2(\Omega) \to
\mcX_h$, we first recall that for the Sobolev spaces $W^{m,q}(\Omega)$, $0
\leqslant m < \infty$, $1 \leqslant q \leqslant \infty$,
an extension operator can be defined
\begin{equation}
  \label{eq:extension-operator-definition}
  E: W^{m,q}({\Omega}) \rightarrow W^{m,q}(\Oast)
\end{equation}
which is bounded
\begin{equation}
  \label{eq:extension-operator-boundedness}
  \norm{E v}_{m,q,\Oast} \lesssim \norm{v}_{m,q,\Omega},
\end{equation}
see \cite{Stein1970}
for a proof. Occasionally, we write $v^\ast = Ev$. 
Choosing some fixed Lipschitz-domain $\Oast$ such
that $\Oast_h \subset \Oast$ for $h \lesssim 1$,
we can define for any interpolation operator $\pi_h: H^s(\Oast_h) \to \mcX_h$
its ``fictitious domain'' variant 
$\pi_h^\ast: H^s(\Omega) \to \mcX_h$ by simply  requiring that
  \begin{equation}
    \pi_h^\ast u = \pi_h(u^\ast)
    \label{eq:clement-interp-ext}
  \end{equation}
for $u \in H^s(\Omega)$.
In particular, we choose $\pi_h$ to be the \emph{Cl\'ement operator},
see for instance \cite{ErnGuermond2004}.
Recall that for $v \in H^r(\Oast)$,
the following interpolation estimates holds
for the Cl\'ement interpolant:
\begin{alignat}{3}
\| v - \pi_h v \|_{s,T} 
& \lesssim
  h^{t-s}| v |_{t,\omega(T)},
&\quad 0\leqslant s \leqslant t \leqslant m \quad &\foralls T\in
  \mcT_h,
  \label{eq:interpest0}
  \\
\| v - \pi_h v \|_{s,F} &\lesssim h^{t-s-1/2}| v |_{t,\omega(F)},
&\quad 0\leqslant s \leqslant t - 1/2 \leqslant m-1/2 \quad &\foralls F\in
  \mcF_h,
  \label{eq:interpest1}
\end{alignat}
with $s, t \in \NN$, $m=\min\{r,k+1\}$ and $k$ the interpolation order of $\mcX_h$.
Here, $\omega(T)$ and $\omega(F)$ are the sets of elements in $\mcT_h$ sharing at least
one vertex with $T$ and $F$, respectively. 
Due to the boundedness of the extension
operator~\eqref{eq:extension-operator-boundedness},
we observe that the extended Cl\'ement interpolant $\pi_h^{\ast}$ satisfies
\begin{alignat}{3}
  \| v^\ast - \pi_h^{\ast} v \|_{s,\mcT_h} 
  & \lesssim
  h^{t-s}\| v \|_{t,\Omega},
  & &\quad 0\leqslant s \leqslant t \leqslant m,
  \label{eq:interpest0-ast}
  \\
  \| v^\ast - \pi_h^{\ast} v \|_{s,\mcF_h} 
  &\lesssim h^{t-s-1/2} \| v \|_{t,\Omega},
  & &\quad 0\leqslant s \leqslant t - 1/2 \leqslant m - 1/2,
  \label{eq:interpest1-ast}
\end{alignat}
where the broken norms $\| \cdot \|_{\mcF_h}$ and $\|\cdot\|_{\mcT_h}$ are defined as
in Section~\ref{ssec:notation}.
In particular, we make use of the stability property
\begin{alignat}{3}
  \| \pi_h^{\ast} v \|_{s,\Omega} &\lesssim \| v \|_{s,\Omega},
  & &\quad 0 \leqslant s \leqslant m.
  \label{eq:interpest0_Clement_H_stability}
\end{alignat}
The Cl\'ement interpolant is denoted by~$\boldsymbol{\pi}_h^{\ast}$
for vector-valued functions~$\bfv$ and by $\Pi_h^{\ast}$
for functions in a product space.

A main ingredient in the analysis of the continuous interior penalty
method in~\cite{BurmanFernandezHansbo2006,BurmanErn2007} is the use
of the \emph{Oswald interpolation operator}.
The Oswald interpolation operator defines a mapping
$\mcO_h: \mcP_{\mathrm{dc}}^k(\mcT_h) \to \mcP^{\max\{1,k\}}(\mcT_h)$,
with $\mcP_{\mathrm{dc}}^k(\mcT_h)$ and $\mcP^k(\mcT_h)$ denoting the space of
discontinuous and continuous piecewise polynomials of order $k$.
More precisely, for $v \in \mcP_{\mathrm{dc}}^k(\mcT_h)$, the function $\mcO_h
v$ is constructed in each interpolation node $\bfx_i$
by the average value
\begin{align}
  \mcO_h v(\bfx_i) = \dfrac{1}{\mathrm{card}(\mcT_h(\bfx_i))} \sum_{T \in
  \mcT_h(\bfx_i)} v|_T(\bfx_i)
  \label{eq:oswald-construction}
\end{align}
where $\mcT_h(\bfx_i)$ is the set of all elements $T\in\mcT_h$ sharing the
node $\bfx_i$.
In particular, it was shown there that for $w_h \in
\mcX_h^{\mathrm{dc}} = \{ v \in L^2(\Oast): v|_T \in  \mcP^k(T)\; \foralls T
  \in \mcT_h\}$,
the fluctuation $ w_h - \mcO_h w_h$ can be controlled 
in terms of jump-penalties:
\begin{lemma}
\label{lem:Oswald-interpolant}
  Let $\phi$ be a piecewise constant function
  and $w_h \in \mcX_h^{\mathrm{dc}}$. Then
\begin{equation}
  \|\phi^{\onehalf}(w_h - \mcO_h w_h)\|_T^2
   \lesssim
   \sum_{F\in\mcF_i(T)}
   \phi_T h
   \|\jump{w_h}\|_F^2
   \quad \foralls w_h\in \mcX_h^{\mathrm{dc}},
\label{eq:Oswald-interpolant}
\end{equation}
where $\mcF_i(T)$ denotes the set of all faces $F \in \mcF_i$ with $F \cap T \neq \emptyset$,
and the hidden constant depends only on the shape-regularity of the
mesh, the order of the finite element space $\mcX_h^{\mathrm{dc}}$ and the dimension $d$.
\end{lemma}
\noindent We refer to \cite{BurmanErn2007} for a proof. The previous lemma
elucidates the role of the CIP stabilization
operators~\eqref{eq:cip-s_beta}--\eqref{eq:cip-s_p}
as a control of certain fluctuations:
\begin{corollary}
\label{cor:cip-fluctuation-control}
  Under the assumptions of Lemma~\ref{lem:Oswald-interpolant} it holds that
\begin{align}
 \| \phi_{\beta}^{\onehalf} (\bfbeta_h\cdot\nabla \bfv_h - \mcO_h(\bfbeta_h\cdot\nabla \bfv_h)) \|_{T}^2
 & \lesssim 
 \sum_{F\in\mcF_i(T)} 
 \phi_{\beta,T} h \|\jump{\bfbeta_h \cdot\nabla \bfv_h }\|_F^2 ,
 \label{eq:oswald-stabilization-beta}
 \\
 \| \phi_p^{\onehalf} (\nabla q_h - \mcO_h (\nabla q_h)) \|_{T}^2
   & \lesssim
\sum_{F\in\mcF_i(T)} \phi_{p,T} h {\|\jump{\bfn_F \cdot \nabla q_h}\|_F^2} ,
 \label{eq:oswald-stabilization-pressure}
\\
\| \phi_u^{\onehalf}(\nabla \cdot \bfv_h - \mcO_h (\nabla \cdot \bfv_h)) \|_{T}^2
   & \lesssim \sum_{F\in\mcF_i(T)}
\phi_{u,T} h {\|\jump{\nabla \cdot \bfv_h}\|_F^2}.
 \label{eq:oswald-stabilization-div}
\end{align}
\end{corollary}
In the forthcoming stability and a priori error analysis we will make
heavy use of certain continuous, piecewise linear versions of the
stabilization parameters 
defined by 
\begin{align}
  \widetilde{\phi} = \mcO_h\phi, \quad \phi \in \{\phi_{\beta}, \phi_u, \phi_p\}.
  \label{eq:smoothed_stab_param}
\end{align}
Then by the definition of the Oswald interpolant and the local comparability of
the stabilization parameters~(\ref{eq:local_equivalence_stab_param_patch}) 
\begin{align}
\label{eq:local_equivalence_stab_param_oswald}
  \widetilde{\phi}|_T \sim \phi_T, \quad  \phi \in \{\phi_{\beta}, \phi_u, \phi_p\}.
\end{align}
We conclude this section by stating and proving a helpful lemma on the 
quasi-local stability of the Oswald interpolation in certain weighted norms.
\begin{lemma}
\label{lem:oswald-quasi-local-stability}
Let $v_h \in \mcX_h$ and let $\phi$ be a piecewise constant function
defined on $\mcT_h$.
Then
\begin{align}
\label{eq:oswald-quasi-local-stability}
  \| \phi^{\onehalf} \mcO_h(v_h)\|_{T} &\lesssim \| \phi_T^{\onehalf}
  v_h \|_{\omega(T)}
\quad  \foralls T \in \mcT_h.
\end{align}
\end{lemma}
\begin{proof}
The proof is a simple consequence of the Lemma~(\ref{eq:Oswald-interpolant})
and the inverse inequality~(\ref{eq:inverse-estimates_F}):
\begin{align}
   \| \phi^{\onehalf} \mcO_h(v_h)\|_{T}^2
= \| \phi^{\onehalf} v_h\|_{T}^2
+  \| \phi^{\onehalf} (\mcO_h(v_h) - v_h)\|_{T}^2
\lesssim
\| \phi^{\onehalf} v_h\|_{T}^2
+ h \|\phi_T^{\onehalf}\jump{v_h}\|_{\mcF_h(T)}^2
\lesssim
\| \phi_T^{\onehalf} v_h\|_{T}^2
+ \|\phi_T^{\onehalf}v_h\|_{\omega(T)}^2.
\end{align}
\end{proof}
\subsection{The Role of the Ghost Penalties}
\label{ssec:norms}
Following \cite{BurmanFernandezHansbo2006,BraackBurmanJohnEtAl2007},
the naturals norms associated with the discrete variational problem
defined for fitted meshes as \mbox{$A_h+S_h=L_h$} in
\eqref{eq:Ah-form-def}--\eqref{eq:Lh-form-def}, are given by
\begin{align}
  \label{eq:oseen-norm-u-fitted}
  \tn \bfu_h  \tn^2 &=
 \| \sigma^{1/2} \bfu_h \|_{\Omega}^2
 + \| \mu^{1/2} \grad \bfu_h\|_{\Omega}^2
+ \| (\gamma \mu/h)^{1/2} \bfu_h \|^2_{\Gamma} + s_u(\bfu_h, \bfu_h)
\nonumber
 \\
   & \quad
+ \| |\bfbeta \cdot \bfn |^{1/2} \bfu_h \|_{\Gamma}^2
+ \| (\gamma \phi_u/h )^{1/2} \bfu_h \cdot \bfn \|^2_{\Gamma}
+  s_{\beta}(\bfu_h, \bfu_h),
\\
  \label{eq:oseen-norm-p-fitted}
  \tn  p_h \tn^2 &=
 \Phi_{p}\| p_h \|_{\Omega}^2 + s_p(p_h,p_h),
\\
 \label{eq:oseen-norm-up-fitted}
 \tn U_h  \tn^2 &=
 \tn \bfu_h  \tn^2 + \tn p_h  \tn^2,
\end{align}
with $\Phi_p$ defined by~\eqref{eq:phi_p_definition}.
Using similar norms, inf-sup stability and energy-type error
estimates were proven in~\cite{BurmanFernandezHansbo2006}.

The main challenge in developing Nitsche-type fictitious domain
methods is now to establish stability and a priori error estimates which are
independent of the positioning of the unfitted boundary within the
background mesh. The key idea is to add certain (weakly) consistent stabilization
terms in the vicinity of the boundary which allow to extend suitable
norms for finite element functions from the physical domain $\Omega$
to the entire fictitious domain $\Oast_h$.
The subsequent lemmas elucidate the role of the different ghost penalties and motivate the definition of the
fictitious domain norms from \eqref{eq:oseen-norm-unfitted-u}--\eqref{eq:oseen-norm-up}.
The following lemma was
proven in~\cite{Burman2010,MassingLarsonLoggEtAl2014}:

\begin{lemma} \label{lem:norm_equivalence}
  Let $\Omega$, $\Oast_h$ and $\mcF_{\Gamma}$ be defined as in
   Section~\ref{ssec:cutfem-spaces}. Then
  for scalar functions \mbox{$p_h\in\mcQ_h$} as well as for vector-valued equivalents \mbox{$\bfu_h \in \mcV_h$} 
  the following estimates hold
  \begin{alignat}{1}
    \label{eq:ghost_penalty:norm_equivalence_L2_pressure}
    \| p_h \|_{\Oast_h}^2
    &\lesssim
    \bigl( \| p_h
    \|_{\Omega}^2
    +
    \sum_{F \in \mcF_{\Gamma}}
    \sum_{1\leqslant j \leqslant k}{h^{2j+1}\langle \jump{\nablan^j p_h}, \jump{\nablan^j p_h}\rangle_F}
    \bigr)
    \lesssim
    \|  p_h \|_{\Oast_h}^2, \\
    \label{eq:ghost_penalty:norm_equivalence_L2}
    \| \bfu_h \|_{\Oast_h}^2
    &\lesssim
    \bigl (
    \| \bfu_h \|_{\Omega}^2
    +
    \sum_{F \in \mcF_{\Gamma}}
    \sum_{1\leqslant j \leqslant k}{h^{2j+1}\langle \jump{\nablan^j \bfu_h}, \jump{\nablan^j \bfu_h}\rangle_F}
    \bigr )
    \lesssim
    \| \bfu_h \|_{\Oast_h}^2, \\
    \label{eq:ghost_penalty:norm_equivalence_H1}
    \| \nabla \bfu_h \|_{\Oast_h}^2
    &\lesssim
    \bigl( \| \nabla \bfu_h
    \|_{\Omega}^2
    +
    \sum_{F \in \mcF_{\Gamma}}
    \sum_{1\leqslant j \leqslant k}{h^{2j-1}\langle \jump{\nablan^j \bfu_h}, \jump{\nablan^j \bfu_h}\rangle_F}
    \bigr)
    \lesssim
    \| \nabla \bfu_h \|_{\Oast_h}^2,
  \end{alignat}
  where the hidden constants depend only on the shape-regularity and the polynomial order, but not on the mesh or the location of~$\Gamma$
within~$\mcT_h$.
\end{lemma}
  Thus the role of the
  ghost-penalties $g_{\mu}$ and $g_{\sigma}$ is to
  extend the control
  of the viscous and reactive 
  element contributions  in the velocity norm~\eqref{eq:oseen-norm-u-fitted}
  from $\Omega$ to the fictitious domain $\Omega_h^{\ast}$:
  \begin{corollary}
    \label{cor:norm_equivalence_visc_reac_pres}
  Let $\Omega$, $\Oast_h$ and $\Fast$ be defined as in Section~\ref{ssec:cutfem-spaces}.
Then for $\bfu_h \in \mcV_h$ the following scaled estimates hold
  \begin{align}
    \label{eq:ghost_penalty:norm_equivalence_H1_viscous}
    \| \mu^{\onehalf}\nabla \bfu_h \|_{\Oast_h}^2
    &\lesssim
    \| \mu^{\onehalf} \nabla \bfu_h
    \|_{\Omega}^2
    +
    g_{\mu}(\bfu_h,\bfu_h)
    \lesssim
    \| \mu^{\onehalf} \nabla \bfu_h \|_{\Oast_h}^2,
\\
    \label{eq:ghost_penalty:norm_equivalence_L2_reaction}
    \| \sigma^{\onehalf} \bfu_h \|_{\Oast_h}^2
    &\lesssim
    \| \sigma^{\onehalf}\bfu_h \|_{\Omega}^2
    +
    g_{\sigma}(\bfu_h,\bfu_h)
    \lesssim
    \| \sigma^{\onehalf}\bfu_h \|_{\Oast_h}^2,
  \end{align}
  where the ghost-penalty operators $g_\mu$ and $g_\sigma$ are defined as in \eqref{eq:ghost-penalty-sigma} and \eqref{eq:ghost-penalty-nu}.
\end{corollary}

One major difference to the numerical analysis proposed by
\citet{BurmanFernandezHansbo2006} consists in the extended inf-sup
stability derived in this work.
In \cite{BurmanFernandezHansbo2006} inf-sup stability was proven with
respect to a weaker semi-norm and the orthogonality and approximation
properties of the $L^2$ projection were exploited to establish
\apriori~error estimates.
\citet{BurmanClausMassing2015} introduced a stabilized,
approximate $L^2$ projection to
facilitate the numerical analysis of a stabilized cut finite element method for the
three field Stokes problem based on equal-order, $P_1$ elements.
The definition of this stabilized $L^2$ projection incorporates
properly scaled ghost-penalty stabilization and leads to a
perturbation of the $L^2$ orthogonality in the vicinity of the
embedded boundary which is difficult to handle for other than first
order elements.
To compensate the lack of a suitable CutFEM variant
of the $L^2$ projection in the subsequent stability analysis,
we will gain
further control over weakly scaled (semi)-norms
\begin{equation}
\| \phi_{u}^{\onehalf} \nabla\cdot \bfu_h \|_{\Oast_h}, \quad
\| \phi_{\beta}^{\onehalf}(\bfbeta \cdot \nabla \bfu_h + \nabla p_h)\|_{\Oast_h}, \quad
\|p_h\|_{\Oast_h}
\label{eq:gls-semi-norms}
\end{equation}
with the help of the CIP stabilization operators \mbox{$s_u,s_\beta,s_p$}
and their ghost penalty counterparts.
Enhancing our natural norms by these residual-based stabilization like
contributions allows us to use alternative approximation operators
such as the Cl\'ement operator for which a proper CutFEM extension
can be defined by~\eqref{eq:clement-interp-ext}.

\begin{remark}
Previously, \citet{Burman2007a} established stability and convergence
results in related norms for stabilized finite element methods for
Friedrich's systems by demonstrating how to control the relevant
$h$-weighted graph-norms using CIP stabilizations.
Our approach here was also inspired by the presentation in \citet{KnoblochTobiska2013},
who showed how to gain control over the semi-norm
$\|\phi_{\beta}^{\onehalf}(\bfbeta_h \cdot \nabla \bfu_h + \nabla
p_h)\|_{\Omega}$  in a local projection stabilized, fitted finite element method for the Oseen
problem.
\end{remark}

To ensure stability and optimality for cut finite element
approximations in the different flow regimes, the
semi-norms~\eqref{eq:gls-semi-norms} need to be extended to the
enlarged domain~$\Oast_h$ with the help of related ghost-penalty
operators \mbox{$g_u,g_\beta,g_p$}.  For this purpose, in the
following useful estimates according to the aforementioned norms are
derived.

\begin{corollary}
\label{cor:norm_equivalence_rsb_pressure_div}
For \mbox{$p_h\in\mcQ_h$} and \mbox{$\bfu_h \in \mcV_h$} with $\Phi_p$ from \eqref{eq:phi_p_definition}, the following estimates hold
\begin{alignat}{1}
    \label{eq:norm_equivalence_rsb_pressure}
    \Phi_p \| p_h \|_{\Oast_h}^2
    &\lesssim
    \Phi_p \|  p_h \|_{\Omega}^2
    + g_{p}(p_h,p_h),
    \\
    \label{eq:norm_equivalence_rsb_div}
    \| \phi_u^{\onehalf} \nabla \cdot \bfu_h \|_{\Oast_h}^2
    &\lesssim
    \| \phi_u^{\onehalf} \nabla \cdot \bfu_h \|_{\Omega}^2
    + g_{u}(\bfu_h,\bfu_h)
  \end{alignat}
  where the ghost-penalty operators $g_p$ and $g_u$ are defined as in
  \eqref{eq:ghost-penalty-u} and \eqref{eq:ghost-penalty-p}.
\end{corollary}

\begin{proof}
The estimate for the scaled pressure $L^2$-norm follows from applying
Lemma~\ref{lem:norm_equivalence} to \mbox{$\Phi_p^{\onehalf}p_h$} and
the fact that \mbox{$\Phi_p h^2 \lesssim \phi_{p,T} \sim \phi_{p,F}$},
see definitions \eqref{eq:phi_p_definition} and
\eqref{eq:cip-s_scalings}, together with \mbox{$h\lesssim C_P$}.

The estimate for the weakly scaled incompressibility results from a
localized variant of the ghost-penalty
Lemma~\ref{lem:norm_equivalence}, as proven in
\cite{MassingLarsonLoggEtAl2014}, and the comparability assumption
\mbox{$\phi_{u,T} \sim \phi_{u,T'} \sim \phi_{u,F}$}
\eqref{eq:local_equivalence_stab_param_patch} on patches of elements
\mbox{$T'\in \omega(T)$} surrounding intersected elements
\mbox{$T\in\mcT_{\Gamma}$}.  Note that the number of element traversals required to walk from a cut element to an uncut
element in a boundary zone patch~$\omega(T)$ 
is bounded independent of~$h$ owing to the mesh assumption~G3,
see Section~\ref{ssec:cutfem-spaces}.
\end{proof}

More subtle is the role of the mixed norm incorporating the advective term and the pressure gradient.
Estimates using the ghost-penalty operators~$g_{\beta}$ and $g_p$ can be deduced as follows.

\begin{lemma}
\label{lem:norm-extension-rbs-norms}
Let $\bfu_h \in \mcV_h$ and
let the scaling functions~\mbox{$\phi_\beta$} and $\phi_p$ be defined as in
\eqref{eq:cip-s_scalings}.
For a piecewise constant approximation \mbox{$\bfbeta_h\in[\mcX_h^{\mathrm{dc},0}]^d$} of~$\bfbeta$ on~$\mcT_h$
which satisfy the approximation properties specified in \eqref{eq:beta-approximation},
the following estimate for the streamline diffusion norm holds
  \begin{align}
    \|
    {\phi}_{\beta}^{\onehalf}
    (\bfbeta_h - \bfbeta) \cdot \nabla \bfu_h
    \|_{\Oast_h}^2
    &\lesssim
    \omega_h
    \bigl(
    \| \mu^{\onehalf}\nabla \bfu_h \|_{\Oast_h}^2
    +
    \| \sigma^{\onehalf} \bfu_h \|_{\Oast_h}^2
    \bigr).
    \label{eq:norm_rsb_beta_diff_const}
\end{align}
The mixed advective-pressure-gradient semi-norm can be estimated as
\begin{align}
    \| 
    \phi_{\beta}^{\onehalf} 
    (
    \bfbeta_h \cdot \nabla \bfu_h + \nabla p_h
    )
    \|_{\Oast_h}^2
    &\lesssim
    \| 
    \phi_{\beta}^{\onehalf} 
    (
    \bfbeta \cdot \nabla \bfu_h + \nabla p_h
    )
    \|_{\Omega}^2
    + 
    g_{\beta}(\bfu_h, \bfu_h)
    +
    g_{p}(p_h, p_h)
    + 
    \omega_h
    \bigl(
    \| \mu^{\onehalf}\nabla \bfu_h \|_{\Oast_h}^2
    +
    \| \sigma^{\onehalf} \bfu_h \|_{\Oast_h}^2
    \bigr),
    \label{eq:norm_equivalence_rsb_beta}
  \end{align}
with the non-dimensional scaling function $\omega_h$ from \eqref{eq:phi_p_definition}.
\end{lemma}
\begin{proof}~\\
\noindent{\bf Estimates~\eqref{eq:norm_rsb_beta_diff_const}.}
A simple application of the Cauchy-Schwarz inequality shows that
\begin{align}
  \|
  {\phi}_{\beta}^{\onehalf}
  (\bfbeta_h - \bfbeta) \cdot \nabla \bfu_h\|^2_{\Oast_h}
  \lesssim
  \sum\nolimits_{T \in \mcT_h}
  \phi_{\beta}
  \|\bfbeta_h - \bfbeta\|_{0,\infty,T}^2 
  \| \nabla \bfu_h\|^2_{T}.
\end{align}
Now using the interpolation property of~$\bfbeta_h$ \eqref{eq:beta-approximation} and
the simple fact that \mbox{$\phi_{\beta} \|\bfbeta\|_{0,\infty,T} \lesssim h$} (see definition \eqref{eq:cip-s_scalings}),
it can be further estimated that
\begin{align}
  \phi_{\beta} \|\bfbeta_h - \bfbeta\|_{0,\infty,T}^2 
  &\lesssim
  \phi_{\beta} 
  \|\bfbeta\|_{0,\infty,T}
  h
  |\bfbeta|_{1,\infty,T}
  \lesssim
  h^2
  |\bfbeta|_{1,\infty,T}.
\end{align}
Then we apply the inverse inequality~\eqref{eq:inverse-estimates_T} to obtain the following simple estimate for
$\|\nabla\bfu_h\|_{T}$,
\begin{align}
  \|\nabla \bfu_h \|_{T}
  =\dfrac{\mu + \sigma h^2}{\mu + \sigma h^2} 
  \| \nabla \bfu_h\|_T^2
  \lesssim \dfrac{1}{\mu + \sigma h^2} 
  ( \| \mu^{\onehalf} \nabla \bfu_h\|_T^2
  + \| \sigma^{\onehalf} \bfu_h \|_T^2
  ),
\end{align}
which gives the desired estimate by taking the maximum over all elements and by defining $\omega_h$ as in \eqref{eq:phi_p_definition}.
\smallskip

\noindent{\bf Estimate~\eqref{eq:norm_equivalence_rsb_beta}.}
  Note that the function \mbox{$\bfv_h := \phi_{\beta}(\bfbeta_h \cdot \nabla \bfu_h + \nabla  p_h)$}
  is a piecewise polynomial function of order \mbox{$k-1$} owing to the fact that
  $\phi_\beta$ and $\bfbeta_h$ are piecewise constant.
  Applying the norm equivalence from Lemma~\ref{lem:norm_equivalence}
  to~$\bfv_h$ and using that \mbox{$\phi_{\beta} \sim \phi_{p}$}, as assumed in \eqref{eq:cip-s_scalings},
  and the local comparability of $\phi$ \eqref{eq:local_equivalence_stab_param_patch} yields
\begin{align}
  &\|
  \phi_{\beta}^{\onehalf}
  (\bfbeta_h \cdot \nabla \bfu_h + \nabla p_h)
  \|^2_{\Oast_h} \nonumber \\
  &\quad\lesssim
  \|
  \phi_{\beta}^{\onehalf}
  (\bfbeta_h \cdot \nabla \bfu_h + \nabla p_h)
  \|^2_{\Omega}
  +
  \sum\nolimits_{j=0}^{k-1}
  \phi_{\beta,F} h^{2j+1}
  \| 
  \jump{\nablan^j(
  \bfbeta_h \cdot \nabla \bfu_h + \nabla p_h
  )}
  \|_{\Fast}^2
  \\
  &\quad\lesssim
  \|
  \phi_{\beta}^{\onehalf}
  (\bfbeta_h \cdot \nabla \bfu_h + \nabla p_h)
  \|^2_{\Omega}
  +
  \sum\nolimits_{j=0}^{k-1}
  \phi_{\beta,F} h^{2j+1}
  \| 
  \jump{
    \nablan^j \nabla p_h
  }
  \|_{\Fast}^2
  +
  \sum\nolimits_{j=0}^{k-1}
  \phi_{\beta,F} h^{2j+1}
  \| 
  \jump{
  \bfbeta_h \cdot \nabla \nablan^j \bfu_h 
  }
  \|_{\Fast}^2
\\
  &\quad\lesssim
  \|
  \phi_{\beta}^{\onehalf}
  (\bfbeta_h \cdot \nabla \bfu_h + \nabla p_h)
  \|^2_{\Omega}
  +
  \sum\nolimits_{j=0}^{k-1}
  \phi_{\beta,F} h^{2j+1}
  \| 
  \jump{
    \nablan^j \nabla p_h
  }
  \|_{\Fast}^2
  \nonumber
  \\
  &\quad\qquad
  +
  \sum\nolimits_{j=0}^{k-1}
  \phi_{\beta,F} h^{2j+1}
  \| 
  \jump{
  (\bfbeta_h -\bfbeta)\cdot \nabla \nablan^j \bfu_h 
  }
  \|_{\Fast}^2
  +
  \sum\nolimits_{j=0}^{k-1}
  \phi_{\beta,F} h^{2j+1}
  \| 
  \jump{
  \bfbeta \cdot \nabla \nablan^j \bfu_h 
  }
  \|_{\Fast}^2
  \\
  &\quad\lesssim
  \| \phi_{\beta}^{\onehalf}
  (\bfbeta_h \cdot \nabla \bfu_h + \nabla p_h)
  \|^2_{\Omega}
  +
  g_{\beta}(\bfu_h, \bfu_h)
  + 
  g_{p}(p_h, p_h)
+ 
    \omega_h
    \bigl(
    \| \mu^{\onehalf}\nabla \bfu_h \|_{\Oast_h}^2
    +
    \| \sigma^{\onehalf} \bfu_h \|_{\Oast_h}^2
    \bigr).
\end{align}
Here, in the last step, the gradient
\mbox{$\nabla p_h = (\nablan p_h) \bfn_F + \boldsymbol{P}_{F} \nabla p_h$}
has been decomposed into its normal and tangential gradient part
using the tangential projection
\mbox{$\boldsymbol{P}_F := \bfI - \bfn_F \otimes \bfn_F$}
and the inverse estimate
\begin{align}
\|\jump{\boldsymbol{P}_{F} \nabla \nablan^j p_h} \|_F^2 =
\|\boldsymbol{P}_{F} \nabla \jump{\nablan^j p_h} \|_F^2 \lesssim
  h^{-2} \|\jump{\nablan^j p_h} \|_F^2
  \label{eq:face-tangential-inverse-eq}
\end{align}
has been employed on the tangential part to obtain 
\begin{align}
  h^{2j+1}
  \|
  \jump{\nablan^j \nabla p_h}
  \|_F^2
  \lesssim
  h^{2j+1}
  \|
  \jump{\nablan^{j+1} p_h}
  \|_F^2
  + 
  h^{2j-1}
  \|
  \jump{\nablan^{j} p_h}
  \|_F^2.
  \label{eq:gp_simple}
\end{align}
Choosing the stabilization parameters \mbox{$\gamma_\beta,\gamma_p>0$}
strictly positive
allows to control the facet terms in the vicinity of~$\Gamma$ by the
two higher-order ghost penalty terms~$g_\beta$ and $g_p$.
For the facet term which includes the difference $\bfbeta_h-\bfbeta$,
applying Cauchy Schwarz on the facet followed by the interpolation estimate \eqref{eq:beta-approximation}
and standard inverse estimates for the remaining normal derivatives of $\bfu_h$,
which completes the proof of \eqref{eq:norm_equivalence_rsb_beta}.
\end{proof}
\\
The previous lemma in combination with Corollary~\ref{cor:norm_equivalence_visc_reac_pres}
explains how the ghost penalty terms help us to extend the
mixed advective-pressure-gradient semi-norm from the physical
domain to the entire active background mesh:
\begin{corollary}
  \label{cor:norm-extension-rbs-norms}
  Under the assumption of Lemma~\ref{lem:norm-extension-rbs-norms}
  it holds that
\begin{align}
  \|
  \phi_{\beta}^{\onehalf}
  (\bfbeta \cdot \nabla \bfu_h + \nabla p_h)
  \|^2_{\Oast_h}
  \lesssim
  \|
  \phi_{\beta}^{\onehalf}
  (\bfbeta \cdot \nabla \bfu_h + \nabla p_h)
  \|^2_{\Omega}
  +
    \omega_h | U_h |_h^2.
\end{align}
\end{corollary}

Finally, following up on Remark~\ref{rem:gbeta-simple},
we present a short proof showing how the
convective and divergence related ghost
penalties~(\ref{eq:ghost-penalty-beta}) and~(\ref{eq:ghost-penalty-u})
can be simplified and replaced by the simple ghost-penalty form $\overline{g}_\beta$~(\ref{eq:gbeta-simple-def}).
\begin{lemma}
  \label{lem:gbeta-simple}
  Let $\bfu_h \in \mcV_h$ and define
  $\overline{\phi}_{\beta} := \|\bfbeta \|_{0,\infty, F}^2 \phi_{\beta}$.
  Then  
  \begin{align}
  \label{eq:gbeta-simple-est}
    g_{\beta}(\bfu_h, \bfu_h)
    &\lesssim
      \gamma_\beta\sum_{F \in \Fast} \sum_{1\leqslant j\leqslant k}
    \overline{\phi}_{\beta} h^{2j -1} (\jump{\nablan^j \bfu_h}, \jump{\nablan^j \bfu_h})_{F},
    \\
  \label{eq:gu-simple-est}
    g_{u}(\bfu_h, \bfu_h)
    &\lesssim
      \gamma_u\sum_{F \in \Fast} \sum_{1\leqslant j\leqslant k}
    \overline{\phi}_{\beta} h^{2j -1} (\jump{\nablan^j \bfu_h}, \jump{\nablan^j \bfv_h})_{F}
      + g_{\sigma}(\bfu_h,\bfu_h)
      + g_{\mu}(\bfu_h,\bfu_h).
  \end{align}
  Similar estimates hold for $s_{\beta}$ and $s_{u}$.
\end{lemma}
\begin{proof}
  Similar to the derivation of~\eqref{eq:gp_simple}, the first
  estimate~(\ref{eq:gbeta-simple-est}) follows from decomposing the higher-order stream-line
  derivative into a face-normal and face-tangential part.
While the latter part vanishes for $j=0$, for the higher order contributions $1\leqslant j < k-1$
a face-based inverse estimate can be applied:
  \begin{align}
    \phi_{\beta}h^{2j+1}
    \|\jump{\bfbeta \cdot \nabla \nablan^j \bfu_h}\|_F^2
    &\lesssim
      \phi_{\beta}h^{2j+1}
      \left(
    \| \bfbeta \cdot \bfn \|_{0,\infty,F}^2
    \|\jump{\nablan^{j+1}\bfu_h}\|_{F}^2
    +
    \| \bfP_F \bfbeta \|_{0,\infty,F}^2
      \|\jump{ \bfP_F\nabla \nablan^j\bfu_h}\|_{F}^2
      \right)
      \\
    &\lesssim
    \phi_{\beta}
      \|\bfbeta\|_{0,\infty,F}^2
     h^{2j+1} 
      \left(
      \|\jump{\nablan^{j+1}\bfu_h}\|_{F}^2
    +
    h^{-2}
    \|\jump{\nablan^j\bfu_h}\|_{F}^2
      \right).
  \end{align}
  Note that the last estimate introduces order-preserving crosswind diffusion also for the summand $j=1$ of \eqref{eq:gbeta-simple-est}.
  The second estimate~\eqref{eq:gu-simple-est}
  follows after similar calculations directly from definition~(\ref{eq:cip-s_scalings}) of $\phi_u$
for which holds $\phi_u \lesssim \|\bfbeta \|_{0,\infty, F}^2 \phi_{\beta} + \sigma h^2 + \mu$.
\end{proof}

\section{Stability Properties}
\label{sec:stability-properties}
In this section, we start with the numerical analysis of the proposed
cut finite element method \eqref{eq:oseen-discrete-unfitted} by
proving that the total bilinear form $A_h + S_h + G_h$ satisfies an inf-sup
condition on $\mcW_h$ with respect to a suitable energy norm 
with the inf-sup constant being independent of how the boundary cuts the
underlying background mesh.
The proof of the inf-sup stability is split into three major steps.
First, we use the coercivity from Lemma~\ref{lem:coercivity_ah} in a semi-norm
\begin{align}
  |U_h|_h^2 = |(\bfu_h, p_h)|_h^2 = \tn \bfu_h \tn_h^2 + |p_h|^2_h,
  \qquad \text{with } \quad | p_h |^2_h = s_p(p_h, p_h) + g_p(p_h,p_h).
  \label{eq:Ah-semi-norm-recalled}
\end{align}
In a second step, we gain control over additional terms,
that are
\begin{equation}
\| \phi_{u}^{\onehalf} \nabla\cdot \bfu_h \|_{\Omega}, \quad
\| \phi_{\beta}^{\onehalf}(\bfbeta \cdot \nabla \bfu_h + \nabla p_h)\|_{\Omega}, \quad \text{and} \quad
\Phi_p^{\onehalf}\|p_h\|_{\Omega},
\end{equation}
proven in Lemma~\ref{lem:divergence-stability},~\ref{lem:convect-pressure-stab} and \ref{lem:pressure-stability}.
Finally, the previous two parts are combined to obtain
the desired inf-sup stability with respect to the full energy norm $\tnorm{(\bfu_h,p_h)}_{h}$.
We begin by showing that the total bilinear form $A_h + S_h + G_h$
is coercive on $\mcW_h$ with respect to the semi-norm $|U_h|_h^2$.
More precisely, we have the following lemma.
\begin{lemma}
For $U_h = (\bfu_h,p_h) \in \mcW_h$ the coercivity estimate 
  \begin{equation}
    | U_h |_h^2
   \lesssim
   A_h(U_h,U_h) + S_h(U_h,U_h) + G_h(U_h,U_h) 
   \label{eq:coercivity_ah}
  \end{equation}
  holds whenever the stability parameters $\gamma,
  \gamma_\mu,\gamma_\sigma,\gamma_\beta,\gamma_u,\gamma_p$ are chosen
  to be
strictly positive
.
  \label{lem:coercivity_ah}
\end{lemma}
\begin{remark}
 \label{rem:hidden-const} 
  Here and in the following estimates,
  the hidden constant dependent only on the dimension $d$,
  the polynomial order $k$, the quasi-uniformness parameters
  and the magnitude of the stability parameters. In particular,
  the hidden constant degenerates as any of the stabilization parameters
  approaches zero.
\end{remark}
\begin{proof}
Starting from the definition of $A_h$, see \eqref{eq:Ah-form-def}, we have
  \begin{align}
   A_h(U_h,U_h)
  &= \normLtwo{\sigmahalf \bfu_h}{\Omega}^2
     + (\bfbeta\cdot\nabla\bfu_h, \bfu_h)_{\Omega}
     - \langle (\bfbeta\cdot\bfn)\bfu_h,  \bfu_h\rangle_{\GammaIn}
     + \| (\gamma (\phi_u/h) )^{\onehalf} \bfu_h \cdot \bfn \|^2_{\Gamma} \nonumber \\
  &  \quad+ \normLtwo{(2\mu)^\onehalf\bfepsilon(\bfu_h)}{\Omega}^2
     - 4\langle \mu\bfepsilon(\bfu_h)\bfn, \bfu_h \rangle_{\Gamma}
     + \| (\gamma(\mu/h))^{\onehalf} \bfu_h\|^2_{\Gamma}. \label{eq:coercivity_ah_Ah}
  \end{align}
  Integration by parts for the advective term together with continuity of $\bfbeta$ yields
  \begin{align}
   (\bfbeta\cdot\nabla\bfu_h, \bfu_h) = \frac{1}{2}\langle
    (\bfbeta\cdot\bfn)\bfu_h,\bfu_h \rangle_{\Gamma} -
    \frac{1}{2}((\div\bfbeta)\bfu_h, \bfu_h)_{\Omega}.
  \end{align}
  Using the assumption $\div\bfbeta = 0$, we can rewrite the advective terms as
  \begin{align}\label{eq:coercivity_adv}
   (\bfbeta\cdot\nabla\bfu_h, \bfu_h)
       - \langle (\bfbeta\cdot\bfn)\bfu_h,  \bfu_h\rangle_{\GammaIn}
   &= \frac{1}{2}\langle(\bfbeta\cdot\bfn)\bfu_h,\bfu_h\rangle_{\Gamma}
       - \langle (\bfbeta\cdot\bfn)\bfu_h,  \bfu_h\rangle_{\GammaIn}
    = \frac{1}{2}\normLtwo{|\bfbeta\cdot\bfn|^\onehalf \bfu_h}{\Gamma}^2.
  \end{align}
  Applying a $\delta$-scaled Cauchy-Schwarz inequality and a trace inequality \eqref{eq:inverse-estimates_Gamma} yields
  \begin{align}
    \langle \mu\bfepsilon(\bfu_h)\bfn, \bfu_h \rangle_{\Gamma}
   & \lesssim \delta 
   \| 
   (\mu h)^\onehalf\bfepsilon(\bfu_h)\bfn
   \|_{\Gamma}^2  
   + \delta^{-1} \| (\mu/h)^{1/2} \bfu_h \|_{\Gamma}^2 
\\
     &\lesssim \delta \| \nuhalf\grad\bfu_h\|_{\Oast_h}^2  + 
     \delta^{-1} \| (\mu/h)^{1/2} \bfu_h\|_{\Gamma}^2  
\\
     &\lesssim \delta 
     \bigl(
     \| \nuhalf\grad\bfu_h\|_{\Omega}^2  + 
     g_{\mu}(\bfu_h, \bfu_h) 
     \bigr)
     +
     \delta^{-1} \| (\mu/h)^{1/2} \bfu_h\|_{\Gamma}^2,
\label{eq:stress-boundary-estimate}
  \end{align}
  where in the last estimate we used the norm equivalence from Corollary~\ref{cor:norm_equivalence_visc_reac_pres}.
  Now observe that thanks to the Nitsche boundary penalty and Korn's inequality shown in~\eqref{eq:korn-III},
the viscous term can be estimated by
\begin{align}
 \|
\mu^\onehalf\bfepsilon(\bfu_h)
\|_{\Omega}^2 
+
\dfrac{1}{2}
 \|(\mu/h)^{1/2} \bfu_h\|_{\Gamma}^2
\gtrsim
\| \mu^{1/2} \nabla \bfu_h \|_{\Omega}^2
\end{align}
which combined with the previous inequality \eqref{eq:stress-boundary-estimate} shows that
  \begin{align}
\label{eq:coercivity_ah_visc}
        \| (2\mu)^\onehalf\bfepsilon(\bfu_h)\|_{\Omega}^2
      &+ g_\mu(\bfu_h,\bfu_h)
      - 4\langle \mu\bfepsilon(\bfu_h)\bfn, \bfu_h \rangle_{\Gamma}
      + \| (\gamma(\mu/h))^{1/2} \bfu_h\|^2_{\Gamma}      
\nonumber \\
      &\gtrsim \normLtwo{\nuhalf\grad\bfu_h}{\Omega}^2
      + \dfrac{1}{2}\| (\gamma(\mu/h))^{1/2} \bfu_h\|^2_{\Gamma}      
        + g_\mu(\bfu_h,\bfu_h)
  \end{align}
  for $\delta>0$ sufficiently small and $\gamma > 0$ large enough.
  The claim follows by combining \eqref{eq:coercivity_ah_Ah}, \eqref{eq:coercivity_adv} and estimate \eqref{eq:coercivity_ah_visc}
and the stabilization terms in $S_h$ and $G_h$.
\end{proof}
\begin{lemma}
  \label{lem:divergence-stability}
There is a constant $c_1 > 0$ such that
for $\bfu_h \in \mcV_h$ there exists a $q_h \in \mcQ_h$ satisfying
 \begin{align}
   - b_h(q_h, \bfu_h) 
   &\gtrsim
   \| \phi_u^{\onehalf} \nabla \cdot \bfu_h \|_{\Omega}^2
   - 
   c_1\Bigl(
   s_u(\bfu_h,\bfu_h)
   + g_u(\bfu_h,\bfu_h)
   + \| h^{-\onehalf}\phi_u^{\onehalf} \bfu_h\cdot \bfn \|_{\Gamma}^2
   \Bigr)
  \label{eq:divergence-stability}
\intertext{and the stability estimate}
  \label{eq:divergence-stability-qh_stab_bound}
    \Phi_p\| q_h\|_{\Omega}^2 + |q_h|_h^2
    &\lesssim 
    \| \phi_u^{\onehalf} \nabla \cdot \bfu_h \|_{\mcT_h}^2
    \lesssim
    \| \phi_u^{\onehalf} \nabla \cdot \bfu_h \|_{\Omega}^2
    + g_u(\bfu_h,\bfu_h)
 \end{align}
  whenever the stability parameters $\gamma,
  \gamma_\mu,\gamma_\sigma,\gamma_\beta,\gamma_u,\gamma_p$ are chosen
  to be
strictly positive
.
\end{lemma}
\begin{proof}
Define $q_h := \mcO_h(\widetilde{\phi}_u \nabla \cdot \bfu_h)$ with $\widetilde{\phi}_u := \mcO_h(\phi_u)$ being a smoothed,
piecewise linear version of $\phi_u$ as defined in
  \eqref{eq:smoothed_stab_param}, satisfying the local comparability  
  $\widetilde{\phi}_u|_T \sim \phi_u|_T$ stated
  in~\eqref{eq:local_equivalence_stab_param_oswald}.
Then
\begin{align}
  -b_h(q_h, \bfu_h) 
 &= \| \widetilde{\phi}_u^{\onehalf} \nabla \cdot \bfu_h\|_{\Omega}^2
+ (\mcO_h(\widetilde{\phi}_u \nabla \cdot \bfu_h) - \widetilde{\phi}_u\nabla \cdot \bfu_h,
\nabla \cdot \bfu_h)_{\Omega}
- (\mcO_h(\widetilde{\phi}_u \nabla \cdot \bfu_h), \bfu_h\cdot \bfn)_{\Gamma}
\\
 &= \| \widetilde{\phi}^{\onehalf}_u \nabla \cdot \bfu_h\|_{\Omega}^2 + I + II.
  \label{eq:inf-sup-b_h-step-I}
\end{align}
By using $\phi_{u,T}^{-1}\widetilde{\phi}_{u,T}^2\lesssim \phi_{u,F}$, now the first term can be treated as follows
\begin{align}
  I 
&=
({\phi}_u^{-\onehalf}(\mcO_h(\widetilde{\phi}_u \nabla \cdot \bfu_h) 
- \widetilde{\phi}_u\nabla \cdot \bfu_h),
{\phi}_u^{\onehalf}\nabla \cdot \bfu_h)_{\Omega}
\\
&\gtrsim
- \delta \|{\phi}_u^{\onehalf}\nabla \cdot \bfu_h\|_{\Omega}^2
  - \delta^{-1} \|{\phi}_u^{-\onehalf}(\mcO_h(\widetilde{\phi}_u \nabla \cdot \bfu_h)
 - \widetilde{\phi}_u\nabla \cdot \bfu_h)\|_{\Omega}^2
\\
&\gtrsim
- \delta \|{\phi}_u^{\onehalf}\nabla \cdot \bfu_h\|_{\Omega}^2
- \delta^{-1} \sum\nolimits_{F\in\mcF_i} \phi_u^{-1} \widetilde{\phi}_u^2
 \| \jump{\nabla \cdot \bfu_h}\|_{F}^2
\\
&\gtrsim
- \delta \|{\phi}_u^{\onehalf}\nabla \cdot \bfu_h\|_{\Omega}^2
-\delta^{-1} s_u(\bfu_h,\bfu_h).
\end{align}
The second one can be dealt with using the Nitsche boundary terms
\begin{align}
  II 
&=
- (h^{\onehalf}\phi_u^{-\onehalf}\mcO_h(\widetilde{\phi}_u \nabla \cdot \bfu_h),h^{-\onehalf}\phi_u^{\onehalf} \bfu_h\cdot \bfn)_{\Gamma}
\\
&\gtrsim
-\delta \|h^{\onehalf}\phi_u^{-\onehalf}\mcO_h(\widetilde{\phi}_u \nabla \cdot \bfu_h)\|_{\Gamma}^2
-\delta^{-1} \|h^{-\onehalf}\phi_u^{\onehalf} \bfu_h\cdot \bfn\|_{\Gamma}^2
\\
&\gtrsim
-\delta \|\phi_u^{-\onehalf}\mcO_h(\widetilde{\phi}_u \nabla \cdot \bfu_h)\|_{\mathcal{T}_h}^2
-\delta^{-1} \|h^{-\onehalf}\phi_u^{\onehalf} \bfu_h\cdot \bfn\|_{\Gamma}^2
\\
&\gtrsim
-\delta \|\phi_u^{\onehalf} \nabla \cdot \bfu_h\|_{\mathcal{T}_h}^2
-\delta^{-1} \|h^{-\onehalf}\phi_u^{\onehalf} \bfu_h\cdot \bfn\|_{\Gamma}^2
\\
&\gtrsim
-\delta (\|\phi_u^{\onehalf} \nabla \cdot \bfu_h\|_{\Omega}^2 + g_u(\bfu_h,\bfu_h) )
-\delta^{-1} \|h^{-\onehalf}\phi_u^{\onehalf} \bfu_h\cdot \bfn\|_{\Gamma}^2
\end{align}
where after applying a $\delta$-scaled Young inequality and the trace inequality \eqref{eq:inverse-estimates_Gamma}
the scaled stability of the Oswald interpolation \eqref{eq:oswald-quasi-local-stability} and the local averaging property $\phi_u^{-\onehalf}\widetilde{\phi}_u \lesssim {\phi}_u^{\onehalf}$
was used, followed by an application of
norm equivalence \eqref{eq:norm_equivalence_rsb_div}  in the final
step.
After inserting the lower bounds for $I$ and $II$
into~\eqref{eq:inf-sup-b_h-step-I} we arrive at
\begin{align}
   - b_h(q_h, \bfu_h) 
   &\gtrsim
   (1-2\delta)\| \phi_u^{\onehalf} \nabla \cdot \bfu_h \|_{\Omega}^2
   - 
   (\delta + \delta^{-1})\Bigl(
   s_u(\bfu_h,\bfu_h)
   + g_u(\bfu_h,\bfu_h)
   + \| h^{-\onehalf}\phi_u^{\onehalf} \bfu_h\cdot \bfn \|_{\Gamma}^2
   \Bigr)
\\
   &\gtrsim
   \| \phi_u^{\onehalf} \nabla \cdot \bfu_h \|_{\Omega}^2
   -
   c_1
   \Bigl(s_u(\bfu_h,\bfu_h)
   + g_u(\bfu_h,\bfu_h)
   + \| h^{-\onehalf}\phi_u^{\onehalf} \bfu_h\cdot \bfn \|_{\Gamma}^2
   \Bigr).
\end{align}
Choosing $\delta>0$ sufficiently small
yields the desired estimate~\eqref{eq:divergence-stability} 
for some constant $c_1>0$.
Finally, the stability bound can be easily proven by observing that
\begin{align}
 \Phi_p\|q_h\|_\Omega^2 + |q_h|_h^2
 &\lesssim
 \Phi_p \|\mcO_h(\widetilde{\phi}_u \nabla \cdot \bfu_h)\|_{\Omega}^2
 +\sum_{F\in\mcF_i}{\sum_{j=0}^{k-1}{\phi_{p,F} h^{2j+1} \| \jump{\partial_{\bfn}^j\mcO_h(\widetilde{\phi}_u \nabla \cdot \bfu_h)} \|_F^2}} \\
 &\lesssim
 \Phi_p \|\widetilde{\phi}_u \nabla \cdot \bfu_h\|_{\mathcal{T}_h}^2
 +\sum_{T\in\mathcal{T}_h}{\phi_{p,T} \|\mcO_h( \widetilde{\phi}_u  \nabla \cdot \bfu_h )\|_T^2} \\
 &\lesssim
 \sum_{T\in\mathcal{T}_h}{\| \phi_u^{\onehalf} \nabla \cdot \bfu_h \|_T^2}
 \lesssim
 \| \phi_u^{\onehalf} \nabla \cdot \bfu_h \|_{\Omega}^2 + g_u(\bfu_h,\bfu_h)
\end{align}
where the trace inequality \eqref{eq:inverse-estimates_F}, the stability of the Oswald interpolant \eqref{eq:oswald-quasi-local-stability} and the averaging property $ \phi_p \widetilde{\phi}_u^2 \lesssim \phi_u$
were used. The last step results from applying the norm equivalence \eqref{eq:norm_equivalence_rsb_div} from Corollary~\ref{cor:norm_equivalence_rsb_pressure_div}.
\end{proof}

The next lemma shows how additional control over a semi-norm of the
form $\| \phi_{\beta}^{\onehalf}(\bfbeta \cdot \nabla \bfu_h + \nabla
p_h)\|_{\Omega}^2$ can be recovered, which is closely related to the
well-known mixed norm control given by residual-based stabilized
SUPG/PSPG formulations, see~\cite{LubeRapin2006,BraackBurmanJohnEtAl2007,MatthiesIonkinLubeEtAl2009}.

\begin{lemma}
  \label{lem:convect-pressure-stab}
There exist a constant $c_2 > 0$ such that for $U_h=(\bfu_h,p_h)\in \mcW_h$
there is a $\bfv_h \in
 V_h$ satisfying
 \begin{align}
   (\bfbeta \cdot \nabla \bfu_h + \nabla p_h, \bfv_h)_{\Omega}
   &\gtrsim
   \| \phi_{\beta}^{\onehalf}(\bfbeta \cdot \nabla \bfu_h + \nabla p_h)\|_{\Omega}^2
   - c_2 (1 + \omega_h) | U_h |_h^2
  \label{eq:convect-pressure-stab}
\intertext{and the stability estimate}
   \tn \bfv_h \tn_h^2
+ \| \phi_{\beta}^{\onehalf}\bfbeta \cdot \nabla \bfv_h\|_{\Omega}^2
+ \| \phi_u^{\onehalf} \nabla \cdot \bfv_h \|_{\Omega}^2
   &\lesssim
   \| \phi_{\beta}^{\onehalf}(\bfbeta \cdot \nabla \bfu_h + \nabla p_h)\|_{\Omega}^2
   + (1 + \omega_h) | U_h |_h^2,
   \label{eq:convect-pressure-stab-vh_stab_bound_2}
 \end{align}
  whenever the stability parameters $\gamma,
  \gamma_\mu,\gamma_\sigma,\gamma_\beta,\gamma_u,\gamma_p$ are chosen
  to be
strictly positive
.
\end{lemma}
\begin{proof}
To gain control of the
$ \| \phi_{\beta}^{\onehalf}(\bfbeta \cdot \nabla \bfu_h + \nabla p_h)\|_{\Omega}^2$ term,
we need to construct a suitable test function~$\bfv_h$.
First, we introduce 
an elementwise constant, vector-valued function $\bfbeta_h$
which satisfies 
the approximation property and stability bound
from~(\ref{eq:beta-approximation}). The simplest choice is
to take the value of $\bfbeta$ at some point of $T$ for each $T \in \mcT_h$.
Set $\bfw_h := \bfbeta_h \cdot \nabla \bfu_h + \nabla p_h$
and introduce the smoothed, piecewise linear stabilization parameter
$\widetilde{\phi}_{\beta} = \mcO_h(\phi_{\beta})$
to finally define the test function $\bfv_h:= \mcO_h(\widetilde{\phi}_{\beta} \bfw_h)$.
Now using the local comparability~(\ref{eq:local_equivalence_stab_param_oswald})
of $\phi_{\beta}$ and $\widetilde{\phi}_{\beta}$
we observe that
\begin{align}
  (\bfbeta \cdot \nabla \bfu_h + \nabla p_h, \bfv_h)
  &=
  \|
  \widetilde{\phi}_{\beta}^{\onehalf}
  (\bfbeta \cdot \nabla \bfu_h + \nabla p_h)
  \|_{\Omega}^2
    +
    \underbrace{
  (
  \bfbeta \cdot \nabla \bfu_h + \nabla p_h,
  \mcO_h(\widetilde{\phi}_{\beta} \bfw_h)
  -
  \widetilde{\phi}_{\beta} \bfw_h
    )_{\Omega}
    }_{I}
  \\
  &\quad 
    +
   \underbrace{(
  \bfbeta \cdot \nabla \bfu_h + \nabla p_h,
  \widetilde{\phi}_{\beta} (\bfbeta_h - \bfbeta)\cdot \nabla \bfu_h
  )_{\Omega}}_{II}
  \\
  &\sim
  \|
  {\phi}_{\beta}^{\onehalf}
  (\bfbeta \cdot \nabla \bfu_h + \nabla p_h)
  \|_{\Omega}^2
  +
  I + II,
\end{align}
leaving us with the remainder terms $I$ and $II$ which we estimate next.
\\
\noindent{\bf Term $I$.}
From successively applying
a $\delta$-Cauchy-Schwarz inequality,
Lemma~\ref{lem:Oswald-interpolant} to estimate the difference
$\mcO_h(\widetilde{\phi}_{\beta} \bfw_h) - \widetilde{\phi}_{\beta}
\bfw_h$, the local comparability 
$\widetilde{\phi}_{\beta,T} \sim {\phi}_{\beta,T}$,
and finally, a Cauchy-Schwarz inequality,
we deduce that
\begin{align}
  I &=
  (
  {\phi}_{\beta}^{\onehalf}
  (\bfbeta \cdot \nabla \bfu_h + \nabla p_h),
   \phi_{\beta}^{-\onehalf}
   (
   \mcO_h(\widetilde{\phi}_{\beta} \bfw_h)
  -
  \widetilde{\phi}_{\beta} \bfw_h
  ))_{\Omega}
  \\
  &\lesssim
  \delta 
  \|
  {\phi}_{\beta}^{\onehalf}
  (\bfbeta \cdot \nabla \bfu_h + \nabla p_h)
  \|_{\Omega}^2
  +
  \delta^{-1}
  \sum_{T \in \mcT_h}
  \|\phi_{\beta,T}^{-\onehalf} 
  (\mcO_h(\widetilde{\phi}_{\beta} \bfw_h)
  -
  \widetilde{\phi}_{\beta} \bfw_h)
  \|_{T \cap \Omega}^2
  \\
  &\lesssim
  \delta 
  \|
  {\phi}_{\beta}^{\onehalf}
  (\bfbeta \cdot \nabla \bfu_h + \nabla p_h)
  \|_{\Omega}^2
  +
  \delta^{-1}
  \sum_{T \in \mcT_h}
  \phi_{\beta,T}^{-1} h
  \| 
  \jump{
  \widetilde{\phi}_{\beta} \bfw_h
  }
  \|_{\mcF_i(T)}^2
  \\
  &\lesssim
  \delta 
  \|
  {\phi}_{\beta}^{\onehalf}
  (\bfbeta \cdot \nabla \bfu_h + \nabla p_h)
  \|_{\Omega}^2
  +
  \delta^{-1}
  \sum_{T \in \mcT_h}
  \phi_{\beta,T}^{-1}\widetilde{\phi}_{\beta,F}^2 h
  \| 
  \jump{
  \bfw_h
  }
  \|_{\mcF_i(T)}^2
  \\
  &\lesssim
  \delta 
  \|
  {\phi}_{\beta}^{\onehalf}
  (\bfbeta \cdot \nabla \bfu_h + \nabla p_h)
  \|_{\Omega}^2
  +
  \delta^{-1}
  \sum_{T \in \mcT_h}
  \phi_{\beta,T} h
  \| 
  \jump{
    \bfbeta \cdot \nabla \bfu_h + \nabla p_h
  }
  \|_{\mcF_i(T)}^2
    \nonumber
\\
  &\phantom{\lesssim}\quad
  + \delta^{-1}\sum_{T \in \mcT_h}
  \phi_{\beta,T} h
  \| 
  \jump{
    (\bfbeta_h - \bfbeta) \cdot \nabla \bfu_h
    }
  \|_{\mcF_i(T)}^2
   \label{eq:supg-pspg-stability-term-I} 
  \\
  &\lesssim
  \delta 
  \|
  {\phi}_{\beta}^{\onehalf}
  (\bfbeta \cdot \nabla \bfu_h + \nabla p_h)
  \|_{\Omega}^2
  +
  \delta^{-1}
  (
  s_{\beta}(\bfu_h, \bfu_h)
  + 
  s_{p}(p_h, p_h)
  +
  \omega_h \tn \bfu_h \tn_h^2
  ),
\end{align}
where in the last step, a combination of the approximation property \eqref{eq:beta-approximation}
and the inverse estimate~(\ref{eq:inverse-estimates_F})
was used to obtain
\begin{align}
  \sum_{T \in \mcT_h}
  \phi_{\beta,T} h
  \| 
  \jump{
    (\bfbeta_h - \bfbeta) \cdot \nabla \bfu_h
    }
  \|_{\mcF_i(T)}^2
\lesssim 
  \omega_h \tn \bfu_h \tn_h^2.
\end{align}
Moreover, splitting the facet terms
\begin{equation}
\|\jump{\bfbeta \cdot \nabla \bfu_h + \nabla p_h}\|_{\mcF_i(T)}^2 
  \lesssim \|\jump{\bfbeta \cdot \nabla \bfu_h} \|_{\mcF_i(T)}^2 
  + \| \jump{\nabla p_h}\|_{\mcF_i(T)}^2
\end{equation}
and using $\phi_{\beta}\sim \phi_p$ \eqref{eq:cip-s_scalings},
these can be bounded by the CIP stabilization operators $s_\beta$ and $s_p$
defined in \eqref{eq:cip-s_beta} and \eqref{eq:cip-s_p}.
\\
\noindent{\bf Term $II$.}
A simple application of a $\delta$-Cauchy-Schwarz inequality and
estimates~\eqref{eq:norm_rsb_beta_diff_const}--\eqref{eq:norm_rsb_beta_diff_const} yields
\begin{align}
  II 
  &\lesssim
  \delta \| 
  {\phi}_{\beta}^{\onehalf}
  (\bfbeta \cdot \nabla \bfu_h + \nabla p_h)
  \|_{\Omega}^2
  +
  \delta^{-1}
  \|
  {\phi}_{\beta}^{\onehalf}
  (\bfbeta_h - \bfbeta) \cdot \nabla \bfu_h\|^2_{\Omega}
  \\
  &\lesssim
  \delta \| 
  {\phi}_{\beta}^{\onehalf}
  (\bfbeta \cdot \nabla \bfu_h + \nabla p_h)
  \|_{\Omega}^2
  +
  \delta^{-1}\omega_h \tn \bfu_h \tn_h^2.
\end{align}
\noindent{\bf Estimate of \eqref{eq:convect-pressure-stab}.}
Now choose $\delta>0$ small enough and combine the estimates for Term $I$
and $II$ to conclude that
\begin{align}
  (\bfbeta \cdot \nabla \bfu_h + \nabla p_h, \bfv_h)
  &\gtrsim
  \|
  {\phi}_{\beta}^{\onehalf}
  (\bfbeta \cdot \nabla \bfu_h + \nabla p_h)
  \|_{\Omega}^2
  - |I| 
  - |II|
  \\
  &\gtrsim
  \|
  {\phi}_{\beta}^{\onehalf}
  (\bfbeta \cdot \nabla \bfu_h + \nabla p_h)
  \|_{\Omega}^2
  - c_1 
  \bigl(s_{\beta}(\bfu_h, \bfu_h) + s_p(p_h, p_h)
  + \omega_h \tn\bfu_h\tn_h^2
  \bigr)
  \\
  &\gtrsim
  \|
  {\phi}_{\beta}^{\onehalf}
  (\bfbeta \cdot \nabla \bfu_h + \nabla p_h)
  \|_{\Omega}^2
  - c_2(1 + \omega_h) | U_h |_{h}^2
\end{align}
for some constant $c_2>0$.\\
{\bf Estimate of $\tn \bfv_h \tn_h$.}
Throughout the next steps, we will make heavy use of
the fact that
\begin{align}
(\mu h^{-2} + \sigma + \|\bfbeta\|_{0,\infty, T}h^{-1}) \phi_{\beta}
\sim
\phi_{u}h^{-2}
\phi_{\beta}
\lesssim 1
\label{eq:phi-beta-bounds}
\end{align}
by the very definition of $\phi_{\beta}$ and $\phi_u$.
We start with the viscous and reaction terms from norm definition~\eqref{eq:oseen-norm-u}.
Then
\begin{align}
  \label{eq:bound-viscous-reaction-terms-for-vh3-I}
  \mu \| \nabla \bfv_h \|_{\mcT_h}^2
  +
  \sigma \| \bfv_h \|_{\mcT_h}^2
  &\lesssim 
  (\mu h^{-2} + \sigma)\| \widetilde{\phi}_{\beta}
  (\bfbeta_h \cdot \nabla \bfu_h + \nabla p_h)\|_{\mcT_h}^2
  \\
  &\lesssim
  \underbrace{
    (\mu h^{-2} + \sigma)
  {\phi_{\beta}}
  }_{\lesssim 1}
  \| 
  {\phi_{\beta}}^{\onehalf}
  (\bfbeta_h \cdot \nabla \bfu_h + \nabla p_h)\|_{\mcT_h}^2.
\end{align}
Turning to the boundary terms appearing in $\tn \cdot
\tn_{h}$, the convective boundary part is bounded by
\begin{align}
  \| | \bfbeta \cdot \bfn |^{\onehalf} \bfv_h \|_{\Gamma}^2
  &\lesssim
  \sum_{T\in\mcT_h}
  \|  \bfbeta \|_{0,\infty,T} 
  \|\widetilde{\phi}_{\beta}   
  (\bfbeta_h \cdot \nabla \bfu_h + \nabla p_h)\|_{T\cap \Gamma}^2
  \\
  &\lesssim
  \sum_{T\in\mcT_h}
  \underbrace{
  \|  \bfbeta \|_{0,\infty,T} h^{-1} \phi_{\beta}
}_{\lesssim 1}
  \|{\phi}_{\beta}^{\onehalf}
  (\bfbeta_h \cdot \nabla \bfu_h + \nabla p_h)\|_{T}^2,
  \label{eq:bound-conv-boundary-term-for-vh3-I}
\end{align}
where the trace inequality~(\ref{eq:inverse-estimates_Gamma})
was used to pass from $T\cap \Gamma$ to $T$.
The remaining boundary terms can be similarly bounded:
\begin{align}
  \mu \| h^{-\onehalf} \bfv_h \|_{\Gamma}^2
  +
  \| \phi_u^{\onehalf} h^{-\onehalf} \bfv_h\cdot \bfn \|_{\Gamma}^2
  &\lesssim
  \sum_{T\in\mcT_h}
  \underbrace{
  (\mu + \phi_u) h^{-2}
  {\phi}_{\beta}
}_{\lesssim 1}
  \| 
  {\phi}_{\beta}^{\onehalf}
  (\bfbeta_h \cdot \nabla \bfu_h + \nabla p_h)
    \|_{T}^2.
    \label{eq:conv-pressure-stab-stab-est-step-1}
\end{align}
Next, we need to estimate the velocity related norm terms contributed from the stabilization
operators $S_h$ and $G_h$. A bound for $s_{\beta}(\bfv_h, \bfv_h)$ and $s_{u}(\bfv_h, \bfv_h)$
can be derived by first employing the inverse
inequalities~(\ref{eq:inverse-estimates_F})
and~(\ref{eq:inverse-estimates_T}) and then recalling the definition of $v_h$ and
estimate~(\ref{eq:phi-beta-bounds}):
\begin{align}
  s_{\beta}(\bfv_h, \bfv_h)
  &=
  \sum_{F\in\mcF_i}
  \phi_{\beta,F} h \|\jump{\bfbeta_h \cdot \nabla\bfv_h}\|_F^2
  \lesssim
\sum_{T\in\mcT_h}
\|  \bfbeta \|_{0,\infty,T}^2 h^{-2} \phi_{\beta,T}
  \|\bfv_h\|_{T}^2
  \\
  &\lesssim
\sum_{T\in\mcT_h}
\underbrace{
\|  \bfbeta \|_{0,\infty,T}^2 h^{-2} \phi_{\beta,T}^2
}_{\lesssim 1}
  \|{\phi}_{\beta}^{\onehalf}
  (\bfbeta_h \cdot \nabla \bfu_h + \nabla p_h)\|_{T}^2
  \\
  s_{u}(\bfv_h, \bfv_h)
  &=
  \sum_{F\in\mcF_i}
  \phi_{u,F} h \|\jump{\nabla \cdot \bfv_h}\|_F^2
  \lesssim
\sum_{T\in\mcT_h}
\phi_{u,T} h^{-2} \| \bfv_h \|_T^2
\\
&\lesssim
\sum_{T\in\mcT_h}
\underbrace{
\phi_{u,T} h^{-2} \phi_{\beta,T} 
}_{\lesssim 1}
  \|{\phi}_{\beta}^{\onehalf}
  (\bfbeta_h \cdot \nabla \bfu_h + \nabla p_h)\|_{T}^2.
\end{align}
The corresponding ghost-penalty terms  
can be estimated in the exact same manner, yielding
\begin{align}
  g_{\beta}(\bfv_h, \bfv_h)
  +
  g_{u}(\bfv_h, \bfv_h)
  \lesssim
  \|{\phi}_{\beta}^{\onehalf}
  (\bfbeta_h \cdot \nabla \bfu_h + \nabla p_h)\|_{\mcT_h}^2.
\end{align}
Finally, another application of the inverse estimate~(\ref{eq:inverse-estimates_F})
in combination with the already established bound~\eqref{eq:bound-viscous-reaction-terms-for-vh3-I}
for the viscous and reaction norm terms gives
\begin{align}
  g_{\mu}(\bfv_h, \bfv_h)
  +
  g_{\sigma}(\bfv_h, \bfv_h)
  &\lesssim
\mu \|\nabla \bfv_h \|_{\mcT_h}^2
+
\sigma \| \bfv_h \|_{\mcT_h}^2
  \lesssim
  \|{\phi}_{\beta}^{\onehalf}
  (\bfbeta_h \cdot \nabla \bfu_h + \nabla p_h)\|_{\mcT_h}^2.
\end{align}
\noindent{\bf Estimate of $\| \phi_{\beta}^{\onehalf} \bfbeta \cdot \nabla \bfv_h \|_{\Omega}$.}
Similarly, the streamline-diffusion term in \eqref{eq:convect-pressure-stab-vh_stab_bound_2} can be bounded by
\begin{align}
  \| \phi_{\beta}^{\onehalf} \bfbeta \cdot \nabla \bfv_h \|^2_{\Omega}
  &\lesssim
    \sum_{T\in\mcT_h}
    \|  \bfbeta \|_{0,\infty,T}^2 h^{-2} \phi_{\beta}
    \|\bfv_h\|_{T}^2
\lesssim
\sum_{T\in\mcT_h}
\underbrace{
\|  \bfbeta \|_{0,\infty,T}^2 h^{-2} \phi_{\beta}^2
}_{\lesssim 1}
  \|{\phi}_{\beta}^{\onehalf}
  (\bfbeta_h \cdot \nabla \bfu_h + \nabla p_h)\|_{T}^2.
\end{align}
\noindent{\bf Estimate of $\|\phi_u^{\onehalf} \nabla \cdot \bfv_h \|_{\Omega}$.}
Combining an inverse inequality with the stability of the Oswald
interpolant,
the incompressibility term can be estimated as follows:
\begin{align}
  \|\phi_u^{\onehalf} \nabla \cdot \bfv_h \|_{\Omega}^2
&
=
  \|\phi_u^{\onehalf} \nabla \cdot  \mcO_h(\widetilde{\phi}_{\beta} ((\bfbeta_h \cdot \nabla) \bfu_h + \nabla p_h)) \|_{\Omega}^2 
  \lesssim
  \| 
  \underbrace{
  (\phi_u^{\onehalf} h^{-1} \phi_{\beta}^\onehalf)
}_{\lesssim 1}
  \phi_{\beta}^{\onehalf}
  (\bfbeta_h \cdot \nabla \bfu_h + \nabla p_h) \|_{\mcT_h}^2.
\end{align}

\noindent{\bf Estimate of \eqref{eq:convect-pressure-stab-vh_stab_bound_2}.}
After collecting all terms and employing estimate~(\ref{eq:norm_equivalence_rsb_beta}),
we arrive at the desired stability bound:
\begin{align}
 & \tn \bfv_h \tn_h^2
+ \| \phi_{\beta}^{\onehalf} \bfbeta \cdot \nabla \bfv_h \|^2_{\Omega}
+ \| \phi_u^{\onehalf} \nabla \cdot \bfv_h \|_{\Omega}^2
\nonumber\\
  &\qquad\qquad\lesssim 
  \|{\phi}_{\beta}^{\onehalf}
  (\bfbeta_h \cdot \nabla \bfu_h + \nabla p_h)\|_{\mcT_h}^2
\\
&\qquad\qquad\lesssim
    \| 
    \phi_{\beta}^{\onehalf} 
    (
    \bfbeta \cdot \nabla \bfu_h + \nabla p_h
    )
    \|_{\Omega}^2
    + 
    g_{\beta}(\bfu_h, \bfu_h)
    +
    g_{p}(p_h, p_h)
    + 
    \omega_h
    \bigl(
    \| \mu^{\onehalf}\nabla \bfu_h \|_{\mcT_h}^2
    +
    \| \sigma^{\onehalf} \bfu_h \|_{\mcT_h}^2
    \bigr)
\\
&\qquad\qquad\lesssim
  \|{\phi}_{\beta}^{\onehalf}
  (\bfbeta \cdot \nabla \bfu_h + \nabla p_h)\|_{\Omega}^2
+ (1 + \omega_h) | U_h |_h^2.
\end{align}
\end{proof}
Next, we collect and prove two estimates which will be useful
in deriving a modified inf-sup condition in Lemma~\ref{lem:pressure-stability}.
\begin{lemma}
  \label{lem:Phi_p_bounds}
Let $\bfu_h,\bfv_h\in \mcV_h$, then the following estimates hold
  \begin{align}
    \tn \bfv_h \tn_h 
    &\lesssim
    \bigl(\mu + \| \bfbeta \|_{0,\infty, \Omega}h + \sigma C_P^2
    \bigr)^{\onehalf}
    \bigl(
    \| \nabla \bfv_h \|_{\mcT_h} + \| h^{-\onehalf} \bfv_h \|_{\Gamma}
    \bigr)
    \lesssim
    \Phi_p^{-\onehalf} 
    \bigl(
    \| \nabla \bfv_h \|_{\mcT_h} + \| h^{-\onehalf} \bfv_h \|_{\Gamma}
    \bigr),
    \label{eq:Phi_p_bounds_1}
  \end{align}
  \begin{align}
    |( \bfu_h, \bfbeta \cdot \nabla \bfv_h)_{\Omega}|
    &\lesssim
    \tn \bfu_h \tn_h
    \dfrac{\|\bfbeta\|_{0,\infty,\Omega} C_P}{\sqrt{\mu + \sigma C_P^2}}
    \| \nabla \bfv_h \|_{\Omega}
    \lesssim
    \tn \bfu_h \tn_h
    \Phi_p^{-\onehalf} 
    \| \nabla \bfv_h \|_{\Omega}.
    \label{eq:Phi_p_bounds_2}
  \end{align}
\end{lemma}
\begin{proof}
We start with estimate (\ref{eq:Phi_p_bounds_1}). Then
the reactive term in the norm
definition~(\ref{eq:oseen-norm-unfitted-u}) can be bounded using
the Poincar\'e inequality~(\ref{eq:Poincare-II}) showing that
  \begin{align}
\mu \|\nabla \bfv_h\|_{\Omega}^2 + \sigma \| \bfv_h \|_{\Omega}^2
&\lesssim
(\mu + \sigma C_P^2) (\| \nabla \bfv_h\|^2_{\Omega} +  \|h^{-\onehalf} \bfv_h \|_{\Gamma}^2),
  \end{align}
while the corresponding ghost-penalties $g_{\sigma}$ and $g_{\mu}$ can be simply estimated
by applying the inverse inequality~(\ref{eq:inverse-estimates_F})
to obtain 
$h^{2j+1}\|\jump{\nablan^j \bfv_h} \|_{F}^2 
\lesssim
h^2 \|\nabla \bfv_h\|_{T_F^+\cup T_F^-}^2$
and thus
\begin{align}
  g_{\mu}(\bfv_h, \bfv_h) 
  +
  g_{\sigma}(\bfv_h, \bfv_h) 
\lesssim
(\mu + \sigma h^2) \| \nabla \bfv_h \|_{\mcT_h}^2
\lesssim
(\mu + \sigma C_P^2) \| \nabla \bfv_h \|_{\mcT_h}^2.
\end{align}
The contribution from the remaining ghost-penalty and stabilization terms
can be treated similarly,
\begin{align}
  (s_u + g_{u})(\bfv_h, \bfv_h) 
  &\lesssim
  \| \phi_{u}^{\onehalf} \nabla \bfv_h \|_{\mcT_h}^2
\lesssim
(\mu + \|\bfbeta\|_{0,\infty,\Omega} h + \sigma h^2)
  \|  \nabla \bfv_h \|_{\mcT_h}^2,
\\
  (s_{\beta} + g_{\beta})(\bfv_h, \bfv_h) 
  &\lesssim
  \sum_{T\in\mcT_h} \|\bfbeta_h\|_{0,\infty,T}^2\phi_{\beta,T} \|\nabla \bfv_h\|_T^2
\lesssim \|\bfbeta \|_{0,\infty,\Omega} h \|\nabla \bfv_h\|_{\mcT_h}^2.
\end{align}
Finally, the boundary contributions are clearly bounded by
\begin{align}
\mu \|h^{-\onehalf}\bfv_h\|_{\Gamma}^2
+ \|\phi_u^{\onehalf} h^{-\onehalf} \bfv_h \|_{\Gamma}^2
+ \| |\bfbeta\cdot \bfn|^{\onehalf} \bfv_h\|_{\Gamma}^2
\lesssim  (\mu + \|\bfbeta\|_{0,\infty,\Omega} h + \sigma h^2) \|h^{-\onehalf} \bfv_h\|_{\Gamma}^2,
\end{align}
which concludes the proof of estimate~(\ref{eq:Phi_p_bounds_1}) after recalling definition~\eqref{eq:phi_p_definition} of $\Phi_p$.
Turning to estimate~(\ref{eq:Phi_p_bounds_2}),
we start with observing that
the $L^2$ norm $\|\bfu_h\|_{\Omega}$
can be
bounded by $\tn \bfu_h \tn_h$ in two different ways. Clearly,
$\|\bfu_h\|_{\Omega} \lesssim \sigma^{-\onehalf} \tn \bfu_h \tn_h$. On the other hand, 
after another application of the Poincar\'e inequality~(\ref{eq:Poincare-II}), we see that
\begin{align}
  \|\bfu_h\|_{\Omega} 
  &\lesssim  C_P (\|\nabla \bfu_h\|_{\Omega} +  \| h^{-\onehalf}\bfu_h\|_{\Gamma})
\lesssim
\mu^{-\onehalf}C_P \tn \bfu_h \tn_h.
\end{align}
Taking the minimum of these two bounds and recalling the definition $\Phi_p$, we conclude that 
\begin{align}
  (\bfu_h, \bfbeta\cdot\nabla\bfv_h)_{\Omega}
\lesssim
\tn \bfu_h \tn_h
\min\{\sigma^{-\onehalf},\mu^{-\onehalf}C_P\} \|\bfbeta\|_{0,\infty,\Omega}
\|\nabla \bfv_h \|_{\Omega}
\lesssim
\tn \bfu_h \tn_h
\dfrac{\|\bfbeta\|_{0,\infty,\Omega} C_P}{\sqrt{\mu + \sigma C_P^2}}
\|\nabla \bfv_h \|_{\Omega}
\lesssim
\tn \bfu_h \tn_h
  \Phi_p^{-\onehalf}
\|\nabla \bfv_h \|_{\Omega}.
\end{align}
\end{proof}
In the final lemma, a modified inf-sup condition for $b_h$ is derived, revealing
how the $L^2$ pressure norm can be controlled by adding the symmetric CIP operator~$s_p$.
\begin{lemma}
There is a constant $c_3 > 0$ such that
for $p_h \in \mcQ_h$ there exists a $\bfv_h \in \mcV_h$ satisfying
 \begin{align}
   b_h(p_h, \bfv_h) 
   &\gtrsim
   \Phi_p \| p_h \|_{\Omega}^2
   - c_3 s_p(p_h, p_h)
  \label{eq:pressure-stability}
\intertext{and the stability estimate}
  \label{eq:pressure-stability-vh_stab_bound}
    \tn \bfv_h \tn_h
    &\lesssim
    \Phi_p^{-\onehalf} (
    \| \nabla \bfv_h \|_{\mcT_h}
    +
    \| h^{-\onehalf} \bfv_h \|_{\Gamma}
    )
    \lesssim
    \Phi_p^{\onehalf} \| p_h \|_{\Omega} + g_p(p_h,p_h)^{\onehalf}
 \end{align}
  whenever the stability parameters $\gamma,
  \gamma_\mu,\gamma_\sigma,\gamma_\beta,\gamma_u,\gamma_p$ are chosen
  to be
strictly positive
.
  \label{lem:pressure-stability}
\end{lemma}
\begin{proof}
  For given $p_h\in\mcQ_h$, we construct $\bfv_h$ in two steps.
  \\
  {\bf Step 1.} Due to the surjectivity of the divergence operator
  $\nabla \cdot : [H^1_0(\Omega)]^d \to L^2(\Omega)$ there exists a function
  $\bfv_p$ such that $\nabla \cdot \bfv_p = -\Phi_p p_h$
  and $\| \bfv_p \|_{1,\Omega} \sim \| \nabla \bfv_p \|_{\Omega} \lesssim \Phi_p \| p_h \|_{\Omega}$.
  Using the Cl\'ement interpolant, we
  set $\bfv_h^1 := \bfpiast \bfv_p \in \mcV_h$ and recall that 
  $\bfv_p |_{\Gamma} = 0$ to obtain the identity
  \begin{align}
    b_h(p_h, \bfv_h^1)
    &=
    b_h(p_h, \bfv_p)
    +
    b_h(p_h, \bfpiast \bfv_p - \bfv_p)
    \label{eq:inf-sup_vh1_constr-I}
    \\
    &= \Phi_p \| p_h \|_{\Omega}^2 
    + \bigl(
    \Phi_p^{\onehalf} h \nabla p_h, \Phi_p^{-\onehalf}
    h^{-1}(\bfpiast \bfv_p - \bfv_p)
    \bigr)_{\Omega},
    \label{eq:inf-sup_vh1_constr-II}
  \end{align}
  where the second term in~\eqref{eq:inf-sup_vh1_constr-I} was
  integrated by parts. Now a combination of a
  $\delta$-scaled Cauchy-Schwarz inequality,
  the interpolation estimate~\eqref{eq:interpest0-ast},
  and finally, the stability bound
  $\| \nabla \bfv_p \|_{\Omega} \lesssim \Phi_p \| p_h \|_{\Omega}$
  yields
  \begin{align}
    b_h(p_h, \bfv_h^1)
    &\gtrsim
    \Phi_p \| p_h \|_{\Omega}^2 
    - \delta^{-1}\|\Phi_p^{\onehalf} h\nabla p_h\|^2_{\Omega}
    - \delta 
    \Phi_p^{-1} \| \bfv_p \|_{1, \Omega}^2
    \\
    &\gtrsim
    (1 - \delta) \Phi_p \| p_h \|_{\Omega}^2 
    - \delta^{-1}\|\Phi_p^{\onehalf} h\nabla p_h\|^2_{\Omega}.
    \label{eq:inf-sup_vh1_constr-III}
  \end{align}
  {\bf Step 2.} 
  Next, we show how to compensate for the
  $\|\Phi_p^{\onehalf} h \nabla p_h\|^2_{\Omega}$
  term appearing in \eqref{eq:inf-sup_vh1_constr-III} using the
  stabilization form $s_p$.
  To construct a suitable test function, set 
  $\bfv_h^2 := \Phi_p h^2 \mcO_h(\nabla p_h)\in \mcV_h$.
  Then inserting $\bfv_h^2$ into
  $b_h$ shows after an integration by parts that
\begin{align}
  b_h(p_h, \bfv_h^2)
  = (\nabla p_h, \bfv_h^2)_{\Omega}
  &
  = 
    \|\Phi_p^{\onehalf} h\nabla p_h\|^2_{\Omega}
  + 
  \bigl(
  \Phi_p^{\onehalf} h \nabla p_h,
  \Phi_p^{\onehalf} h
  (\mcO_h(\nabla p_h) -
  \nabla p_h)
  \bigr)_{\Omega}
  \\
  &\gtrsim
  (1 -\delta)
    \|\Phi_p^{\onehalf} h\nabla p_h\|^2_{\Omega}
  -\delta^{-1} s_p(p_h,p_h),
    \label{eq:inf-sup_vh2_constr-III}
\end{align}
where we combined a $\delta$-Young inequality and
the Oswald interpolant Lemma~\ref{lem:Oswald-interpolant}
to obtain
\begin{align}
  \|
  \Phi_p^{\onehalf} h
  (\mcO_h(\nabla p_h) -
  \nabla p_h)
  \|_{\Omega}^2
  \lesssim
  \sum_{F\in \mcF_i}
  \underbrace{\Phi_p h^2}_{\lesssim \phi_{p,F}} h \|\jump{\nabla p_h}\|_F^2
  \lesssim
  s_p(p_h, p_h).
\end{align}
Finally, 
using the same $\delta \sim 0.5$ 
in~\eqref{eq:inf-sup_vh1_constr-III}
and~\eqref{eq:inf-sup_vh2_constr-III},
we set $\bfv_h = \bfv_h^1 + 2\delta^{-1} \bfv_h^2$
yielding
\begin{align}
  b_h(p_h, \bfv_h)
  \gtrsim
  \Phi_p \|p_h\|_{\Omega}^2
  - c_3 s_p(p_h,p_h)
\end{align}
for some constant $c_3>0$.
\\
{\bf Estimate~\eqref{eq:pressure-stability-vh_stab_bound}.}
Utilizing the stability bound \eqref{eq:Phi_p_bounds_1} for~$\bfv_h$ it is sufficient to prove
\begin{align}
  \| \nabla \bfv_h^1 \|_{\mcT_h}
  + \| h^{-\onehalf} \bfv_h^1 \|_{\Gamma}
  &= \| \nabla \bfpiast \bfv_p \|_{\mcT_h}
  + \| h^{-\onehalf} (\bfpiast \bfv_p - \bfv_p)\|_{\Gamma}
  \\
  &\lesssim \| \nabla \bfpiast \bfv_p \|_{\mcT_h}
  + \| h^{-1} (\bfpiast \bfv_p - \bfv_p)\|_{\mcT_h}
  \\
  &\lesssim \| \nabla \bfv_p \|_{\Omega}
  \lesssim \Phi_p \| p_h \|_{\Omega},
  \label{eq:stability-bound-v_h1}
\end{align}
where the fact was used that \mbox{$\bfv_p|_\Gamma=0$}, followed by a trace inequality
and the interpolation and stability properties of the Cl\'ement interpolant \eqref{eq:interpest0-ast} and \eqref{eq:interpest0_Clement_H_stability}.
Similarly,
\begin{align}
  \| \nabla \bfv_h^2 \|_{\mcT_h}
  + \| h^{-\onehalf} \bfv_h^2 \|_{\Gamma}
  &\lesssim
  \| h^{-1} \mcO_h(h^2 \Phi_p \nabla p_h) \|_{\mcT_h}
  \lesssim
  \|{\Phi}_p p_h \|_{\mcT_h}
  \\
  &\lesssim
  {\Phi}_p  \| p_h \|_{\Omega}
  + {\Phi}_p^{\onehalf} g_p(p_h, p_h)^{\onehalf},
  \label{eq:stability-bound-v_h2}
\end{align}
where in the last step, the ghost-penalty norm equivalence for the pressure from Corollary~\ref{cor:norm_equivalence_rsb_pressure_div} was applied.
Consequently, combing the stability
bounds~\eqref{eq:stability-bound-v_h1} and
\eqref{eq:stability-bound-v_h2} with stability
bound~\eqref{eq:Phi_p_bounds_1}
from Lemma~\ref{lem:Phi_p_bounds} gives
\begin{align}
  \tn \bfv_h \tn_h 
  \lesssim
  \Phi_p^{-\onehalf} (\| \nabla \bfv_h \|_{\mcT_h} + \|h^{-\onehalf}
  \bfv\|_{\Gamma})
  \lesssim 
  \Phi_p^{\onehalf} \|p_h\|_{\Omega}
  + g_p(p_h,p_h)^{\onehalf}.
\end{align}

\end{proof}
\begin{remark}
We point out that the classical proof of the modified inf-sup
condition~(\ref{eq:pressure-stability}) uses the $L^2$ projection to
construct a proper test function $\bfv_h^1$
and exploits the $L^2$ orthogonality to insert $\mcO_h(\nabla p_h) \in
\mcV_h$ in ~\eqref{eq:inf-sup_vh1_constr-II} and then use ~(\ref{eq:oswald-stabilization-pressure})
from Corollary~\ref{cor:cip-fluctuation-control}
directly. While a stabilized/perturbed $L^2$ projection
was successfully used in \cite{BurmanClausMassing2015} to analyze
a first-order cut finite element method for the three-field Stokes problem,
its theoretical treatment for higher-order elements is not trivial.
Consequently, we used an alternative route to establish~(\ref{eq:pressure-stability})
without relying on some sort of $L^2$ orthogonality.
\end{remark}

As a consequence of the previous lemma we can now show that
the bilinear form $A_h+S_h+G_h$ satisfies an inf-sup
condition with respect to the norm
\begin{align}
  \tn U_h \tn_h^2 
  = | U_h |_h^2 + \| \phi_u^{\onehalf} \nabla \cdot \bfu_h \|_{\Omega}^2
  + \dfrac{1}{1 + \omega_h} \|\phi_{\beta}^{\onehalf}(\bfbeta\cdot \nabla \bfu_h + \nabla p_h) \|_{\Omega}^2
  + \Phi_p \| p_h\|_{\Omega}^2,
  \label{eq:Ah-norm}
\end{align}
which ensures existence and uniqueness of a discrete velocity and pressure solution.
\begin{theorem}
\label{theorem:inf-sup_condition_total}
 Let $U_h=(\bfu_h,p_h)\in \mcW_h$.
Then, under the assumptions of Lemma~\ref{lem:coercivity_ah}, \ref{lem:divergence-stability}, \ref{lem:convect-pressure-stab}
and~\ref{lem:pressure-stability} on the stabilization parameters,
the cut finite element method is inf-sup stable 
 \begin{equation}
\label{eq:inf-sup_condition_Ah_weak_norm}
\tn U_h \tn_h
\lesssim
\sup_{V_h\in \mcW_h \setminus \{0\}}
\dfrac{
A_h(U_h,V_h) + S_h(U_h,V_h)
+ G_h(U_h,V_h)
}
{\tn V_h \tn_h },
\end{equation}
where the hidden stability constant is independent of the mesh size $h$ and the position of the boundary
relative to the background mesh.
\end{theorem}
\begin{proof}
Given $U_h\in\mcW_h$ we construct a suitable test function $V_h\in\mcW_h$
based on the Lemma~\ref{lem:coercivity_ah}, \ref{lem:divergence-stability}, \ref{lem:convect-pressure-stab}
and~\ref{lem:pressure-stability}.
\\
\noindent{\bf Step 1.}
To gain control over the weakly scaled divergence $\nabla \cdot\bfu_h$,
define the test function $V_h^1 := (0,q_h^1)$ with $q_h^1$ chosen as in
Lemma~\ref{lem:divergence-stability}. Then 
\begin{align}
  (A_h
  + S_h
  + G_h)(U_h,V_h^1)
  &= -b_h(q_h^1, \bfu_h) + s_p(p_h,q_h^1) + g_p(p_h, q_h^1)
  \\
    \nonumber
  &\gtrsim
  \| \phi_u^{\onehalf} \nabla \cdot \bfu_h \|_{\Omega}^2
   - 
   c_1 \bigl(
   s_u(\bfu_h,\bfu_h)
   + g_u(\bfu_h,\bfu_h)
   + \| h^{-\onehalf}\phi_u^{\onehalf} \bfu_h\cdot \bfn \|_{\Gamma}^2
   \bigr)
  \\
  &\phantom{\gtrsim}\quad -\delta^{-1}
  \bigl(
  s_p(p_h,p_h) + g_p(p_h, p_h)
  \bigr)
  -\delta
    \bigl(
  s_p(q_h^1,q_h^1) + g_p(q_h^1, q_h^1)
  \bigr)
  \\
  &\gtrsim
  (1-\delta)\| \phi_u^{\onehalf} \nabla \cdot \bfu_h \|_{\Omega}^2
   - C_1(\delta) |U_h|_h^2.
\end{align}
\noindent{\bf Step 2.}
Next, we set $V_h^2 := (\bfv_h^2,0)$ with $\bfv_h^2$ taken from Lemma~\ref{lem:convect-pressure-stab}.
Inserting $V_h^2$ into $A_h + S_h + G_h$ and integrating $b_h$ by parts leads us to
\begin{align}
A_h(U_h, V_h^2) 
&= a_h(\bfu_h, \bfv_h^2) + b_h(p_h, \bfv_h^2)
\\
&\gtrsim 
- \tn \bfu_h \tn_h
\tn \bfv_h^2 \tn_h
  + (\bfbeta\cdot\nabla \bfu_h + \nabla p_h, \bfv_h^2)_{\Omega}
  \label{eq:use-conv-pressure-stab-step-1}
\\
&\gtrsim 
- \delta^{-1}\tn \bfu_h \tn_h^2
- \delta \tn \bfv_h^2 \tn_h^2
+ \| \phi_{\beta}^{\onehalf}(\bfbeta \cdot \nabla \bfu_h + \nabla p_h)\|_{\Omega}^2
   - c_2(1 + \omega_h) | U_h |_h^2
\\
&\gtrsim 
(1-\delta)\| \phi_{\beta}^{\onehalf}(\bfbeta \cdot \nabla \bfu_h + \nabla p_h)\|_{\Omega}^2
  - C_2(\delta)(1 + \omega_h) | U_h |_h^2
\label{eq:A_h-step-2}
\end{align}
where we employed the stability bound~(\ref{eq:convect-pressure-stab-vh_stab_bound_2})
after an application of a $\delta$-scaled Young inequality.
Employing the same steps to the remaining stabilization terms shows that
\begin{align}
(S_h + G_h)(U_h, V_h^2) 
&\gtrsim
- \delta^{-1}\tn \bfu_h \tn_h^2
- \delta \tn \bfv_h^2 \tn_h^2
\\
&\gtrsim
-\delta\| \phi_{\beta}^{\onehalf}(\bfbeta \cdot \nabla \bfu_h + \nabla p_h)\|_{\Omega}^2
  - C_2(\delta)(1 + \omega_h) | U_h |_h^2.
\label{eq:S_h-G_h-step-2}
\end{align}
Thus after
combining~(\ref{eq:A_h-step-2}) and~(\ref{eq:S_h-G_h-step-2}),
we find that
\begin{align}
(A_h+G_h + S_h)(U_h, V_h^2) 
\gtrsim 
(1-2\delta)
\| \phi_{\beta}^{\onehalf}(\bfbeta \cdot \nabla \bfu_h + \nabla p_h)\|_{\Omega}^2
  - 2C_2(\delta)(1 + \omega_h) | U_h |_h^2.
\end{align}
\noindent{\bf Step 3.}
The $L^2$ pressure norm term can be constructed by testing with
$V_h^3 := (\bfv_h^3,0)$ where $\bfv_h^3$ is now chosen as in
Lemma~\ref{lem:pressure-stability}.
Integrating the advective term by parts and
making use of the estimates~\eqref{eq:Phi_p_bounds_1}, \eqref{eq:Phi_p_bounds_2}, and
the stability bound~\eqref{eq:pressure-stability-vh_stab_bound}
allows us to deduce that
\begin{align}
A_h(U_h, V_h^3) 
&= a_h(\bfu_h, \bfv_h^3) + b_h(p_h, \bfv_h^3)
\\
&\gtrsim 
- \tn \bfu_h \tn_h
\tn \bfv_h^3 \tn_h
- (\bfu_h, \bfbeta\cdot\nabla \bfv_h^3)_{\Omega}
+ 
  \Phi_p \| p_h \|_{\Omega}^2 - c_3 s_p(p_h, p_h)
  \\
  &\gtrsim
  - \delta^{-1}\tn \bfu_h \tn_h^2
  - \delta \Phi_p^{-1}(\|\nabla \bfv_h^3\|_{\mcT_h}^2 + \|h^{-\onehalf}
  \bfv_h^3\|_{\Gamma}^2)
  - \delta^{-1} \tn \bfu_h \tn_h^2 
  - \delta \Phi_p^{-1} \| \nabla \bfv_h^3 \|_{\Omega}^2
  + 
  \Phi_p \| p_h \|_{\Omega}^2 - c_3 s_p(p_h, p_h)
  \\
  &\gtrsim
  (1-2\delta) \Phi_p \| p_h \|_{\Omega}^2
  -2 \delta^{-1} \tn \bfu_h \tn_h^2 - c_3s_p(p_h,p_h) - 2\delta
  g_p(p_h,p_h)
  \\
  &\gtrsim
  (1-2\delta) \Phi_p \| p_h \|_{\Omega}^2
  - (c_3 + 2\delta + 2 \delta^{-1}) | U_h |_{h}^2.
  \label{eq:A_h-step-3-a}
\end{align}
Analogously to {\bf Step 2}, the stabilization terms can be estimated as
\begin{align}
 (S_h + G_h )(U_h, V_h^3)
&\gtrsim
- \tn \bfu_h \tn_h
\tn \bfv_h^3 \tn_h
\\
&\gtrsim
  - \delta \Phi_p\| p_h\|_{\Omega}^2
  - \delta g_p(p_h,p_h)
  - \delta^{-1} \tn \bfu_h \tn_h^2
\\
&\gtrsim
  - \delta \Phi_p\| p_h\|_{\Omega}^2 
  - (\delta+\delta^{-1}) | U_h |_{h}^2
\label{eq:S_h-G_h-step-3-a}
\end{align}
such that after combining \eqref{eq:A_h-step-3-a} and \eqref{eq:S_h-G_h-step-3-a}
\begin{align}
(A_h+S_h+G_h)(U_h, V_h^3) 
\gtrsim 
  \Phi_p \| p_h \|_{\Omega}^2  - C_3(\delta) | U_h |_{h}^2.
\end{align}
\noindent{\bf Step 4.}
To gain control over the last missing $|U_h|_h$ term,
set $V_h^4 := U_h$. By Lemma~\ref{lem:coercivity_ah}, it holds
\begin{align}
  A_h(U_h, V_h^4) + S_h(U_h, V_h^4) + G_h(U_h, V_h^4) \gtrsim |U_h|_h^2.
\end{align}
\noindent{\bf Step 5.}
Finally, for given $U_h\in\mcV_h$ we choose $\delta$ sufficiently small and define $V_h^5 := \eta \bigl(V_h^1 + (1+\omega_h)^{-1}V_h^2 + V_h^3\bigr) + V_h^4$
for some $2 \eta \sim (C_1(\delta) + C_2(\delta) + C_3(\delta))^{-1}$. 
Thanks to the stability estimate~\eqref{eq:convect-pressure-stab-vh_stab_bound_2}
and norm definition~\eqref{eq:Ah-norm}, we have
$\tn (1+\omega_h)^{-1} V_h^2 \tn_h 
\lesssim (1+\omega_h)^{-\onehalf} \tn U_h \tn_h
\lesssim  \tn U_h \tn_h$
and as a result $\tn V_h^2\tn_h \lesssim \tn U_h \tn_h$.
Similarly, the stability bounds~(\ref{eq:divergence-stability-qh_stab_bound})
and~(\ref{eq:pressure-stability-vh_stab_bound})
imply that
$\tn V_h^1 + V_h^3 \tn_h \lesssim \tn U_h \tn_h$
and thus $\tn V_h^5 \tn_h \lesssim \tn U_h \tn_h$.
Consequently,
\begin{align}
  (A_h+S_h + G_h)(U_h, V_h^5) 
  &\gtrsim
(1 - \eta(C_1(\delta) + C_2(\delta) + C_3(\delta)))| U_h |_{h}^2 \nonumber\\
  &\qquad\quad
  +
  \eta \left(\|\phi_u^{\onehalf}\nabla\cdot\bfu_h\|_{\Omega}^2
  +
  \dfrac{1}{1 + \omega_h} 
  \| \phi_{\beta}^{\onehalf}((\bfbeta \cdot \nabla) \bfu_h + \nabla p_h)\|_{\Omega}^2
  + \Phi_p \|p_h\|_{\Omega}^2 \right)
\\
&\gtrsim \tn U_h \tn_h^2
\gtrsim \tn U_h \tn_h \tn V_h^5 \tn_h,
\end{align}
 which concludes the proof
by choosing the supremum over $V_h\in\mcW_h\backslash \{0\}$.
\end{proof}

\begin{remark}
 Note that in the previous theorem, the inf-sup stability is proven with respect to
 an energy-norm \mbox{$\tn U_h \tnast$}, which is based on the underlying (active) background mesh~$\mcT_h$.
 Thereby, the different ghost-penalty operators \mbox{$g_\mu,g_\sigma,g_\beta,g_u,g_p$} ensure sufficient control over discrete polynomials
defined on the entire (active) computational mesh
and so significantly improve the system conditioning of the resulting linear matrix system - for all different flow regimes and
independent of how the boundary intersects the mesh.
For further details on the improvement of the system conditioning owing to the use of ghost-penalties, the reader is referred to 
works by \citet{BurmanHansbo2012} and \citet{MassingLarsonLoggEtAl2013}.
\end{remark}

\section{A Priori Error Estimates}
\label{sec:apriori-analysis}
The goal of this section is to prove the main a priori
estimates~\eqref{eq:apriori-estimate-cutfem}
for the error in the discrete velocity and pressure solution.
We proceed in three steps.
First, two lemmas are provided which are concerned with potential
consistency errors introduced by the stabilization forms $S_h$ and $G_h$.
Second, interpolation error estimates are derived.
Finally, the estimates for the interpolation and consistency error are
combined with the inf-sup stability
result~\eqref{eq:inf-sup_condition_Ah_weak_norm} from the previous section
to establish the final \apriori~estimate
in Theorem~\ref{thm:apriori-estimate}.

\subsection{Consistency Error Estimates}
We start with showing that the discrete formulation~\eqref{eq:oseen-discrete-unfitted}
satisfies a weakened form of the
Galerkin orthogonality.
\begin{lemma}
\label{lem:weak-orthogonality}
  Suppose that the solution $U = (\bfu,p)$ of the variational
  formulation~\eqref{eq:oseen_weak} is in
  $[H^2(\Omega)]^d \times H^1(\Omega)$
  and let
  $U_h = (\bfu_h, p_h) \in \mcV_h \times \mcQ_h$ 
  be the finite element solution to the
  discrete weak formulation~\eqref{eq:oseen-discrete-unfitted}.
  Then
  \begin{align}
    A_h(U - U_h, V_h) = S_h(U_h, V_h) + G_h(U_h, V_h).
    \label{eq:weak-orthogonality}
  \end{align}
\end{lemma}
\begin{proof}
  The proof follows immediately 
  the fact that the continuous
  solution satisfies $A_h(U,V_h) = L(V_h)$
  due to the definition of the weak
  problem~\eqref{eq:oseen_weak}.
\end{proof}

The next lemma ensures that the remainder term arising in the
weakened Galerkin orthogonality~\eqref{eq:weak-orthogonality}
is weakly consistent and thus
does not deteriorate the convergences rate of the proposed scheme.
\begin{lemma}
\label{lem:weak-consistency}
Assume that $(\bfu, p) \in [H^r(\Omega)]^d \times H^s(\Omega)$
and let $r_{u} := \min\{r, k+1\}$ and $s_p := \min\{s, k+1\}$
where $k$ is the polynomial degree of the approximation spaces
for the velocity and pressure. Then
\begin{align}
  S_h(\Piast U, \Piast U)
  +
  G_h(\Piast U, \Piast U)
  &\lesssim
  ( \mu + \|\bfbeta\|_{0,\infty,\Omega}h +  \sigma h^{2})
   h^{2r_{u} -2}\| \bfu \|_{r_{u}, \Omega}^2
     +
   \max_{T \in \mcT_h}
   \left\{
   \dfrac{1}{\mu + \|\bfbeta \|_{0,\infty, T} h + \sigma h^2}
   \right\}
   h^{2 s_p}
   \| p \|_{s_p, \Omega}^2.
\end{align}
\end{lemma}
\begin{proof}
Recalling the definition of $S_h$ and $G_h$, it is enough to derive the desired estimate  
for $G_h$
  \begin{align}
  G_h(\Piast U, \Piast U) 
  &=  g_\sigma(\bfpiast \bfu, \bfpiast \bfu)
  +  g_\mu(\bfpiast \bfu, \bfpiast \bfu)
  +  g_\beta(\bfpiast \bfu, \bfpiast \bfu)
  +  g_u(\bfpiast \bfu, \bfpiast \bfu)
+  g_p(\piast p, \piast p),
  \end{align}
since the contributions to $S_h$ can be treated in the same way.
We start with considering~$g_{\beta}$.
 Since $\bfu \in [H^{r}(\Omega)]^d$, its traces $\nablan^j \bfu|_{F}$ are uniquely defined for
$0\leqslant j \leqslant r_u -1$ and therefore, $\jump{\bfbeta_h \cdot \nabla \nablan^j \bfu} = 0$
for $0 \leqslant j \leqslant r_u - 2$. Consequently,
\begin{align}
  g_{\bfbeta}(\bfpiast \bfu, \bfpiast \bfu)
&=
\sum_{j=0}^{r_u -2 } h^{2j-1} 
\sum_{F \in \mcF_{\Gamma}}
\phi_{\bfbeta,F}
\| 
 \jump{
\bfbeta
\cdot \nabla \nablan^j(\bfpiast \bfu - \bfu)} 
\|^2_{F}
+
\sum_{j=r_u-1}^{k-1} h^{2j-1} 
\sum_{F \in \mcF_{\Gamma}}
\phi_{\bfbeta,F}
\| 
 \jump{
\bfbeta
\cdot \nabla \nablan^j\bfpiast \bfu} 
\|^2_{F}
= I + II.
\end{align}
The interpolation estimate~(\ref{eq:interpest0_Clement_H_stability})
together with the fact
that by definition 
$\phi_{\beta,T} \|
\bfbeta
\|_{0,\infty,T}^2 \lesssim \|\bfbeta \|_{0,\infty, T}h$
implies now that 
\begin{align}
  I
\lesssim
\sum_{T \in \mcT_h}
\sum_{j=0}^{r_u - 2 } h^{2j-1} 
\phi_{\beta,T} \|
\bfbeta
\|_{0,\infty,T}^2
\| 
 \nabla \nablan^j(\bfpiast \bfu - \bfu)
\|^2_{\partial T}
\lesssim
\sum_{T \in \mcT_h}
\|\bfbeta\|_{0,\infty,T} h^{2r_u - 1}
 \| \bfu^\ast \|_{r_u, \omega(T)}^2
\lesssim
\|\bfbeta\|_{0,\infty,\Omega} 
h^{2r_u-1} 
\| \bfu \|_{r_u, \Omega}^2.
\label{eq:g_beta_consistent_part-I}
\end{align}
Turning to the second term $II$, a simple application of the inverse
estimate~(\ref{eq:inverse-estimates_F}) shows that
\begin{align}
  II
&\lesssim
\sum_{T \in \mcT_h}
h^{2r_u-2}
\phi_{\beta,T} \| \bfbeta \|_{0,\infty,T}^2 \| D^{r_u}\bfpiast \bfu \|^2_T
\lesssim
\|\bfbeta\|_{0,\infty,\Omega} 
h^{2r_u-1} 
\| \bfu \|_{r_u, \Omega}^2.
\label{eq:g_beta_inconsistent_part-II}
\end{align}
after observing that $\bfpiast$ is stable thanks to~(\ref{eq:interpest0-ast}).
Similarly, the remaining inconsistency terms can be bounded as follows:
\begin{align}
  g_{u}(\bfpiast \bfu, \bfpiast \bfu)
&\lesssim
(\mu + \|\bfbeta \|_{0,\infty,\Omega}h + \sigma h^2)h^{2r_u -2} \| \bfu \|_{r_u,\Omega}^2,
\\
  g_{p}(\piast p, \piast p)
&\lesssim
\sum_{T\in \mcT_h}
\dfrac{1
}
{
\mu + \|\bfbeta\|_{0,\infty,T}h + \sigma h^2
}
h^{2s_p} 
 \| p^{\ast} \|_{s_p, T}^2
\\
&\lesssim 
   \max_{T \in \mcT_h}
   \left\{
   \dfrac{1}{\mu + \|\bfbeta \|_{0,\infty, T} h + \sigma h^2}
   \right\}
h^{2s_p} 
 \| p \|_{s_p, \Omega}^2,
\\
  g_{\sigma}(\bfpiast \bfu, \bfpiast \bfu)
&\lesssim
\sigma h^{2r_u} \|\bfu\|_{r_u,\Omega}^2,
\\
  g_{\mu}(\piast p, \piast p)
&\lesssim
\mu h^{2r_u-2} \|\bfu\|_{r_u,\Omega}^2.
\end{align}
Applying the same arguments to $S_h(\Piast U, \Piast U)$ and collecting all estimates concludes the proof.
\end{proof}

\subsection{Interpolation Error Estimates}
The next lemma ensures that
the interpolation error between continuous solution and its Cl\'ement interpolation
converges with optimal rates.

\begin{lemma}
\label{lem:interpolation-estimate}
Assume that $(\bfu,p) \in [H^r(\Omega)]^d \times H^s(\Omega)$
and let $r_{u} := \min\{r, k+1\} \geqslant 2$, $s_p := \min\{s,k+1\}$
where $k$ is the polynomial degree of the approximation spaces
for the velocity and pressure. Then
\begin{align}
\tn \bfu  - \bfpi \bfu \tn 
&\lesssim
(\mu +  \|\bfbeta\|_{0,\infty,\Omega} h + \sigma h^2)^{\onehalf}
h^{r_u-1}\|\bfu \|_{r_u,\Omega},
\label{eq:interpolation-estimate-u}
\\
\|p^{\ast}  - \piast p \|_{\Omega} 
&\lesssim
h^{s_p}\|p \|_{s_p,\Omega}.
\label{eq:interpolation-estimate-p}
\end{align}
\end{lemma}
\begin{proof}
  We only sketch the proof for~\eqref{eq:interpolation-estimate-u}
  since the second estimate~\eqref{eq:interpolation-estimate-p} follows directly from 
  the interpolation estimate~(\ref{eq:interpest0-ast}).
  Starting from the interpolation estimate~(\ref{eq:interpest0-ast}),
  an application of the
  trace inequality~(\ref{eq:trace-inequality-cut})
  together with the definition of $\phi_u$ 
  shows that the boundary terms can be estimated in terms of the element contributions:
\begin{align}
\|(\mu + \phi_u)^{\onehalf} h^{-\onehalf}(\bfu^{\ast} - \bfpiast \bfu) \|_{\Gamma}^2
&\lesssim
(\mu +  \|\bfbeta\|_{0,\infty,\Omega}h + \sigma h^2)
\bigl(
h^{-2}\| \bfu^{\ast} - \bfpiast \bfu \|_{\mcT_h}^2
+\|\nabla(\bfu^{\ast} - \bfpiast \bfu)\|_{\mcT_h}^2
\bigr)
\\
&\lesssim
(\mu +  \|\bfbeta\|_{0,\infty,\Omega}h + \sigma h^2) h^{2(r_u-1)}\|\bfu \|_{r_u,\Omega}^2,
\\
\| |\bfbeta\cdot\bfn|^{\onehalf} (\bfu^{\ast} - \bfpiast \bfu) \|_{\Gamma}^2
&\lesssim
\|\bfbeta\|_{0,\infty,\Omega}
(h^{-1}\| \bfu^{\ast} - \bfpiast \bfu \|_{\mcT_h}^2
+ h \|\nabla(\bfu^{\ast} - \bfpiast \bfu)\|_{\mcT_h}^2)
\\
&\lesssim
(\|\bfbeta\|_{0,\infty,\Omega}h) h^{2(r_u-1)}\|\bfu \|_{r_u,\Omega}^2.
\end{align}
Next, the viscous and reactive parts can be estimated analogously by
\begin{align}
 \| \mu^{\onehalf} \nabla(\bfu^{\ast} - \bfpiast \bfu) \|_{\Omega}^2 + \|\sigma^{\onehalf} (\bfu^{\ast} - \bfpiast \bfu) \|_{\Omega}^2
 \lesssim (\mu + \sigma h^2) h^{2(r_u-1)}\|\bfu \|_{r_u,\Omega}^2.
\end{align}
It only remains to bound $s_{\beta}$ and $s_u$ which can be done 
exactly in the same way as in the consistent part in the error estimate for $g_{\beta}$
and $g_u$, see~(\ref{eq:g_beta_consistent_part-I}).
\end{proof}

\subsection{A Priori Error Estimates}

The subsequent theorem states the main \apriori~error estimate for the velocity in a natural energy norm
and for the pressure in an $L^2$-norm.

\begin{theorem}
  \label{thm:apriori-estimate}
Assume that  $U = (\bfu, p) \in [H^r(\Omega)]^d \times H^s(\Omega)$
is the weak solution of the Oseen problem~\eqref{eq:oseen_weak} and let
$U_h = (\bfu_h,p_h) \in \mcV_h \times \mcQ_h$ be the discrete solution
of problem~\eqref{eq:oseen-discrete-unfitted}. Then
\begin{align}
    \tn \bfu -  \bfu_h \tn
  + \Phi_p^{\onehalf}\| p - p_h \|_{\Omega}
    &\lesssim
    (1+\omega_h)^{\onehalf}
    \bigl(
    \mu + \|\bfbeta \|_{0,\infty,\Omega} h + \sigma h^2
    \bigr)^{\onehalf}
    h^{r_{u} - 1}
    \| \bfu \|_{r_{u},\Omega}
      \nonumber
    \\
    &\phantom{\lesssim}\quad
      +
    \left(
    \Phi_p +
   \max_{T\in \mcT_h}
   \left\{\dfrac{1}{\mu + \|\bfbeta \|_{0,\infty, T} h + \sigma h^2}
   \right\}
   \right)^{\onehalf}
   h^{ s_p}
   \| p \|_{s_p, \Omega},
    \label{eq:apriori-estimate}
  \end{align} 
where $r_{u} := \min\{r, k+1\}$ and $s_p := \min\{s, k+1\}$ \color{black}and $\Phi_p$ and $\omega_h$ as in \eqref{eq:phi_p_definition}
Note that the hidden constants are independent of $h$ and particularly independent of how
the boundary intersects the mesh $\mcT_h$.
\end{theorem}
\begin{proof}
Recalling the norm definitions from Section~\ref{ssec:cutfem-oseen}, we can
split the total discretization error into an interpolation and discrete error part,
\begin{align}
  \tn \bfu^{\ast} - \bfu_h \tn
+
  \Phi_p^{\onehalf}\| p^{\ast} - p_h \|_{\Omega}
  & \lesssim
  \tn \bfu^{\ast} - \bfpiast \bfu \tn
    + \Phi_p^{\onehalf} \| p^{\ast} - \piast p \|_{\Omega}
    + 2\tn \Piast U - U_h \tn_h.
  \label{eq:apriori_split_velocity}
\end{align}
Then thanks to the interpolation estimates~(\ref{eq:interpolation-estimate-u}) and (\ref{eq:interpolation-estimate-p}),
it is enough to consider the discrete error $ \tn \Piast U - U_h \tn_h$.
The inf-sup condition~\eqref{eq:inf-sup_condition_Ah_weak_norm}
and the weak Galerkin orthogonality~\eqref{eq:weak-orthogonality}
ensures there exists a $V_h$
with $\tn V_h \tn_h = 1$ such that
\begin{align}
  \tn \Piast U - U_h \tn_h
  &\lesssim
  A_h(\Piast U - U_h, V_h) 
  + S_h(\Piast U - U_h, V_h) 
  + G_h(\Piast U - U_h, V_h) 
  \\
  &= 
  A_h(\Piast U - U, V_h) 
  + S_h(\Piast U, V_h) 
  + G_h(\Piast U, V_h).
  \label{eq:discrete-error-est-I}
\end{align}
After combining a Cauchy-Schwarz inequality with 
Lemma~\ref{lem:weak-consistency}, 
the last two terms in \eqref{eq:discrete-error-est-I} can be bounded by
\begin{align}
\label{eq:consistency_error}
  S_h(\Piast U, V_h) + G_h(\Piast U, V_h)
  &\lesssim
  ( \mu + \|\bfbeta\|_{0,\infty,\Omega}h +  \sigma h^{2})^{\onehalf}
   h^{r_{u} -1}\| \bfu \|_{r_{u}, \Omega}
    +
   \max_{T \in \mcT_h}
   \left\{
   \dfrac{1}{\mu + \|\bfbeta \|_{0,\infty, T} h + \sigma h^2}
   \right\}^{\onehalf}
   h^{ s_p}
   \| p \|_{s_p, \Omega} 
\end{align}
and thus it remains to estimate 
$A_h(\Piast U - U, V_h)$.
After recalling definition~(\ref{eq:Ah-form-def}), 
integrating 
 $b_h(q_h, \bfpiast \bfu - \bfu)$ 
and
the convective part in $a_h$ by parts, and a final application of a
Cauchy-Schwarz inequality to the remaining terms in $a_h$,
we see that 
\begin{align}
  A_h(\Piast U - U, V_h)
 &= a_h( \bfpiast \bfu - \bfu, \bfv_h)
 + 
 b_h(\piast p - p, \bfv_h)
  - b_h(q_h, \bfpiast \bfu - \bfu)
\\
&\lesssim \tn  \bfpiast \bfu - \bfu \tn \tn \bfv_h \tn
- (\bfpiast \bfu - \bfu, \bfbeta \cdot \nabla \bfv_h + \nabla q_h)
+ b_h(\piast p - p, \bfv_h)
\\
  &=  I + II + III
 \end{align}
which we estimate next. 
\\
\noindent{\bf Term $I$.}
A simple application of
the interpolation estimate~(\ref{eq:interpolation-estimate-u})
together with the inequality $\tn \bfv_h\tn \lesssim \tn V_h\tn_h$
gives
\begin{align}
  I \lesssim (\mu + \|\bfbeta \|_{0,\infty,\Omega}h + \sigma h^2)^{\onehalf}
 h^{r_u - 1}\| \bfu \|_{r_u,\Omega} \tn V_h \tn_h.
\end{align}
\noindent{\bf Term $II$.}
Applying a Cauchy-Schwarz inequality, followed by the interpolation estimate \eqref{eq:interpest0}
together with the definition of~$\phi_\beta$ yields
\begin{align}
  II &\lesssim
(1+\omega_h)^{\onehalf} \| \phi^{-\onehalf}_{\beta}(\bfpiast \bfu_h - \bfu) \|_{\Omega}
\cdot
(1+\omega_h)^{-\onehalf} \| \phi_{\beta}^{\onehalf} (\bfbeta \cdot \nabla \bfv_h + \nabla q_h)\|_{\Omega}
\\
&\lesssim
(1+\omega_h)^{\onehalf} (\mu + \|\bfbeta \|_{0,\infty,\Omega}h + \sigma h^2)^{\onehalf}
 h^{r_u - 1}\| \bfu \|_{r_u,\Omega} 
\tn V_h \tn_h.
\end{align}
\noindent{\bf Term $III$.}
Similarly it holds
\begin{align}
  III &= -(\piast p - p, \nabla \cdot \bfv_h)_{\Omega}
  + (\piast p - p, \bfv_h\cdot \bfn)_{\Gamma}
\label{eq:apriori_estimate_III_a}
\\
&\lesssim
\bigl(
  \| \phi_u^{-\onehalf} (\piast p - p)\|_{\Omega} 
  + \|h^{\onehalf} \phi_u^{-\onehalf}  (\piast p - p)\|_{\Gamma}
\bigr)
\bigl(
\| \phi_u^{\onehalf}\nabla \cdot \bfv_h\|_{\Omega}
+
    \| (\phi_u/h)^{\onehalf} \bfv_h\cdot \bfn\|_{\Gamma}
\bigr)
\\
&\lesssim
  \| \phi_u^{-\onehalf} (\piast p - p)\|_{\mcT_h}  \tn V_h \tn_h
\\
&\lesssim
\left(
\sum_{T\in \mcT_h}
\dfrac{  h^{2s_p} 
 \| p^{\ast} \|_{s_p, \omega(T)}^2
}
{
\mu + \|\bfbeta\|_{0,\infty,T}h + \sigma h^2
}
\right)^{\onehalf}  \tn V_h \tn_h
\\
&\lesssim
\max_{T\in \mcT_h}
\left\{
\dfrac{1}{\mu + \|\bfbeta\|_{0,\infty,T}h + \sigma h^2}
\right\}^{\onehalf}
h^{s_p} 
 \| p\|_{s_p, \Omega}
\tn V_h \tn_h.
\end{align}
To conclude the proof of the \apriori~error estimate \eqref{eq:apriori-estimate},
we collect the estimates for $I$, $II$ and $III$, keeping in mind that $\tn V_h \tn_h = 1$,
and combine~\eqref{eq:apriori_split_velocity}, \eqref{eq:discrete-error-est-I}, \eqref{eq:consistency_error}
with the interpolation error estimates from Lemma~\ref{lem:interpolation-estimate}.
\end{proof}
\begin{remark}
  We like to point out that the final \apriori estimate derived in this work closely resemble
  the original estimate for the \emph{fitted} CIP method
  presented in \citet{BurmanFernandezHansbo2006}.
  This is on purpose as it demonstrates that the convergence properties of the original
  scheme for the fitted mesh can be carried over to the corresponding \emph{unfitted}
  domain discretization in a geometrically robust way by adding the proper
  ghost-penalty forms to the original formulation.
\end{remark}
\begin{remark}
  \label{rem:oseen:low-order-L2-velocity-optimality}
  Note that similar as shown in the work by
  \citet{BurmanFernandezHansbo2006}, for the low-Reynolds-number case,
  i.e., $\mu \geqslant \|\bfbeta\|_{0,\infty,\Omega} h$, an optimal
  error convergence with respect to the velocity $L^2$-norm $\|
  \bfu^{\ast} - \bfu_h \|_{\Omega} = \mcO(h^{r_u})$ might be derived. A
  proof of this uses the standard Aubin--Nitsche duality technique and
  the deduced energy-norm estimate.
\end{remark}

\section{Conclusions}
\label{sec:conclusions}
In this work, a stabilized cut finite element method for the Oseen
problem has been proposed and analyzed.  The main ingredients of our
formulation can be summarized as follows: Since the computational mesh
is not fitted to the domain, boundary conditions are imposed weakly by
a stabilized Nitsche-type method which accounts for the different flow
regimes.  To sufficiently control the weak formulation in the interior
of the domain for convective-dominant flow and to allow for equal
order interpolation spaces for velocity and pressure, the continuous
interior penalty method is employed. The jump-penalty terms for
velocity and pressure are evaluated at all inter-element faces.  In
the boundary zone of cut meshes, these are extended to the entire cut
faces.  To ensure inf-sup stability and guarantee optimal error bounds
for the different flow regimes, higher-order CIP-like ghost-penalty
terms are added to the formulation including a viscous and a reactive
ghost-penalty operator.  A stability and an \apriori~error analysis
for an energy-type norm is presented and two- and three-dimensional
numerical convergence studies corroborate the theoretical
findings. Optimality is proved for low and high Reynolds numbers and,
in particular, is thereby independent of how the boundary intersects
the computational mesh. As a consequence, the issue of matrix
conditioning is highly improved by the addition of these ghost
penalties.  Furthermore, the applicability of the proposed cut finite
element method for solving the non-linear incompressible Navier-Stokes
equations is verified by simulating a helical pipe flow.

The present work provides an important step in the development of cut
finite element methods for flow problems and, as major aspect,
addresses the need for different stabilization techniques in
convective-dominant flows.  With particular emphasis on the low and
high Reynolds number flow regime, the proposed theoretical analysis
mainly focuses on the numerical treatment of advective term, the
incompressibility constraint and on how to ensure inf-sup stability on
cut meshes.  The theoretical and numerical validations of our discrete
formulation have been proposed for single-phase flows in this work
and, moreover, are of great importance for further developments on
unfitted methods for coupled flow problems like, for instance,
composite-grid techniques, multiphase flows and fluid-structure
interaction.

\section*{Acknowledgements}
This work is supported by the Swedish Foundation for Strategic Research Grant No.\ AM13-0029 and the
Swedish Research Council Grant 2013-4708.
The support of the second author through
the International Graduate School of Science and Engineering (IGSSE)
of the Technical University of Munich, Germany, under project
6.02, is gratefully acknowledged.
This work was also partially supported by a Center of Excellence grant from the
Research Council of Norway to the Center for Biomedical Computing at
Simula Research Laboratory.
The first author wishes to express his gratitude to Prof. Erik Burman for
numerous interesting discussions on continuous interior penalty and cut finite
element methods in the past years.
Finally, the authors wish to thank the anonymous referees for the valuable comments and
suggestions which helped to improve the presentation of this work.

\bibliographystyle{model1-num-names}
\bibliography{bibliography}

\end{document}